\documentclass[12pt,a4paper,reqno, twoside]{amsart}
\usepackage{accents}
\usepackage{esint}
 \usepackage{multirow,bigdelim}
\usepackage{mathtools}
\DeclareUnicodeCharacter{02B9}{'}
\mathtoolsset{showonlyrefs}
\usepackage[english]{babel}
 \usepackage[T1]{fontenc}
\usepackage{csquotes}
\usepackage{epsfig}
\usepackage{graphicx}
\usepackage{bbold}
\usepackage{amsfonts,amsmath,amsthm,amssymb}

\newtheorem{theorem}{Theorem}[section] %
\newtheorem{proposition}[]{Proposition}[section] 

\newtheorem{remark}{Remark}[section]
\usepackage{tikz}
\usepackage[normalem]{ulem}
\usepackage{hyperref}
\hypersetup{
  colorlinks   = true, 
  urlcolor     = MyBlue, 
  linkcolor    = MyBlue, 
  citecolor   = darkspringgreen 
}
\usepackage[margin=1in]{geometry}

\usepackage{enumitem}
\usepackage[giveninits=true,uniquename=false,uniquelist=false,style=alphabetic,sortcites=false,natbib=true,backend=biber,datamodel=mrnumber, maxbibnames=99]{biblatex}
\DeclareFieldFormat{mrnumber}{%
  MR\addcolon\space
  \ifhyperref
    {\href{http://www.ams.org/mathscinet-getitem?mr=#1}{\nolinkurl{#1}}}
    {\nolinkurl{#1}}}

\renewbibmacro*{doi+eprint+url}{%
  \iftoggle{bbx:doi}
    {\printfield{doi}}
    {}%
  \newunit\newblock
  \printfield{mrnumber}%
  \newunit\newblock
  \iftoggle{bbx:eprint}
    {\usebibmacro{eprint}}
    {}%
  \newunit\newblock
  \iftoggle{bbx:url}
    {\usebibmacro{url+urldate}}
    {}}

\numberwithin{equation}{section}
\usepackage{subfigure}
\usetikzlibrary{fit,
  matrix,
  positioning,
  decorations,
  decorations.pathreplacing
}

\usepackage{latexsym}
\usepackage{xparse}
\usepackage{float,soul,ulem,cancel}
\providecommand*\email[1]{\href{mailto:#1}{#1}}

\definecolor{darkgreen}{rgb}{0.0,0.5,0.0}
\definecolor{darkspringgreen}{rgb}{0.05, 0.5, 0.06}
\usepackage[]{youngtab}

\usepackage{xcolor}
\definecolor{MyBlue}{rgb}{0.0,0.0,0.55}
\colorlet{NextBlue}{MyBlue!20}
\colorlet{SecondBlue}{MyBlue!40}
\newcommand{\Rho}{\mathrm{P}}

\NewDocumentCommand{\tens}{t_}
 {%
  \IfBooleanTF{#1}
   {\tensop}
   {\otimes}%
 }
\NewDocumentCommand{\tensop}{m}
 {%
  \mathbin{\mathop{\otimes}\displaylimits_{#1}}%
 }

\newcounter{rhp}
\newenvironment{rhp}[1][]{\refstepcounter{rhp}\par\medskip
   \noindent \textbf{Riemann-Hilbert problem~\therhp. #1} \rmfamily}{\medskip}
   \numberwithin{rhp}{section} 

\usepackage{xparse}
\usepackage{float}
\usepackage{todonotes}
\usepackage{etoolbox}

\usepackage{framed}
\definecolor{shadecolor}{rgb}{0.9, 0.9, 0.86}
\usetikzlibrary{arrows,positioning,cd,calc,math,decorations.pathreplacing,decorations.markings,decorations.pathmorphing,patterns
}
\tikzset{wave/.style={decorate, decoration=snake}}

\usepackage{xcolor}
\definecolor{linkBlue}{rgb}{0.0,0.0,0.55}
\definecolor{MyBlue}{rgb}{0.0,0.0,0.55}
\colorlet{NextBlue}{MyBlue!20}
\colorlet{SecondBlue}{MyBlue!40}
\definecolor{MyGreen}{rgb}{0.25,0.75,0.35}

\newcommand{\tr}{\operatorname{tr}}

\newcommand{\dd}{\mathrm{d}}
\newcommand{\cn}{\mathrm{cn}}
\newcommand{\sn}{\mathrm{sn}}

\newcommand{\dn}{\mathrm{dn}}

\newcommand{\sfa}{{\sf a}}
\newcommand{\sfr}{{\sf r}}
\newcommand{\sfb}{{\sf b}}
\newcommand{\sfc}{{\sf c}}

\newcommand{\ii}{\mathrm{i}}
\newcommand{\ee}{\mathrm{e}}

\renewcommand{\Im}{\mathrm{Im} } 
\renewcommand{\Re}{\mathrm{Re} } 
 
\newcommand{\pain}[1]{Painlev\'e}

\title{Nonlinear steepest descent on a torus: a case study of the Landau-Lifshitz equation}
\author{Harini Desiraju}
\author{Alexander Its}
\author{Andrei Prokhorov$^*$}
\thanks{$^*$Corresponding author}

\address[Harini Desiraju]{Sydney Mathematical Research Institute (SMRI), School of Mathematics and Statistics, University of Sydney, Camperdown, NSW 2006, Australia.}
\email{\href{mailto:harini.desiraju@sydney.edu.au}{harini.desiraju@sydney.edu.au}}
\address[Alexander Its]{Department of Mathematical Sciences, Indiana University, Indianapolis, 402 N. Blackford St. Indianapolis, IN 46202-3267, USA,
and St. Petersburg University, 7/9 Universitetskaya nab. 199034 St. Petersburg, Russia.}
\email{\href{mailto:aits@iupui.edu}{aits@iu.edu}}
\address[Andrei Prokhorov]{Department of Mathematics, University of Michigan, East Hall, 530 Church St., Ann Arbor, MI 48109, USA, and St. Petersburg University, 7/9 Universitetskaya nab. 199034 St. Petersburg, Russia, and Department of Statistics, University of Chicago, 5747 South Ellis Avenue, Chicago, IL 48109, USA}
\email{\href{mailto:andreip@uchicago.edu}{andreip@uchicago.edu}}

\begin{document}

\begin{abstract}
We obtain rigorous large time asymptotics for the Landau-Lifshitz equation in the soliton free case by extending the nonlinear steepest descent method to genus 1 surfaces. The methods presented in this paper pave the way to a rigorous analysis of other integrable equations on the torus and enable asymptotic analysis on different regimes of the Landau-Lifshitz equation. 
\end{abstract}
\maketitle
\tableofcontents
\newpage
\section{Introduction}

In modern theory of integrable systems, the key conceptual role is played by the identification of an integrable nonlinear PDE or ODE as
a compatibility condition of a certain linear system-Lax pair.  Usually, Lax pairs are realized 
as matrix valued first order differential linear operators with respect to
the basic dynamical time-space variables whose
coefficients are rational functions of the additional complex
''spectral'' parameter. For the vast majority of integrable PDEs,  this spectral parameter varies on the
Riemann sphere. For this case, the apparatus for the analysis of integrable systems has been very well developed.  The situation when the spectral parameter belongs to an algebraic curve of a higher genera possess very interesting and serious challenges.
We refer the reader to the 1981 work of I. Krichever and S. Novikov \cite{krichever1981holomorphic} and the 1987 work of N. Hitchin \cite{hitchin1987stable} for a detailed exposition of the specific algebra-geometric features that the integrable systems whose Lax pairs are set on Riemann surfaces do possess. \\
 
The first example of an integrable system whose Lax pair ``lives'' on a Riemann surface
with nontrivial genus is the classical Landau-Lifshitz (LL) equation whose definition we will
soon recall.  The associated Lax pair representation was found in the late 70s by
E. Sklyanin \cite{Sklyanin:121210}  and, independently, by A. Borovik \cite{borovik1981linear}  and it is set on
a torus, i.e. on an elliptic (genus one) curve. Soon after this discovery, several
authors had extended to the Landau-Lifshitz equation such key ingredients of the standard inverse scattering technique
as dressing procedure, finite gap integration, and the Riemann-Hilbert representation.
We refer the reader to the survey paper \cite{bikbaev2014landau} for comprehensive descriptions
of the concrete results obtained in 70-80s (see also the end of this introduction) and for a detailed history of the subject. We want to highlight specifically the work of A. Mikhailov \cite{mikhailov1982landau} where the exact formulation of the
associated Riemann-Hilbert problem on the torus was presented and the
works of Y. Rodin  \cite{rodin1989,rodin_1988,rodin1984}, where  the solvability of this Riemann-Hilbert problem
was thoroughly analyzed. In our work, we shall constantly use the results of \cite{mikhailov1982landau,rodin1989,rodin_1988,rodin1984}.
\\

Since the above mentioned papers written in the 70-80s not much of the further progress in studying
integrable systems on the algebraic curves has been achieved owing to the non-trivial topological environment.
It is only in recent years that we have seen a rigorous generalization of some techniques
such as Fredholm determinants  \cite{del2020isomonodromic}, Pad\'e approximations \cite{bertola2021pade}, orthogonal polynomials \cite{bertola2022nonlinear,desiraju2023class}, and nonlinear waves \cite{Kamvissis_Teschl_2012, Kruger_Teschl_2009, Mikikits-Leitner_Teschl}
among others, to genus 1 surfaces. We should specifically mention M.  Bertola's paper \cite{bertola2022nonlinear}
where a key analytical apparatus of the theory integrable systems -  the nonlinear steepest descent method
of Deift and Zhou, or, rather, its orthogonal polynomial version, has been extended to the genus one curves.
In our paper we are suggesting a similar extension of the PDE version of the method
(i.e., large time analysis of the solutions of Cauchy problems for integrable PDEs).   We take the LL equation as a case study. 
As has already been mentioned, the corresponding Lax pair is set on the algebraic curve of genus 1. We show how one can extend the nonlinear steepest descent method
on this nontrivial genus situation and obtain the large time asymptotic of the solution of the
Cauchy problem in the Schwartz class.\\
 
As in the genus zero case, the two versions (orthogonal polynomials and large time PDE analysis)  of
the nonlinear steepest descent differ in several methodological aspects. Hence, in our analysis
we cannot directly use the results of \cite{bertola2022nonlinear}. In particular, to accommodate a
rather unusual symmetry-normalization of the Riemann-Hilbert problem associated
with LL equation, we have to use, instead of the discussed in \cite{bertola2022nonlinear} matrix Cauchy kernel,
a very special Gusman-Rodin Cauchy kernel (\cite{Gusman}) on the torus which creates certain extra difficulties in the analysis - see section \ref{sec:singular_integral_equation} for more details. It should be noted that the Gusman-Rodin kernel coincides with the kernel considered in \cite{bertola2021pade} (see Remark \ref{rem:cauchy_kernel} for details). This scalar kernel should be compared with the different matrix Cauchy kernel in \cite{ Kamvissis_Teschl_2012, Kruger_Teschl_2009, Mikikits-Leitner_Teschl} also used to construct solutions of the Riemann-Hilbert problem on the Riemann surface that
satisfy the symmetry conditions. We believe that the constructions of
 paper  \cite{bertola2022nonlinear} will be relevant and indeed very useful  for another important
question in the asymptotic theory of PDE which is dispersionless limits. \\

We should also mention that the Riemann-Hilbert problem on Riemann surfaces was
used to construct explicit solutions of model problems in terms of hyperelliptic functions  for the necessities of Deift-Zhou nonlinear steepest descent method;
see \cite{Deift_Its_Zhou,Korotkin,Piorkowski_Teschl,Kamvissis,Kamvissis_Teschl_2006,Egorova_Michor_Teschl}. In a setting similar to ours, when the solution is constructed using the singular integral equation, it
appeared in \cite{Kamvissis_Teschl_2012, Kruger_Teschl_2009, Mikikits-Leitner_Teschl}. It should though to be  pointed out that in our case the
Riemann surfaces of nontrivial genus appear from the very beginning, from the Lax pair,
while in the just mentioned works the nontrivial genus comes from the specific
class of the solutions they consider – initial data on quasiperiodic background.\\

 The large time behavior of the solution of the Cauchy problem for the LL equation in the framework of the Riemann-Hilbert method has already been considered in the paper \cite{bikbaev1988asymptotics}. In that paper, the asymptotics of the solution of  Mikhailov's
  Riemann-Hilbert problem was obtained via the WKB-analysis of the direct scattering problem for the space part of the Sklyanin-Borovik Lax pair. This methodology has been used in the theory of integrable systems since the pioneering work of S.  Manakov and V. Zakharov \cite{zakharov1976asymptotic} and before the introduction of the nonlinear steepest descent method. In \cite{bikbaev1988asymptotics}, some of the  steps of the analysis are not
 fully justified. In principle, they could be made rigorous using the ideas of Kitaev's paper \cite{kitaev1991justification}.
 The nonlinear steepest descent technique we are developing in this paper not only provides a ``short cut'' to the rigorous derivation of the asymptotics of the particular problem associated with the particular system - Landau-Lifshitz equations.
 We also believe that it can be used in the further analysis of the nontrivial genus integrable systems, which have recently attracted a lot of interest.
 Specifically, we are having in mind the already mentioned articles \cite{del2020isomonodromic}, \cite{bertola2021pade}, \cite{bertola2022nonlinear,desiraju2023class}, and \cite{Kamvissis_Teschl_2012, Mikikits-Leitner_Teschl}. \\

The Landau-Lifshitz (LL) equation was introduced by Landau and Lifshitz in 1935 \cite{Landau:437299} to describe the time evolution of magnetism of ferromagnetic material or spin waves and since has become a fundamental tool in the magnetic recording industry \cite{wei2012micromagnetics}. It can be interpreted as the continuum limit of the Heisenberg spin chain \cite{quispel1982equation}. The LL equation in one space and one time dimension (1+1) is the most general model of magnetism that is integrable and has garnered much attention in the literature due to its applications in mathematics and physics. It can be viewed as the universal integrable equation due to the appearance of several integrable equations 
in appropriate limits \cite{faddeev1987hamiltonian,de2022recent,de2020landau}, it has rich physical properties such as the existence of soliton solutions \cite{bobenko1983landau, gamayun2019domain} and the relation to classical mechanics \cite{veselov}.  
A rigorous analysis of its solutions is therefore consequential to a variety of fields.

The LL equation reads
\begin{align}\label{def:eqLL}
      {\partial_t L(x,t)}{} = L(x,t) \times   {\partial_x^2 L(x,t)} + L(x,t) \times J L(x,t); && \sum_{j=1}^3\left(L_j(x,t)\right)^2=1 && 
\end{align}
with 
\begin{align*}
    L(x,t) = \left( L_1(x,t), L_2(x,t), L_3(x,t)\right),&&  J = \textrm{diag} \left(J_1, J_2, J_3 \right), && J_1<J_2<J_3.
\end{align*}
Here, the vector $L(x,t)$ describes the magnetization of the material.
It is immediate to note that for $J_1=J_2=J_3$, the above equation describes a continuous Heisenberg spin chain, 
whose  Lax pair integrability  was established  in 1978 by L. Takhtajan \cite{takhtajan1977integration}, and in matrix form, it relates to the nonlinear Schr\"odinger equation \cite{zakharov1979equivalence}.
The seminal papers by Sklyanin \cite{Sklyanin:121210} and Borovik \cite{borovik1981linear} showed that the LL equation is associated to the following linear system on the torus (the Lax pair)
\begin{align}\label{def:LaxU}
     {\partial _x\Psi(\lambda,x,t)}=U(\lambda,x,t) \Psi(\lambda,x,t), \\ \label{def:LaxV}   {\partial_t \Psi(\lambda,x,t)}=V(\lambda,x,t) \Psi(\lambda,x,t),
\end{align}
where $\lambda\in \mathbb{T}^2=\left\lbrace \lambda: |\Re (\lambda)|\leq 2K, |\Im (\lambda)|\leq 2K'\right\rbrace$ shown in Figure \ref{rectangle}, and $K, K'$ are complete elliptic integrals of moduli $k$, $k' = \sqrt{1-k^2}$ respectively. 
\begin{figure}[h!]
\centering
\includegraphics[trim={0cm 0cm 0cm 0cm},clip, width=6cm]{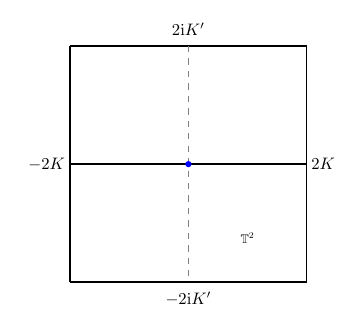}
\caption{Torus.} \label{rectangle}
\end{figure}
It is noteworthy that the linear system above takes values on the torus while the non-trivial topology is not reflected in the LL equation. The matrices 
\begin{align}\label{def:UV}
    U(\lambda,x,t) &:= - \ii \sum_{j=1}^3 \sigma_jL_{j}(x,t)  w_j(\lambda), \\V(\lambda,x,t) &:=  \ii \sum_{\substack{j,m,n=1\\
    j \neq m \neq n}}^3 \sigma_jL_{j}(x,t)  w_m(\lambda) w_n(\lambda) + \ii \sum_{j=1}^3\sigma_j P_j(x,t)  w_j(\lambda),\\
     P_j(x,t) &:=   \left({\partial_x L}(x,t) \times L(x,t)\right)_j\label{eq:momentum}
\end{align}
are written with $\sigma_\alpha$ being the Pauli matrices\footnote{The Pauli matrices are given by
\begin{align*}
    \sigma_1=\left(\begin{array}{cc}
       0  & 1 \\
       1  & 0
    \end{array} \right), && \sigma_2=\left(\begin{array}{cc}
       0  & -\ii \\
       \ii  & 0
    \end{array} \right), && \sigma_3=\left(\begin{array}{cc}
       1  & 0 \\
       0  & -1
    \end{array} \right).
\end{align*}}, and 
\begin{align}\label{def:w1w2w3}
    &&  w_1(\lambda)=\rho\,\frac{1}{\mathrm{sn}(\lambda, k)},&& w_2(\lambda)=\rho\,\frac{\mathrm{dn}(\lambda, k)}{\mathrm{sn}(\lambda, k)},&& w_3(\lambda)=\rho\,\frac{\mathrm{cn}(\lambda, k)}{\mathrm{sn}(\lambda, k)}.
\end{align}
The elliptic curve is given by 
\begin{align}\label{def:elliptic_curve}
    w_i^2(\lambda) - w_j^2(\lambda) =- \frac{1}{4}(J_i-J_j); \qquad i,j=1,2,3,
\end{align}
and the parameter $\rho$ and elliptic modulus $k$ are given respectively by
\begin{align}\label{def:rho}
    \rho = \frac{\sqrt{J_3- J_1}}{2}, &&  k = \sqrt{\frac{J_2-J_1}{J_3-J_1}}.
\end{align}
The equation \eqref{def:eqLL} then arises as the compatibility condition (Lax or zero curvature equations)
\begin{align}\label{eq:comp-cond}
    {\partial_t U}(\lambda,x,t) - {\partial_x V}(\lambda,x,t)+ \left[ U(\lambda,x,t),V(\lambda,x,t)\right]=0.
\end{align}

The Cauchy problem for \eqref{def:eqLL} was recast as a matrix Riemann-Hilbert boundary problem in \cite{mikhailov1982landau, rodin1984} using Jost solutions of the linear equation (\ref{def:LaxU}). In the pure soliton case, this  Riemann-Hilbert problem can be solved by means of linear algebra,
and the corresponding multi soliton solutions were studied by several authors, see \cite{,bikbaev2014landau,bobenko1983landau} and references therein. Finite-gap solutions were constructed in \cite{bobenko1985}. \\

The asymptotic behaviour of soliton free solution of \eqref{def:eqLL} for $t\to +\infty$ was then obtained in \cite{bikbaev1988asymptotics}.
 The basic idea of the approach of \cite{bikbaev1988asymptotics} is the construction of the asymptotic solution of the Riemann - Hilbert  problem from \cite{mikhailov1982landau, rodin1984} using the asymptotic solution of the direct monodromy problem for an auxiliary linear system posed on the same torus. The solution of this direct monodromy problem, in turn, uses certain prior information about the solution of the LL equation.   This approach is a genus one version of the method first suggested in \cite{its1985ismonodromic} for the case of genus zero of the mKdV equation
and, as already mentioned, is not fully rigorous.  In this paper, as we have also mentioned above,  we generalize the nonlinear steepest descent technique to the case of the torus and obtain rigorous large time asymptotics for the LL equation. We are now going to state the main results of this paper. For that, we need to introduce some additional notation.

 The general reference to the scheme we are following in our analysis is the Inverse Scattering Method (IST). This, in particular, means that our goal is to present the asymptotics of the Cauchy problem for the LL equation in terms of the {\it scattering data} associated with the first linear operator of the Lax pair, i.e. the $x$ - equation (\ref{def:LaxU}).  This is performed in two steps.  First, following \cite{Sklyanin:121210},
 we introduce the {\it direct scattering map} from the initial data $L_{t=0} \equiv  L(x)$  to the {\it reflection coefficient} ${\sf r}(\lambda)$ whose properties are described in detail in section \ref{sec:direct}. {\it In this first step, as everywhere else in the paper, we assume that the initial data $L(x)$ satisfies the soliton-free condition, i.e., that the corresponding transition coefficient $\sfa(\lambda)$
 has no zeros. }
  Then, following \cite{mikhailov1982landau,rodin1989,rodin_1988,rodin1984},
we set the {\it inverse scattering map} formulated as a certain Riemann-Hilbert problem on the torus $\mathbb{T}^2$ and defined by the given reflection coefficient $\sfr(\lambda)$. This is the Riemann-Hilbert problem \ref{rhp1} which is described in detail in section \ref{sec:inverse}. The global in $x$ solvability (and uniqueness) of this Riemann-Hilbert problem was proven in \cite{rodin1989}. This constitutes the first step in the construction of the inverse scattering map $\sfr(\lambda) \to L(x)$.
   In our first main result, we complete the construction of this map by analyzing the large $|x|$ asymptotics of the solution of the Riemann-Hilbert problem \ref{rhp1}. This analysis allows us
to deduce the belonging of the ``potentials" $L_1(x), L_2(x), L_3(x)-1$ to the Schwartz class and to establish that the corresponding reflection coefficient coincides with the given function ${\sf r}(\lambda)$.
 
\begin{theorem}\label{thm:initial_condition_asymptotic}
Let $Y(\lambda, x)$ be the solution of the Riemann-Hilbert problem \ref{rhp1} with ${\sf r}(\lambda)$ satisfying 
properties \eqref{r_property_1}-\eqref{r_property_6} of section \ref{sec:direct}. Then the function $L(x)$ constructed from it by formula
\begin{equation}\label{sol_LL_RHP_init_0}
   Y(0,x) \sigma_3 \left(Y(0, x)\right)^{-1} = \sum_{j=1}^3 L_j(x) \sigma_j ,
\end{equation}
(see also \eqref{sol_LL_RHP_init}) belongs to the Schwartz class: $L_1(x), L_2(x), L_3(x)-1 \in \mathcal{S}(\mathbb{R})$, and it defines the
initial data for the LL equation whose reflection coefficient is given by ${\sf r}(\lambda)$.
\end{theorem}
Next, we are  `` turning on'' the time which means passing from the Riemann-Hilbert problem \ref{rhp1} to the Riemann-Hilbert problem \ref{rhp2} of section \ref{sec:time}. 
The solvability of this problem for all  real $x$ and $t>0$ is shown in \cite{rodin1989} and  formula 
\begin{align}\label{sol_LL_RHP_init_1}
   Y(0,x,t) \sigma_3 \left(Y(0, x, t)\right)^{-1} = \sum_{j=1}^3 L_j(x,t) \sigma_j ,
\end{align}
where $Y(\lambda,x,t)$ is the solution of the Riemann-Hilbert problem \ref{rhp2},  determines the solution $L(x,t)$ of the Cauchy problem for the LL equation whose initial data $L(0, x) \equiv L(x)$ corresponds to the reflection
coefficient ${\sf r}(\lambda)$. Our second main result is the extension of the nonlinear steepest descent
method of Deift and Zhou to the Riemann-Hilbert problem \ref{rhp2} and evaluation of the leading term of the large time asymptotics
of the corresponding solution $L(x,t)$ of the LL-equation.
\begin{theorem}\label{thm:asymp_LL}
  Let $Y(\lambda,x,t)$ be the solution of the Riemann-Hilbert problem \ref{rhp2} with ${\sf r}(\lambda)$ satisfying the properties \eqref{r_property_1}-\eqref{r_property_6} 
  of section \ref{sec:direct} and
let the vector function $L(x,t)$ be determined by equation \eqref{sol_LL_RHP_init_1}. Then, $L(x,t)$ solves the Cauchy problem for the LL equation characterized by the reflection coefficient ${\sf r}(\lambda)$ and its large time asymptotic behavior is given by the formulae \begin{align}
        L_1(x,t) &= \frac{1}{\rho}\left(\frac{2\nu}{t\varphi_0} \right)^{1/2} { w_2(\lambda_0) \cos\theta(x,t)}{} + \mathcal{O}(t^{-\frac{2}{3}}), \\
        L_2(x,t) &= \frac{1}{\rho}\left(\frac{2\nu}{t\varphi_0} \right)^{1/2} { w_1(\lambda_0) \sin\theta(x,t)}{}+ \mathcal{O}(t^{-\frac{2}{3}}), \\
        L_3(x,t) &= 1 - \frac{1}{2} \left(L_1^2(x,t) + L_2^2(x,t) \right)+\mathcal{O}(t^{-\frac{7}{6}}),
    \end{align}
    $$
    t\to \infty, \quad \frac{x}{t}  = \varkappa, \quad 0< m \leq\varkappa \leq M,
    $$
    where
    \begin{align}
         \theta(x,t) &=  2tp(\lambda_0,\varkappa)+ \nu \log t-\frac{\pi}{4}- \arg\Gamma(\ii\nu) + \arg \sfr_0 - 2\sfc_0 + \nu \log\left(\frac{2\varphi_0 }{\beta_0^2} \right),
    \end{align}
     \begin{align}
         \eqref{def:plambda}:\,\, & p(\lambda,\varkappa)= \varkappa\, w_3(\lambda)-2w_1(\lambda)w_2(\lambda),
     \end{align}
     and the value of the stationary point $\lambda_0\in [-2K,0]$ is determined by the equation \\\eqref{eq:pprimelambda}: $\partial_\lambda p(\lambda_0,\varkappa)=0$. With such $\lambda_0$, the parameter $\varphi_0=-\partial^2_\lambda p(\lambda_0,\varkappa)$ is obtained from  \begin{align*}
         \eqref{def:phi_0}:\,\, &  \varphi_0 =\frac{1}{\rho^2}\left( 8 w_1(\lambda_0) w_2(\lambda_0) w_3^2(\lambda_0)+(w_1^2(\lambda_0)+w_2^2(\lambda_0))(2w_1(\lambda_0)w_2(\lambda_0)-\varkappa w_3(\lambda_0)) \right),
     \end{align*} the reflection coefficient 
     \begin{align*}
         \sfr_0\equiv \sfr(\lambda_0),  &&\eqref{def:nu}:\,\, & \nu=\frac{1}{2\pi }\log(1+|{\sf r}_0|^2).
     \end{align*} The remaining terms are determined as follows:
     \begin{align*}
          \eqref{def:beta}:\,\, & \beta(\lambda) = \frac{\sigma(\lambda) \sigma(\lambda-2K)}{\sigma(\lambda +2\ii K') \sigma(\lambda-2\ii K'-2K)}, &&  \eqref{def:beta0}: &  \beta_0 := \frac{\sigma(-2K)}{\sigma(2\ii K') \sigma(-2\ii K'-2K)}, 
     \end{align*}
     \begin{align*}
         \eqref{def:c_0}:\,\, &  \sfc_0 =\frac{1}{2\pi  } \int^{0}_{\lambda_0} d\left(\log \left( 1+ |{\sf r}(\eta)|^2 \right) \right)  \log\beta(\eta- \lambda_0),
     \end{align*}
     where $\sigma(\lambda)$ denotes the Weierstrass sigma function.
\end{theorem}

 This theorem describes the asymptotic behavior of the solution of the Cauchy problem for the Landau-Lifshitz equation with the initial data of the Schwartz class under the assumption of the absence of solitons.  
As has already been mentioned, the asymptotic formulae presented in this theorem have already been
obtained in \cite{bikbaev1988asymptotics} using a different approach.

The paper is organized as follows. In section \ref{sec:direct}, we recap the formulation of the direct scattering transform for
linear system (\ref{def:LaxU}) through Jost solutions. In section \ref{sec:inverse} we define the Riemann-Hilbert problem associated
with the inverse scattering transform for (\ref{def:LaxU}), and adapt the singular integral equation for our situation. Using it we provide the proof of Theorem \ref{thm:initial_condition_asymptotic} and guarantee that corresponding initial condition is from the Schwartz class. In section \ref{sec:time}
we formulate the time-dependent Riemann-Hilbert problem associated with the solution of
the Cauchy problem for the LL-equation.  In section \ref{sec:NSD}, we perform the nonlinear steepest descent analysis and note that the local parametrices are described by parabolic cylinder functions. In section \ref{Sec:asymp}, we obtain the large time asymptotic behavior of $L(x,t)$ and prove Theorem \ref{thm:asymp_LL}. In the appendices \ref{sec:jacobi-peoperties}, \ref{sec:Weierstrass_zeta}, \ref{sec:parabolic} we collect the facts about the elliptic functions and the parabolic cylinder functions. In appendix \ref{prop_proof} we provided the proof of more technical propositions related to the direct scattering problem.

\subsection*{Acknowledgements} Part of the work was carried out during the authors' residence at the Mathematical Sciences Research Institute in Berkeley, California, during the Fall 2021 semester. We acknowledge the support of National Science Foundation Grant No. DMS-1928930. H.D is partly supported by ARC grant DP-200100210, INI-Simons fellowship, and SMRI postdoctoral fellowship. A.I is partially supported by NSF grant DMS:1955265, by RSF grant No. 22-11-00070, and by a Visiting Wolfson Research Fellowship from the Royal Society. AP was supported by NSF MSPRF grant DMS-2103354 and RSF grant 22-11-00070.

\section{Direct scattering problem}\label{sec:direct}
We begin this section with a brief recap of the construction of the Riemann-Hilbert  boundary problem using Jost solutions for time $t=0$ in \cite{mikhailov1982landau, rodin1984,Sklyanin:121210}, and a study of its properties. We will then study the properties of the reflection coefficient for non-zero time. It will then provide the motivation to construct the Riemann-Hilbert problem for $t> 0$ in section \ref{sec:time}.

A starting point in this analysis is the assumption that the initial data for the Cauchy problem satisfies conditions $L_1(x), L_2(x), L_3(x)-1 \in \mathcal{S}(\mathbb{R})$, where $\mathcal{S}(\mathbb{R})$ is the Schwartz class. 
In particular
\begin{align}\label{eq:L_decay}
    \lim_{x\to \pm \infty} L(x) = (0,0,1).
\end{align}
Denote \begin{align}
\Gamma_1=\{\lambda\in \mathbb{T}^2:\mathrm{Im} (\lambda) =0\}&&\Gamma_2=\{\lambda\in \mathbb{T}^2:\mathrm{Im} (\lambda) =2K'\}.\end{align} We will restrict ourselves to $\lambda\in\Gamma_1\cup\Gamma_2$ for now. Due to the condition \eqref{eq:L_decay} in \eqref{def:LaxU} we can define the following Jost solutions:
\begin{align}\label{asymp_jost}
    F_{\pm}( \lambda,x)  = \ee^{- \ii x w_3(\lambda)  \sigma_3}+ o(1), &&\mbox{at}&&  x\to \pm \infty &&\mbox{for}&& \lambda\in\Gamma_1\cup\Gamma_2.
\end{align}
The determinant of the above expression is given by
\begin{align}
\tr(U(\lambda,x))=0&&\Rightarrow&&\det(F_{\pm}(\lambda,x))=1, \label{eq:det-symmetry}
\end{align}
and symmetry properties of Jost solutions follow from the relations \eqref{conjugation-w} of functions $w_1(\lambda)$, $w_2(\lambda)$, $w_3(\lambda)$:
\begin{align}
\overline{U(\lambda,x)}=\sigma_2U(\overline{\lambda},x)\sigma_2&&\Rightarrow&&\overline{F_{\pm}(\lambda,x)}=\sigma_2F_{\pm}(\overline{\lambda},x)\sigma_2. \label{eq:conjugation-symmetry}
\end{align}
Since Jost solutions both solve the differential equation \eqref{def:LaxU}, they are related through the scattering matrix
\begin{align}\label{jump_jost}
    F_{+}(\lambda,x) = F_{-}(\lambda,x) S(\lambda)&&\lambda\in \Gamma_1\cup\Gamma_2.
    \end{align}
\noindent Furthermore, the identities \eqref{eq:det-symmetry}, \eqref{eq:conjugation-symmetry} imply the following structure of the scattering matrix
\begin{align}\label{jump_jost_scattering}
S(\lambda)=\left(\begin{array}{cc}
       {\sf a}(\lambda)  & -\overline{{\sf b}({\lambda})} \\
       {\sf b}(\lambda)  & \overline{{\sf a}({\lambda})}
    \end{array} \right), && |{\sf a}(\lambda)|^2 + |{\sf b}(\lambda)|^2=1.
\end{align}
Here we used the fact that Jost solutions are periodic with respect to the shift by $4K$, $4\ii K'$ and that $\overline{\lambda}=\lambda-4\ii K'$ for $\lambda\in \Gamma_2$.

For further analysis of the properties of the Jost solutions, we 
define the following functions
\begin{align}
    \Upsilon_{\pm}(\lambda,x) &= F_{\pm}(\lambda,x) 
 \ee^{\ii x w_3(\lambda)  \sigma_3} =: \left(
    {\upsilon_{\pm}^{(1)}(\lambda,x)},
    \upsilon_{\pm}^{(2)}(\lambda,x)\right).
\end{align}
Moreover, since $F_{\pm}(\lambda,x)$ solves \eqref{def:LaxU}, $\Upsilon_{\pm}(\lambda,x)$ solves the ODE
\begin{align}\label{eq:Ydeq}
    {\partial_x \Upsilon_{\pm}}(\lambda,x) =  U(\lambda,x) \Upsilon_{\pm}(\lambda,x) +\ii w_3(\lambda) \Upsilon_{\pm}(\lambda,x) \sigma_3,
\end{align}
and the asymptotic conditions \eqref{asymp_jost} imply that
\begin{align}\label{eq:yp_asym}
&&\upsilon_\pm^{(1)}(\lambda,x)\simeq \left( \begin{array}{c}
       1 \\
         0 
    \end{array}\right),\quad x\to \pm\infty&& \upsilon_\pm^{(2)}(\lambda,x)\simeq \left( \begin{array}{c}
        0  \\
        1 
    \end{array}\right),\quad x\to \pm\infty.
\end{align}
Using observations above, we can obtain the alternative characterization of $\Upsilon_{\pm}(\lambda,x)$.

\begin{proposition}[\cite{rodin1984}]
The solutions $\widehat{\upsilon}_{\pm}^{(j)}(\lambda,x)$  of the integral equations 
\begin{align}
    \widehat{\upsilon}_{\pm}^{(1)}(\lambda,x) &= \left( \begin{array}{c}
        1 \\
         0 
    \end{array}\right) +  \int^{x}_{\pm \infty} \ee^{\ii (\mathbb{1}-\sigma_3) w_3(\lambda) (x-\tau)}\left(U(\lambda,\tau) +\ii w_3(\lambda) \sigma_3 \right) \widehat{\upsilon}_{\pm}^{(1)}(\lambda,\tau) d\tau,\label{eq:int_ypm1}\\
     \widehat{\upsilon}_{\pm}^{(2)}(\lambda,x) &= \left( \begin{array}{c}
         0  \\
       1
    \end{array}\right) + \int^{x}_{\pm \infty} \ee^{-\ii (\mathbb{1}+\sigma_3) w_3(\lambda) (x-\tau)} \left(U(\lambda,\tau) +\ii w_3(\lambda) \sigma_3 \right) \widehat{\upsilon}_{\pm}^{(2)}(\lambda,\tau) d\tau,\label{eq:int_ypm2}
\end{align}
solve the differential equation \eqref{eq:Ydeq} and satisfies asymptotic conditions
\begin{align}\label{eq:hat_yp_asym}
&&\widehat{\upsilon}_\pm^{(1)}(\lambda,x)\simeq \left( \begin{array}{c}
       1 \\
         0 
    \end{array}\right),\quad x\to \pm\infty&& \widehat{\upsilon}_\pm^{(2)}(\lambda,x)\simeq \left( \begin{array}{c}
        0  \\
        1 
    \end{array}\right),\quad x\to \pm\infty.
\end{align}
\end{proposition}
\begin{proof}
 We compute the derivative
 \begin{align*}
     \partial_x   \widehat{\upsilon}_{\pm}^{(1)}(\lambda,x)&=\left(U(\lambda,\tau) +\ii w_3(\lambda) \sigma_3 \right) \widehat{\upsilon}_{\pm}^{(1)}(\lambda,\tau)\\ &+ \ii (\mathbb{1}-\sigma_3) w_3(\lambda)\int^{x}_{\pm \infty} \ee^{\ii (\mathbb{1}-\sigma_3) w_3(\lambda) (x-\tau)}\left(U(\lambda,\tau) +\ii w_3(\lambda) \sigma_3 \right) \widehat{\upsilon}_{\pm}^{(1)}(\lambda,\tau) d\tau.
 \end{align*}
Replacing the integral term using 
\eqref{eq:int_ypm1} we get
\begin{align*}
    \partial_x   \widehat{\upsilon}_{\pm}^{(1)}(\lambda,x)&=\left(U(\lambda,\tau) +\ii w_3(\lambda) \sigma_3 \right) \widehat{\upsilon}_{\pm}^{(1)}(\lambda,\tau)+ \ii (\mathbb{1}-\sigma_3) w_3(\lambda)\left( \widehat{\upsilon}_{\pm}^{(1)}(\lambda,x) - \left( \begin{array}{c}
        1 \\
         0 
\end{array}\right)\right)\\
&=\left(U(\lambda,\tau) +\ii w_3(\lambda) \right) \widehat{\upsilon}_{\pm}^{(1)}(\lambda,\tau).
\end{align*}
Similarly we note that
\begin{align*}
    \partial_x   \widehat{\upsilon}_{\pm}^{(2)}(\lambda,x)=\left(U(\lambda,\tau) -\ii w_3(\lambda) \right) \widehat{\upsilon}_{\pm}^{(2)}(\lambda,\tau)
\end{align*}
which gives \eqref{eq:Ydeq}. The check of asymptotic conditions \eqref{eq:hat_yp_asym} is trivial.
\end{proof}

As a result, we see that functions ${\upsilon}_{\pm}^{(j)}(\lambda,x)$ satisfy not only differential equation \eqref{eq:Ydeq}, but also the integral equations \eqref{eq:int_ypm1}, \eqref{eq:int_ypm2}. Let us now denote
\begin{align}
    \Omega_{+} =  \left\lbrace \lambda: 0 \leq \Im (\lambda )\leq 2K'; |\Re \lambda| \leq 2K \right\rbrace,&&\Omega_{-} = \left\lbrace \lambda: -2K' \leq \Im (\lambda) \leq 0; |\Re \lambda| \leq 2K \right\rbrace.
\end{align}
In the next step, we are willing to extend the function $\Upsilon_{\pm}(\lambda,x)$ in the domains $\Omega_{\pm}$. Using properties \eqref{w3-complexplane}, \eqref{eq:special values}, and Figure \ref{fig:foobar}, we note that the elliptic function $\Im(w_3(\lambda))$ satisfies the following inequalities
\begin{equation} \label{eq:theta-lambda-map}
   -\infty<\Im (w_3(\lambda))\leq 0 \quad\mbox{for}\quad\lambda \in\Omega_{+}; \qquad
   0\leq \Im (w_3(\lambda))< \infty \quad\mbox{for}\quad \lambda\in\Omega_{-}.
\end{equation}
Consequently, we get the following result. Notice that columns of $\Upsilon_{\pm}(\lambda,x)$ are analytic in opposite domains.
\begin{proposition}[\cite{rodin1984}]\label{Prop:Rodin}
The functions $\upsilon_{\pm}^{(1)}(\lambda,x)$, $\upsilon_{\mp}^{(2)}(\lambda,x)$ are analytic in the domains $\Omega_{\pm}$  and bounded in the domains $\overline{\Omega}_{\pm}$ respectively. In addition \\$\upsilon_{\pm}^{(1)}(\lambda,x)$, $\upsilon_{\pm}^{(2)}(\lambda,x)$ $\in C^{\infty}(\Gamma_1)$.
\end{proposition}
\noindent For reader's convenience we  present the proof of this Proposition in appendix \ref{prop_proof}. 

 The integral equations \eqref{eq:int_ypm1}, \eqref{eq:int_ypm2} further imply the following result.
\begin{proposition}\label{prop:int_ab}
The functions $\sf{a}(\lambda)$, $\sf{b}(\lambda)$ admit two alternative expressions.
\begin{enumerate}
\item They can be written as the integrals
\begin{align}\label{eq:b_integral}
\left(\begin{array}{c}
    {{\sf a}({\lambda})}   \\
      {\sf b}(\lambda)
    \end{array} \right)= \left( \begin{array}{c}
         1  \\
        0
    \end{array}\right) +  \int_{-\infty}^{ \infty} \ee^{-\ii (\mathbb{1}-\sigma_3) w_3(\lambda) \tau}  \left(U(\lambda,\tau) +\ii w_3(\lambda) \sigma_3 \right)  \upsilon_{-}^{(1)}(\lambda,\tau) d\tau.
\end{align}
\item They can also be expressed as the following determinants 
\begin{align}
     {\sf a(\lambda)}=\det(\upsilon_{+}^{(1)}(\lambda,x),\upsilon_{-}^{(2)}(\lambda,x)),\label{eq:a_det_form}\\
     {\sf b(\lambda)}= \ee^{2\ii w_3(\lambda) x} \det(\upsilon_{-}^{(1)}(\lambda,x),\upsilon_{+}^{(1)}(\lambda,x)).\label{eq:b_det_form}
\end{align}
\end{enumerate}
\end{proposition}
 We refer the reader to appendix \ref{prop_proof} for the proof.

With the above expressions we obtain the following result for the analyticity of the coefficients.
\begin{proposition}\label{prop:a_b_smooth} The coefficient 
${\sf b}(\lambda)\in C^{\infty}(\Gamma_1\cup \Gamma_2)$, and $\sfa(\lambda)$ is analytic in $\Omega_+$.  Moreover, $\frac{d^n{\sfb}(\lambda)}{d\lambda^n} = \mathcal{O} (\lambda^m), \quad \lambda \to 0, \quad \forall n, m \in {\mathbb{N}}.$
\end{proposition}
\noindent The fact that ${\sf b}(0)=0$ can be found in the literature, see \cite[\text{(3.14)-(3.17)}]{Sklyanin:121210} and \cite[\text{(3.2)-(3.3)}]{rodin1984}.  We refer the reader to appendix \ref{prop_proof} for the proof.

Besides properties \eqref{eq:det-symmetry}, \eqref{eq:conjugation-symmetry} we can also observe the following additional symmetry properties of the Jost functions
\begin{proposition}
    The function $F_\pm(\lambda,x)$ has the following periodicity properties:
    \begin{align}\label{symmetry1}
 \sigma_{3} F_\pm(\lambda + 2K,x) \sigma_{3} = F_\pm(\lambda,x), &&
 \sigma_{1} F_\pm(\lambda + 2\ii K',x) \sigma_{1} = F_\mp(\lambda,x). &&
\end{align}
\end{proposition}
\begin{proof}
    The relations \eqref{symm:w1}--\eqref{symm:w3}  imply that
    \begin{align}\label{symm:U}
       \sigma_3 U(\lambda+2K,x) \sigma_3=  U(\lambda,x), && \sigma_{1} U(\lambda + 2\ii K',x) \sigma_{1} = U(\lambda,x). &&
    \end{align}
    and formulae \eqref{symmetry1} follow. 
  \end{proof}
 Furthermore, equations \eqref{symmetry1} lead to the following    symmetry relations of the scattering data,
 \begin{align}
\label{eq:a_symm}\sfa(\lambda+2K)=\sfa(\lambda),&&\sfb(\lambda+2K)=-\sfb(\lambda),\\
    \sfa(\lambda+2\ii K')=\overline{\sfa(\bar{\lambda})},&& \sfb(\lambda+2\ii K')=-\overline{\sfb(\bar{\lambda})}.\label{eq:b_symm}
\end{align}

 We conclude this section by introducing   the {\it reflection} coefficient ${\sf r}(\lambda)$ as the ratio of ${\sf a}(\lambda), {\sf b}(\lambda)$:
\begin{align*}
{\sf r}(\lambda)=\frac{{\sf b}(\lambda)}{{\sf a}(\lambda)}.
\end{align*}
The above established properties of the scattering data ${\sf a}(\lambda)$ and ${\sf b}(\lambda)$,
(under the soliton free assumption, ${\sf a}(\lambda) \neq 0$ !) imply the following list of the properties
of ${\sf r}(\lambda)$.
\begin{enumerate}
\item ${\sf r}(\lambda) \in C^{\infty}(\Gamma_1\cup \Gamma_2)$\label{r_property_1}
\item $\sfr(\lambda +2K) = -\sfr(\lambda)$\label{r_property_3}
 \item $\sfr(\lambda +2\ii K') = -\overline{{\sfr}(\bar{\lambda})}$\label{r_property_4}
 \item ${\sf r}(0) =0$
\item $\frac{d^n{\sfr}(\lambda)}{d\lambda^n} = \mathcal{O} (\lambda^m), \quad \lambda \to 0, \quad \forall n, m \in { \mathbb{N}}.$\label{r_property_6}
\end{enumerate}
\begin{remark} The last property follows from the fact that near $\lambda =0$ the Lax pair  we are studying becomes effectively the Lax pair of the isotropic version of the LL equation which
is equivalent to the NLS equation and the behavior of the reflection coefficients $\sfr (\lambda)$ at $\lambda =0$
mimics the behavior of the NLS reflection coefficient $\sfr(\lambda)$ at $\lambda = \infty$ (see \cite{takhtajan1977integration,zakharov1976asymptotic,faddeev1987hamiltonian, rodin1984}
for more details).{We should also notice that these properties hold near $\lambda=2K, 2\ii K', 2K+2\ii K'$ due to shift properties \eqref{r_property_3}, \eqref{r_property_4}}.
\end{remark}

As it is usual in the scattering theory,  under our {\it standing assumption} that ${\sf a}(\lambda)$ has no zeros in $\Omega_+$ we have that the reflection coefficient ${\sf r}(\lambda)$ determine all the scattering data by the genus 1 version of the classical  dispersion relation (see \cite{Sklyanin:121210}),
\begin{equation}\label{ar}
 \sfa(\lambda) := \exp\left\lbrace -\frac{1}{2\pi \ii}\int_{-2K}^{0} \log \left( 1+ |{\sf r}(\eta)|^2 \right) \frac{w_3(\eta-\lambda)}{\rho} d\eta \right\rbrace,\quad \lambda\in\Omega_+.
\end{equation}

\section{Inverse  scattering problem }\label{sec:inverse}

With the definitions
\begin{align}
       &Y_{+}(\lambda,x) :=\left(
    \frac{\upsilon_{+}^{(1)}(\lambda,x)}{{\sf a}(\lambda)},
    \upsilon_{-}^{(2)}(\lambda,x)\right),\quad Y_{-}(\lambda,x) :=\left(
    {\upsilon_{-}^{(1)}(\lambda,x)}{},
    \frac{\upsilon_{+}^{(2)}(\lambda,x)}{\overline{{\sf a}(\bar{\lambda})}}\right),
    \end{align}
the functions $Y_{+} (\lambda,x)$ and $Y_{-} (\lambda,x)$ are analytic in $\Omega_+$ and $\Omega_-$,
respectively. Also, using relation \eqref{jump_jost} we can write $Y_\pm(\lambda,x)$ in terms of $\Upsilon_+(\lambda,x)$ as
    \begin{align}
    Y_{+}(\lambda,x) & = \Upsilon_{+}(\lambda,x) \left(\begin{array}{cc}
        \frac{1}{\sf a(\lambda)} & \overline{\sf b(\lambda)} \ee^{-2 \ii x w_3(\lambda)}  \\
        0 & {\sf a(\lambda)}
    \end{array} \right) \label{JostPsi1}\\
    Y_{-}(\lambda,x) & = \Upsilon_{+}(\lambda,x) \left(\begin{array}{cc}
        \overline{\sf a(\lambda)} & 0 \\
        -{\sf b(\lambda)} \ee^{2\ii x w_3(\lambda)} & \frac{1}{\sf a(\lambda)}
    \end{array} \right)\label{JostPsi2}
\end{align}
The symmetry properties of $F_\pm(\lambda,x)$, $\sfa(\lambda)$ and $\sfb(\lambda)$ then imply that
the function $Y_\pm(\lambda,x)$ has the following periodicity properties:
\begin{align}\label{symmetry}
    \begin{split}
  \sigma_{3} Y_\pm(\lambda + 2K,x) \sigma_{3} = Y_\pm&(\lambda,x), \qquad
 \sigma_{1} Y_\pm(\lambda + 2\ii K',x) \sigma_{1} = Y_\mp(\lambda,x).
    \end{split}
\end{align}

We are ready now to formulate  the Riemann-Hilbert problem (RHP) associated with the inverse scattering
problem for the linear equation (\ref{def:LaxU}). Define the piecewise analytic function 
\begin{align*}
Y(\lambda,x)=\begin{cases}
    Y_+(\lambda,x),\quad \lambda\in\Omega_+,\\
    Y_-(\lambda,x),\quad \lambda \in \Omega_-.
\end{cases}
\end{align*}

Then,   under the assumption $\sfa(\lambda)\neq 0$ for $\lambda\in \overline{\Omega_+}$, we obtain that the piecewise 
analytic function $Y(\lambda,x)$ is solving the following Riemann-Hilbert problem on the torus.
\begin{rhp}\label{rhp1}
\begin{enumerate}
    \item The function $Y(\lambda,x)$
    is bounded, doubly periodic, and piecewise analytic for $\lambda\in \mathbb{T}^2/(\Gamma_{1} \cup \Gamma_{2})$.
    Orientation of the contours $\Gamma_{1}, \Gamma_{2}$ is as specified in Figure \ref{Contour_1}.
\begin{figure}[H]
\centering
\includegraphics[trim={0cm 0cm 0cm 0cm},clip, width=7cm]{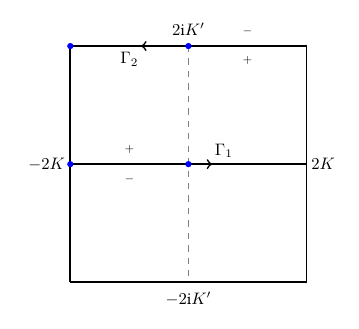}
\caption{Contours $\Gamma_1$ and $\Gamma_2$.} \label{Contour_1}
\end{figure}
    \item For $\lambda \in \Gamma_{1}, \Gamma_{2}$, the following jump condition holds 
    \begin{align}\label{rhp1:jump}
    Y_{+} (\lambda,x) = 
    Y_{-}(\lambda,x) G(\lambda,x), && G(\lambda,x)=\begin{pmatrix}
 1+|{\sf r}(\lambda)|^2 & \overline{{\sf r}  (\lambda)} \ee^{-2\ii x w_3(\lambda)} \\
         {\sf r} (\lambda)\ee^{2\ii x w_3(\lambda)} & 1 \end{pmatrix}. 
\end{align}
\item The function $Y(\lambda,x)$ satisfies the following symmetry conditions
  \begin{align}\label{symm_Y}
 \sigma_{3} Y(\lambda + 2K,x) \sigma_{3} = Y(\lambda,x), &&
 \sigma_{1} Y(\lambda + 2\ii K',x) \sigma_{1} = Y(\lambda,x).
\end{align}
\item Function $Y(\lambda,x)$ satisfies normalization condition
$\det(Y(\lambda,x))=1$.
\end{enumerate}
\end{rhp}
The inverse scattering problem for  the linear equation (\ref{def:LaxU}) consists in
the  analysis of the RHP (\ref {rhp1}). The latter includes, given the function $\sfr(\lambda)$
satisfying the properties \eqref{r_property_1}-\eqref{r_property_6}, to show that:\begin{enumerate}[label=(\alph*)]
    \item 
the RHP (\ref {rhp1}) is uniquely solvable for all real $x$, \item 
its solution is differentiable and yields the linear system (\ref{def:LaxU}) for the Schwartz class potentials  $L_1(x)$, $L_2(x)$, $L_3(x)-1$,
\item the reflection coefficient corresponding to these potentials coincides with the given  function $\sfr(\lambda)$.
\end{enumerate} 
We shall start with a simplest part -- uniqueness  of the solution of the problem   (\ref {rhp1}).
\begin{proposition}\label{prop:uniq_init}
Solution $Y(\lambda,x)$ of the RHP \ref{rhp1} is unique up to a sign.
\end{proposition}
\begin{proof}

Let us assume that there are two solutions of RHP \ref{rhp1} $Y(\lambda,x)$ and 
$\widetilde{Y}(\lambda,x)$. Then the function $$\widetilde{Y}(\lambda,x)\left(Y(\lambda,x)\right)^{-1}$$ is bounded elliptic matrix function, therefore constant. The symmetry conditions \eqref{symm_Y} further imply that it is scalar. The determinant condition implies that it can only be $\pm 1$.
\end{proof}

The nontrivial  question of solvability of the RHP  (\ref{rhp1}) has been positively resolved in  \cite{rodin1989}.
Therefore we only need to concentrate on the properties of the solution $Y(\lambda,x)$ as a function of $x$
and establish the validity of Theorem \ref{thm:initial_condition_asymptotic}.
To this end we need first to introduce the relevant apparatus of singular integral equations  associated with
our Riemann-Hilbert problem. 
\subsection{Singular integral equation for Riemann-Hilbert problem on the torus}\label{sec:singular_integral_equation}
To study the solution of the Riemann-Hilbert problems on the torus, we will need the singular integral equation. The first step to writing the singular integral equation is constructing an appropriate Cauchy kernel. In general, the Cauchy kernel $C(\mu, \lambda)\dd\mu$ is a meromorphic function in $\lambda$ and a one form in $\mu$, with residues $\pm 1$ at the poles of $\mu$, {\it i.e} it is an Abelian differential of the third kind. On the torus, such a function is not uniquely defined. For the purposes of this paper, we consider the Gusman-Rodin kernel \cite{Gusman, rodin1984} 
\begin{align}\label{def:Cauchy}
     C(\mu, \lambda) := \zeta(\mu- \lambda) - \zeta(\mu- \ii K') +\zeta(\lambda- K- \ii K')+ \zeta(K), 
\end{align}
where $\zeta(.)$ is the Weierstrass $\zeta$-function. It is important to note however, that the asymptotic result is of course independent of the choice of the kernel.
The Cauchy kernel above has the following properties.
\begin{itemize}
    \item The periodicity properties of the function $\zeta(.)$ \eqref{per:zeta1}--\eqref{per:zeta3} imply that
    \begin{align}\label{per:Cauchy}
        C(\mu+4K, \lambda)= C(\mu, \lambda+4K) = C(\mu+4\ii K',\lambda) =C(\mu,\lambda+4\ii K') = C(\mu,\lambda).
    \end{align}
     \item It has poles at $\mu=\lambda$,  $\mu=\ii K'$, $\lambda=K+\ii K'$ with residues $1$,$-1$, and $1$ respectively.
          \item It has zeros for $\lambda=\ii K'$, $\mu=K+\ii K'$.
\end{itemize}
\begin{remark}\label{rem:cauchy_kernel}
    We could compare the Cauchy kernel \eqref{def:Cauchy} with the kernel that appeared in \cite[(2.2)]{bertola2021pade}. They match after identification $z=\lambda-\ii K’, w=\mu-\ii K’, a=K$. The shift is justified by the fact that we would like to separate the singularities of the Cauchy kernel and the singularities of function $w_3(\lambda)$ which lie in $\lambda=0,2K,2\ii K',2K+2\ii K'$ and appear in the jump matrix $G(\lambda,x)$ of RHP \ref{rhp1}. Other than that our choice is arbitrary and follows \cite{rodin1984}.
\end{remark}

\noindent Consider the following analog of RHP \ref{rhp1} with normalization at $\lambda=\ii K'$.
\begin{rhp}\label{rhp:generic1}
\begin{enumerate}
    \item The function $\Phi(\lambda,x)$
    is bounded and piecewise analytic for $\lambda\in \mathbb{T}^2/(\Gamma_1\cup \Gamma_2)$.
    \item For $\lambda \in \Gamma_1\cup\Gamma_2$, the following jump condition holds
    \begin{align}\label{rhpgeneric:jump}
   \Phi_{+} (\lambda,x) = 
    \Phi_{-}(\lambda,x)G(\lambda,x).
    \end{align}
\item We have the normalization $\Phi(\ii K',x)=\mathbb{1}$.
\end{enumerate}
\end{rhp}
According to \cite{rodin1989} the RHP \ref{rhp:generic1} has solution $\Phi(\lambda,x)$ under the soliton - free condition $\sfa(\lambda)\neq 0$. We want to get singular integral representation for the solution.

With the Cauchy kernel \eqref{def:Cauchy}, we associate the following
equation to RHP \ref{rhp:generic1}.
\begin{align}\label{def:W}
    \chi(\lambda,x) = \mathbb{1} + \frac{1}{2\pi \ii } \int_{\Gamma_1\cup \Gamma_2} \chi(\mu,x) \left( G(\mu,x)- \mathbb{1}\right) C(\mu, \lambda-\ii 0) d\mu.
\end{align}

\begin{proposition}
Assume that ${\sf r}(\lambda)$ satisfies
properties \eqref{r_property_1}--\eqref{r_property_6} and the solution of \eqref{def:W} exists. Then the solution of RHP \ref{rhp:generic1} is described by

\begin{align}\label{def:N}
   \Phi(\lambda,x) = \mathbb{1} + \frac{1}{2\pi \ii } \int_{\Gamma_1\cup\Gamma_2} \chi(\mu,x) \left( G(\mu,x)- \mathbb{1}\right) C(\mu, \lambda) d\mu.
\end{align}
\end{proposition}
\begin{proof}
From the expression for $\Phi(\lambda,x)$ it seems that it might have pole at $\lambda=K+\ii K'$. Let us show that in fact it is not true. As it has already been mentioned the RHP \ref{rhp:generic1} has solution $\Phi(\lambda,x)$ under the soliton - free condition $\sfa(\lambda)\neq 0$. Denote the right hand side of \eqref{def:N} as $\widetilde{\Phi}(\lambda,x)$. Consider $\widehat{\Phi}(\lambda,x)=\widetilde{\Phi}(\lambda,x)\left(\Phi(\lambda,x)\right)^{-1}$.{ Functions $\widetilde{\Phi}(\lambda,x)$ and $\Phi(\lambda,x)$ have the same jump on the contour $\Gamma_1\cup\Gamma_2$. }Therefore it is an elliptic function with a possible simple pole at $\lambda=K+\ii K'$. Since there is no elliptic function with a single simple pole, $\widehat{\Phi}(\lambda,x)$ has no singularity at $\lambda=K+\ii K'$, {\it i.e} the residue at the afore-stated point is zero. Therefore, the same holds for $\widetilde{\Phi}(\lambda,x)=\widehat{\Phi}(\lambda,x)\Phi(\lambda,x)$. {Hence the right hand side of \eqref{def:N} has no pole at $\lambda=K+\ii K'$ and} we can identify it with the solution of the RHP \ref{rhp:generic1} since it satisfies the same jump condition.
\end{proof}

\noindent We can use symmetrization to construct solution of the RHP \ref{rhp1}.
\begin{proposition}
The solution $Y(\lambda,x)$, can be recovered through the following symmetrization procedure of the solution of the above RHP \ref{rhp:generic1}
\begin{align}
Y(\lambda)=\frac{1}{\sqrt{c}}\left(\Phi(\lambda)+ \sigma_{3} \Phi(\lambda + 2K) \sigma_{3}+ \sigma_{1}\Phi(\lambda + 2\ii K') \sigma_{1}+ \sigma_{2} \Phi(\lambda + 2K+2\ii K') \sigma_{2}\right),\\ \label{eq:N-Z}\\
c={\det\left(\Phi(\lambda)+ \sigma_{3} \Phi(\lambda + 2K) \sigma_{3}+ \sigma_{1}\Phi(\lambda + 2\ii K') \sigma_{1}+ \sigma_{2} \Phi(\lambda + 2K+2\ii K') \sigma_{2}\right)}.
\end{align}
\end{proposition}
\begin{proof}
    Using symmetries \eqref{eq:a_symm}, \eqref{eq:b_symm} we can get symmetries for the jump $G(\lambda,x,t)$:
    \begin{align}
G(\lambda,x,t)=\sigma_3G(\lambda+2K,x,t)\sigma_3,&&G(\lambda,x,t)=\sigma_1G^{-1}(\lambda+2\ii K',x,t)\sigma_1,
\end{align}
\begin{align}
G(\lambda,x,t)=\sigma_2G^{-1}(\lambda+2K+2\ii K',x,t)\sigma_2.
    \end{align}
    They imply that the numerator of \eqref{eq:N-Z} has the same jump as $\Phi(\lambda,x,t)$. In particular, it implies that its determinant does not depend on $\lambda$ and dividing by the determinant does not violate jump and analyticity condition. As a result, we get the unimodular function $Y(\lambda,x,t)$. The points where the determinant $c$ is zero correspond to the singularities of the solution $Y(\lambda,x,t)$. Based on the large $x$ asymptotic of $\Phi(\lambda,x)$ obtained in the next section \ref{sec:initial_condition_asymptotic} we can tell that the determinant $c$ is not identically zero.
\end{proof}
\subsection{Reconstruction of the linear equation \eqref{def:LaxU} from $Y(\lambda, x)$} \label{sec:initial_condition_asymptotic}   
The solution of the RHP \ref{rhp1} solves the differential equation \eqref{def:LaxU}. We will omit the proof of this statement and instead refer the reader to Proposition \ref{prop:lax_pair_proof} where we prove an analogous statement for the time depended  RHP ,
(\ref{rhp2}). Considering $t=0$ case of it we can see that solution of  the RHP \ref{rhp1} denoted by $Y(\lambda,x,0)$ solves equation \eqref{eq:Ydeq}.

\begin{proposition}\label{prop:init_solution_formula}
Consider solution of  the RHP \ref{rhp1} denoted by $Y(\lambda,x)$. It solves equation \eqref{eq:Ydeq} with $L(x)$ appearing in the coefficients of the differential equation. It can be recovered from $Y(\lambda,x)$ using the relation
\begin{align}\label{sol_LL_RHP_init}
   Y(0,x) \sigma_3 \left(Y(0,x)\right)^{-1} = \sum_{j=1}^3 L_j(x) \sigma_j .
\end{align}
\end{proposition}
\begin{proof}
The above expression is obtained by considering asymptotic behaviour $\lambda\to 0$ of \eqref{eq:Ydeq}. Precisely, according to Figure \ref{fig:foobar} near $\lambda=0$, the elliptic functions $w_{\alpha}$ behave as $1/\lambda$ for all $\alpha=1,2,3$. Therefore, comparing the coefficients of $1/\lambda$ in \eqref{eq:Ydeq}, we obtain \eqref{sol_LL_RHP_init}.
\end{proof}
Now, we are coming to the main part of this section.

\begin{proof}[Proof of Theorem \ref{thm:initial_condition_asymptotic}]
To apply standard asymptotic analysis of Riemann-Hilbert problems to our case we need to construct the analytic continuation of reflection coefficient $\sfr(\lambda)$. We follow closely \cite{deift_zhou_annals,deift1994long}.

Observing that $w_3(\lambda)$ is monotone function on the interval $[0,2K]$, we can introduce a new variable $\Theta=w_3(\lambda)$ such that $-\infty<\Theta<\infty$.  Denoting $\Rho(\lambda)= \frac{{\sf r}(\lambda) }{ 1 + \vert {\sf r}(\lambda) \vert^2}$, let us now split $\Rho(\lambda)$ into analytic and decaying parts respectively \begin{align}\Rho(\lambda(\Theta))\ee^{2\ii x \Theta}=\Rho_{dec}(\lambda(\Theta),x)+\Rho_{ana}(\lambda(\Theta),x)\label{eq:Rho_split_init}
\end{align}
with
        \begin{align}
\Rho_{dec}(\Theta,x)=\intop_{x}^{\infty}\widehat{{\sf \Rho}}(s) \ee^{\ii (2x-s) \Theta}ds,&& \Rho_{ana}(\Theta,x)=\intop_{-\infty}^{x}\widehat{{\sf \Rho}}(s) \ee^{\ii (2x-s) \Theta}ds,
\end{align}
\begin{align}
\widehat{\Rho}(s)=\frac{1}{2\pi}\intop_{-\infty}^{\infty}{{\sf \Rho}}(\lambda(\Theta)) \ee^{\ii s \Theta}d\Theta.\label{eq:tzerofourier}
    \end{align}
Since $\Rho(\lambda(\Theta))\in \mathcal{S}(\mathbb{R})$, we have $\widehat{\Rho}(s)\in \mathcal{S}(\mathbb{R})$. It implies that $\Rho_{dec}(\Theta,x)$ is decaying as $x\to +\infty$ faster than any power together with all its derivatives. We also observe that $\Rho_{ana}(\Theta,x)$ is analytic in $\Theta$ and decaying exponentially as $x\to \infty$ for $\Theta\in \mathbb{C^+}$. Besides properties with respect to variable $x$ we can observe that $\Rho_{dec}(\Theta,x)$ and $\Rho_{ana}(\Theta,x)$ behave as $\mathcal{O}(\frac{1}{\Theta})$ for $\Theta\to\infty$. We further note that the properties \eqref{eq:theta-lambda-map} imply that the function $w_3(\lambda)$ maps the neighborhood of $[0,2K]$ in $\Omega_\pm$ to the neighborhood of real line in $\mathbb{C}_{\mp}$.

We construct the analogs of $\Rho_{ana}(\Theta({\lambda}),x)$ and   $\Rho_{dec}(\Theta({\lambda}),x)$ for the other parts of $\Gamma_1\cup\Gamma_2$. We can now factorize the jump in \eqref{rhp1:jump} as follows
    \begin{align*}
G(\lambda,x)    &=\left( \begin{array}{cc}
      1 & 0 \\
       \Rho_{ana}(\Theta(\lambda),x)  & 1 
    \end{array} \right) \left( \begin{array}{cc}
      1 & 0 \\
       \Rho_{dec}(\Theta(\lambda),x)  & 1 
    \end{array} \right)\\  &\times \left( \begin{array}{cc}
      1 + \vert {\sf r}(\lambda) \vert^2   & 0\\
       0 & \frac{1}{ 1 + \vert {\sf r}(\lambda) \vert^2} 
    \end{array} \right) \left( \begin{array}{cc}
      1   &  \overline{\Rho_{dec}(\Theta({\lambda}),x)}\\
       0  & 1 
    \end{array} \right)\left( \begin{array}{cc}
      1   &  \overline{\Rho_{ana}(\Theta(\bar{\lambda}),x)}\\
       0  & 1 
    \end{array} \right).
    \end{align*}
     To get rid of the diagonal jump, we introduce the function
    \begin{align}
        \label{eq:delta}\delta(\lambda) = \exp\left[\frac{1}{2\pi \ii} \int_{-2K}^{0} \log\left(1+ |r(\eta)|^2 \right) \frac{w_3(\eta-\lambda)}{\rho} d\eta \right].
    \end{align}
It satisfies the jump condition 
\begin{align}
\delta_+(\lambda)=\delta_-(\lambda)(1+|\sfr(\lambda)|^2),\quad \lambda\in\Gamma_1\cup\Gamma_2.
\end{align}
In particular, using the identity $1+|\sfr(\lambda)|^2=\frac{1}{|\sfa(\lambda)|^2}$, we can notice that 
\begin{align*}
    \delta(\lambda)=\begin{cases}
        \frac{1}{\sfa(\lambda)},\quad \lambda\in\Omega_+,
        \\ \\
{\overline{\sfa(\bar{\lambda})},\quad \lambda\in\Omega_-}.
    \end{cases}
\end{align*}

which confirms relation \eqref{ar}. We can then define
\begin{align}
    T^{(0)}(\lambda,x):=\begin{cases}
Y(\lambda,x)\left( \begin{array}{cc}
      1 & 0 \\
       {\Rho_{ana}(\Theta(\lambda),x) } & 1 
    \end{array} \right)\delta(\lambda)^{-\sigma_3},\\\hspace{4cm}\mbox{for} \begin{array}{c}-\varepsilon<\Im(\lambda)<0,\quad\mbox{and}\\ 
    -2K'<\Im(\lambda)<-2K'+\varepsilon,\quad \lambda \in \Omega_-\end{array}\\ 
    Y(\lambda,x)\left( \begin{array}{cc}
      1   & - \overline{\Rho_{ana}(\Theta(\bar{\lambda}),x)}\\
       0  & 1 
    \end{array} \right)\delta(\lambda)^{-\sigma_3},
    \\\hspace{4cm}\mbox{for}\begin{array}{c}0<\Im(\lambda)<\varepsilon,\quad\mbox{and}\\
    2K'-\varepsilon<\Im(\lambda)<2K',\quad \lambda \in \Omega_+\end{array}\\
    Y(\lambda,x)\delta(\lambda)^{-\sigma_3},\\\hspace{4cm}\mbox{otherwise.} 
\end{cases}\label{eq:T_0_def}
\end{align}
The resulting jump for $T^{(0)}(\lambda,x)$ is given by $G_{T^{(0)}}(\lambda,x)$:
\begin{align}
G_{T^{(0)}}(\lambda,x):=\begin{cases}\left( \begin{array}{cc}
      1 & 0 \\
(\delta(\lambda))^{-2}\Rho_{ana}(\Theta(\lambda),x) & 1 
    \end{array} \right),\quad \begin{array}{c}\Im(\lambda)=-\varepsilon,\quad\mbox{and}\\
\Im(\lambda)=-2K'+\varepsilon,\quad \lambda \in \Omega_-\end{array}\\ \\
    \left( \begin{array}{cc}
      1   &  (\delta(\lambda))^2\overline{\Rho_{ana}(\Theta(\bar{\lambda}),x)}\\
       0  & 1 
    \end{array} \right),\quad \begin{array}{c}\Im(\lambda)=\varepsilon,\quad\mbox{and}\\
    \Im(\lambda)=2K'-\varepsilon,\quad \lambda \in \Omega_+\end{array}\\ \\
    \left( \begin{array}{cc}
      1 & 0 \\
      (\delta_+(\lambda)\delta_-(\lambda))^{-1} \Rho_{dec}(\Theta(\lambda),x)  & 1 
    \end{array} \right)\times\\
    \left( \begin{array}{cc}
      1   & \delta_+(\lambda)\delta_-(\lambda) \overline{\Rho}_{dec}(\Theta({\lambda}),x)\\
       0  & 1 
    \end{array} \right),\quad \lambda\in\Gamma_1\cup\Gamma_2
    \end{cases}
\end{align}
Note that $G_{T^{(0)}}(\lambda,x)-\mathbb{1}$ is small for positive $x$ and large $\Theta$, that is $|G_{T^{(0)}}-\mathbb{1}|=\mathcal{O}(\frac{x^{-k}}{1+|\Theta(\lambda)|}), \ x \to \infty$ for any $k\in\mathbb{
N}$. 
 We now use the singular integral equation and the symmetrization procedure to derive the asymptotic of $L(x)$ following section \ref{sec:singular_integral_equation}. To begin with, consider singular integral equation
 \begin{align}\label{def:chi-_zero}
    \chi^{(0)}(\lambda,x) = \mathbb{1} + \frac{1}{2\pi \ii } \int_{\Sigma^{(0)}} \chi^{(0)}(\mu,x) \left( G_{{T}^{(0)}}(\mu,x)- \mathbb{1}\right) C(\mu, \lambda-\ii 0) d\mu,
\end{align}
and
\begin{align}
    \Phi^{(0)}(\lambda,x) = \mathbb{1} + \frac{1}{2\pi \ii } \int_{\Sigma^{(0)}} \chi^{(0)}(\mu,x) \left( G_{{T}^{(0)}}(\mu,x)- \mathbb{1}\right) C(\mu, \lambda) d\mu.
\end{align}
where 
 \begin{align*}
\Sigma^{(0)}=\Gamma_1\cup\Gamma_2 \cup\{\Im(\lambda)=\varepsilon,\lambda\in \Omega_+\}\cup\{\Im(\lambda)=2K'-\varepsilon,\lambda\in \Omega_+\}\\
 \hspace*{2cm}\cup\{\Im(\lambda)=-\varepsilon,\lambda\in \Omega_-\}\cup\{\Im(\lambda)=-2K'+\varepsilon,\lambda\in \Omega_-\}.
 \end{align*}
According to \cite{rodin1989} normalized version of RHP \ref{rhp1} is solvable. Therefore the solution of normalized version of Riemann-Hilbert problem for $T^{(0)}(\lambda,x)$ is also solvable and solution is given by $ \Phi^{(0)}(\lambda,x)$. We proceed with standard small norm theorem computation. The estimate for $G_{T^{(0)}}$ and the boundedness of singular integral operator with Cauchy kernel in $L_2$ provides us with estimate $\|\chi^{(0)}-\mathbb{1}\|_{L_2(\Sigma^{(0)})}=\mathcal{O}(x^{-k})$ for any $k\in\mathbb{N}$. We proceed to estimate $\Phi^{(0)}(\lambda,x)$
\begin{align}
   |\Phi^{(0)}(\lambda,x)-\mathbb{1}|\leq \frac{1}{2\pi }\left| \int_{\Sigma^{(0)}}  \left( G_{{T}^{(0)}}(\mu,x)- \mathbb{1}\right) C(\mu, \lambda) d\mu\right|\\+\frac{1}{2\pi  } \left|\int_{\Sigma^{(0)}} (\chi^{(0)}(\mu,x)-\mathbb{1}) \left( G_{{T}^{(0)}}(\mu,x)- \mathbb{1}\right) C(\mu, \lambda) d\mu\right|
\end{align}
Using the estimate for $G_{T^{(0)}}$ we deduce that
\begin{equation}
|\Phi^{(0)}(\lambda,x)-\mathbb{1}|=\mathcal{O}(x^{-k}), \ x \to \infty \mbox{ for any } k\in\mathbb{
N} \mbox{ and for } \mathrm{dist}(\lambda,\Sigma^{(0)})\geq \frac{c}{1+|w_3(\lambda)|}.
\end{equation}
We can notice that function 
$T^{(0)}(\lambda,x) $ can be obtained from  $\Phi^{(0)}(\lambda,x) $ using symmetrization procedure
    \begin{align}
T^{(0)}(\lambda)=\frac{1}{\sqrt{c}}\left(\Phi^{(0)}(\lambda)+ \sigma_{3} \Phi^{(0)}(\lambda+2K) \sigma_{3}+ \sigma_{1}\Phi^{(0)}(\lambda+2 \ii K') \sigma_{1}\right.\\\left.+\sigma_{2} \Phi^{(0)}(\lambda+2K+2\ii K')\sigma_{2}\right),\\c=\det\left(\Phi^{(0)}(\lambda)+ \sigma_{3} \Phi^{(0)}(\lambda+2K) \sigma_{3}+ \sigma_{1}\Phi^{(0)}(\lambda+2 \ii K') \sigma_{1}\right.\\\left.+\sigma_{2} \Phi^{(0)}(\lambda+2K+2\ii K')\sigma_{2}\right).
\end{align}
and $|T^{(0)}(\lambda,x)-\mathbb{1}|=\mathcal{O}(x^{-k})$, for any $k\in\mathbb{
N}$ and for $\mathrm{dist}(\lambda,\Sigma^{(0)})\geq \frac{c}{1+|w_3(\lambda)|}$. To obtain scattering data corresponding to solution $L(x)$ we look at $x\to+\infty$ asymptotics of $Y(\lambda,x)$ which we just obtained and relate them to Jost solutions $F_+(\lambda,x)$ corresponding to $L(x)$. More precisely,
\begin{align*}
Y_+(\lambda,x)=T^{(0)}(\lambda,x)\delta(\lambda)^{\sigma_3}\begin{pmatrix}
    1& \Rho_{ana}(\Theta(\lambda),x)\\
   0&1
\end{pmatrix}\\
Y_-(\lambda,x)=T^{(0)}(\lambda,x)\delta(\lambda)^{\sigma_3}\begin{pmatrix}
    1&0\\
    -\Rho_{ana}(\Theta(\lambda),x)&1
\end{pmatrix}.
\end{align*}
Here $|T^{(0)}(\lambda,x)-\mathbb{1}|=\mathcal{O}(x^{-k}),\quad x\to\infty$ for any $k\in\mathbb{
N}$ and for $\mathrm{dist}(\lambda,\Sigma^{(0)})\geq \frac{c}{1+|w_3(\lambda)|}$.
Denote
\begin{align}
    Y_\pm(\lambda,x) = \left(y_{\pm}^{(1)}(\lambda,x), y_{\pm}^{(2)}(\lambda,x) \right).
\end{align}
Then from the computation above we can observe that 
\begin{align*}
        F_{+}(\lambda,x)=(y_+^{(1)}(\lambda,x)(\delta_+(\lambda))^{-1},y_-^{(2)}(\lambda,x) \delta_-(\lambda))\ee^{- \ii x w_3(\lambda)\sigma_3}. 
\end{align*}
As a result, we also can get decay properties of $L(x)$ for $x\to\infty$. Actually, we see that
\begin{align*}
Y(0,x)=T^{(0)}(0,x),
\end{align*}
which follows from the relations $\Rho_{ana}(\Theta(0),x)=0,\quad \sfa(0)=1$. The last relation follows from the identity
\begin{align*}
    \sfa(0)=\exp\left[\frac{1}{2\pi \ii} \int_{-2K}^{0} \log\left(1+ |\sfr(\eta)|^2 \right) \frac{w_3(\eta)}{\rho} d\eta \right]\\=\exp\left[\frac{1}{2\pi \ii} \int_{0}^{2K} \log\left(1+ |\sfr(\eta)|^2 \right) \frac{w_3(\eta)}{\rho} d\eta \right].
\end{align*}
Here we used the symmetries $\lambda\to-\lambda$ and $\lambda\to\lambda+2K$.

Similarly, for $x\to-\infty$ we can use alternative factorization
 \begin{align*}
G(\lambda,x)    =\left( \begin{array}{cc}
      1 &  \overline{{\sf r}(\lambda)}\ee^{-2 \ii x w_3(\lambda)}  \\
      0 & 1 
    \end{array} \right) \left( \begin{array}{cc}
      1 & 0\\
      {{\sf r}(\lambda)}\ee^{2 \ii x w_3(\lambda)}  & 1 
    \end{array} \right) .
    \end{align*}
    Now, instead of splitting $\Rho(\lambda)\ee^{2\ii x w_3(\lambda)}$ we split ${\sf r}(\lambda)\ee^{2 \ii x w_3(\lambda)}$ in decaying and analytic terms. Repeating similar computations as above, we obtain the decay properties of $L(x)$ for $x\to -\infty$. Moreover, we can write the second Jost solution in terms of $Y(\lambda,x)$.
    \begin{align*}
    F_{-}(\lambda,x)=(y_-^{(1)}(\lambda,x),y_+^{(2)}(\lambda,x))\ee^{- \ii x w_3(\lambda)\sigma_3} 
\end{align*}

Using the jump condition \eqref{rhp1:jump} we can check that Jost solution satisfy \eqref{jump_jost}, with $\sfa(\lambda)=\frac{1}{\delta(\lambda)}$ and $\sfb(\lambda)=\sfr(\lambda)\sfa(\lambda)$, as expected.
\end{proof}
\section{Turning on time }\label{sec:time}

We shall start with the fundamental (though conditional) fact established in \cite{Sklyanin:121210} that,  assuming that the LL equation has solution in the
Schwartz class, the corresponding scattering data evolve in a very simple way.
\begin{proposition}(\cite{Sklyanin:121210})\label{prop:t_dependence} Assume $t>0$ is fixed and $$L_1(x,t), L_2(x,t), L_3(x,t)-1 \in \mathcal{S}(\mathbb{R}).$$ Denote by $F_\pm(\lambda,x,t)$ the Jost solutions corresponding to $L(x,t)$ with $t>0$.
    Then the matrix-valued functions $$J_\pm(\lambda,x,t)=F_\pm(\lambda,x,t)\ee^{2\ii t w_1(\lambda) w_2(\lambda)\sigma_3}$$ solves equations \eqref{def:LaxU}, \eqref{def:LaxV}. Moreover the scattering data depends on time as \begin{align}\label{eq:scat_evolve}
{\sf a}(\lambda,t)={\sf a}(\lambda), \quad {\sf b}(\lambda,t)={\sf b}(\lambda)\ee^{-4\ii tw_1(\lambda)w_2(\lambda)}.
\end{align}
\end{proposition}
\begin{proof}
As \eqref{def:LaxU} is the derivative w.r.t $x$, it is immediate that $J_\pm(\lambda,x,t)$ solves it. To show that \eqref{def:LaxV} is satisfied by $J_\pm(\lambda,x,t)$, denote $$W_\pm(\lambda,x,t)=\partial_t J_\pm(\lambda,x,t) -V(\lambda,x,t)J_\pm(\lambda,x,t).$$ Using compatibility condition \eqref{eq:comp-cond} we can see that $W_\pm(\lambda,x,t)$ satisfies \eqref{def:LaxU}, which implies that $W_\pm(\lambda,x,t)=F_\pm(\lambda,x,t) C_\pm(\lambda,t) $ for some constant in $x$ matrices $C_\pm(\lambda,t)$. Using asymptotic condition \eqref{asymp_jost} we see that $C_\pm(\lambda,t)\equiv0$. This shows that $J_\pm(\lambda,x,t)$ solves \eqref{def:LaxV}.

Finally, consider the scattering matrix $$S(\lambda,t)=\ee^{2\ii t w_1(\lambda) w_2(\lambda)\sigma_3}\left(J_-(\lambda,x,t)\right)^{-1}J_+(\lambda,x,t)\ee^{-2\ii t w_1(\lambda) w_2(\lambda)\sigma_3}.$$ Differentiating the expression above, we see that $S(\lambda,t)$ solves equation 
\begin{align}\label{eq:S_diff}
        {\partial_t S(\lambda,t)}{}=-2\ii w_1(\lambda)w_2(\lambda)[S(\lambda,t),\sigma_3],
    \end{align}
    which implies \eqref{eq:scat_evolve}.
\end{proof}

Proposition (\ref{prop:t_dependence}) shows that the Riemann-Hilbert problem associated with the LL equation \eqref{def:eqLL},
i.e. whose solution yields the solution of the Cauchy problem for the LL equation, can be formulated
in the following way.
 .

\begin{rhp}\label{rhp2}
\begin{enumerate}
    \item The function $Y(\lambda,x,t)$
    is bounded and piecewise analytic for $\lambda\in \mathbb{T}^2/(\Gamma_{1} \cup \Gamma_{2})$, with $\mathbb{T}$ denoting the fundamental domain of a torus, and the contours $\Gamma_{1}, \Gamma_{2}$ as specified in Figure \ref{Contour_1}. Note that $Y(\lambda,x,t)$ is doubly periodic in $\lambda$.
    \item For $\lambda \in \Gamma_{1}, \Gamma_{2}$, the following jump condition holds
    \begin{equation}\label{rhp2:jump}
    Y_{+} (\lambda,x,t) = 
    Y_{-}(\lambda,x,t)G(\lambda,x,t); \qquad G(\lambda,x,t) = \begin{pmatrix}
 1+|{\sf r}(\lambda,t)|^2 & \overline{{\sf r}(\lambda,t) }\ee^{-2\ii t p(\lambda,\kappa)} \\
         {\sf r} (\lambda,t)\ee^{2\ii t p(\lambda,\varkappa)} & 1 \end{pmatrix}.
\end{equation}
 where
\begin{align}\label{def:plambda}
   p(\lambda,\varkappa):= \varkappa w_3(\lambda)-2w_1(\lambda)w_2(\lambda),\quad \varkappa=\frac{x}{t}.
\end{align}
\item The function $Y(\lambda,x,t)$ satisfies the following symmetry conditions, as can be see though the identities in \eqref{eq:w-derivatives}:
  \begin{align}\label{rhp2:symmetry}
 \sigma_{3} Y(\lambda + 2K,x,t) \sigma_{3} = Y(\lambda,x,t), &&
 \sigma_{1} Y(\lambda + 2\ii K',x,t) \sigma_{1} = Y(\lambda,x,t).
\end{align}
\item Function $Y(\lambda,x,t)$ satisfies normalization condition
$\det(Y(\lambda,x,t))=1$.
\end{enumerate}
\end{rhp}

It is convenient for further computation to introduce notation
\begin{equation}\label{def:asympY}
    Y(0,x,t)=\Psi_1(x,t)
\end{equation}
\begin{proposition}
The solution $Y(\lambda,x,t)$ of the RHP \ref{rhp2} is unique up to a sign.
\end{proposition}
\begin{proof}
The proof is similar to the proof of Proposition \ref{prop:uniq_init} (see also \cite{rodin1984}). 
\end{proof}
\begin{proposition}
    The RHP above is solvable for any $\sfr(\lambda)$ satisfying properties \eqref{r_property_1}-\eqref{r_property_6} and the solution is smooth function
    with respect to $x$ and $t$
\end{proposition}
Solvability of the RHP \ref{rhp2} was demonstrated in \cite{rodin1989}. The smoothness in $x$ and $t$ follows from property
(6) of the reflection coefficient $\sfr(\lambda)$. Strictly speaking, this fact is a standard properties of the Riemann-Hilbert
problems posed on the Riemann sphere (see e.g. \cite{zhou1989riemann}). the extension to the Torus is quite straightforward and
we omit it for the sake of not making our paper too long.

As usual with the integrable equations, the RHP above is instrumental in studying the asymptotics of the LL equation in the following way. First, it gives the solution to the linear system \eqref{def:LaxU}, \eqref{def:LaxV}. Secondly, studying its asymptotics will lead to the asymptotic behaviour of $L(x,t)$ in \eqref{def:eqLL}. We start  with the following results.
\begin{proposition}\label{prop:lax_pair_proof}
   The function 
   \begin{align}
      \Psi(\lambda,x,t)=Y(\lambda,x,t)\ee^{-\ii t p(\lambda, \varkappa)  \sigma_3}
   \end{align}
   solves \eqref{def:LaxU}, \eqref{def:LaxV} with $L_j(x,t)$ given by \eqref{sol_LL_RHP_init_1}.
\end{proposition}
\begin{proof}
Let us begin with the logarithmic derivative of $\Psi(x,t)$ w.r.t $x$. We can notice that the expression $\partial_{x} \Psi(\lambda, x, t) \Psi(\lambda, x, t)^{-1}$ has no jump, {\it i.e}
    \begin{align}
        \partial_{x}\Psi_{+}(\lambda, x, t) \Psi_{+}(\lambda, x, t)^{-1} =  \partial_x \Psi_{-}(\lambda, x, t) \Psi_{-}(\lambda, x, t)^{-1}.
    \end{align}
Therefore    $\partial_{x} \Psi(\lambda, x, t) \Psi(\lambda, x, t)^{-1} $ is an elliptic function. We analyse its behaviour at $\lambda\to 0$:
    \begin{align}
    \partial_{x} \Psi(\lambda, x, t) \Psi(\lambda, x, t)^{-1} = \partial_{x} Y(\lambda, x, t) Y(\lambda, x, t)^{-1} - \ii w_{3}(\lambda) Y(\lambda, x, t) \sigma_{3} Y(\lambda, x, t)^{-1},
\end{align}
We rewrite it using notation 
\begin{align}\label{eq:asymp_Psi}
    \partial_{x} \Psi(\lambda, x, t) \Psi(\lambda, x, t)^{-1} = - \ii w_{3}(\lambda) \Psi_1(x, t) \sigma_{3} \Psi_1(x, t)^{-1} + \mathcal{O}(1),\quad \lambda\to 0.
\end{align}
We conclude that $\partial_{x} \Psi(\lambda, x, t) \Psi(\lambda, x, t)^{-1} $ has a simple pole at $\lambda=0$ with residue $$- \ii w_{3}(\lambda) \Psi_1(x, t) \sigma_{3} \Psi_1(x, t)^{-1}.$$ Moreover, it satisfies the symmetries
\begin{align}\label{eq:log_der_symm_1}
 \sigma_{3}  \partial_{x}\Psi(\lambda+2K, x, t) \Psi(\lambda+2K, x, t)^{-1} \sigma_{3} &=  \partial_{x}\Psi(\lambda, x, t) \Psi(\lambda, x, t)^{-1}, \\
 \sigma_{1}  \partial_{x}\Psi(\lambda+2\ii K', x, t) \Psi(\lambda+2\ii K', x, t)^{-1} \sigma_{1} &=  \partial_{x}\Psi(\lambda, x, t) \Psi(\lambda, x, t)^{-1},\label{eq:log_der_symm_2}
\end{align}
which imply the presence of simple poles at $\lambda=2K, 2\ii K', 2K+2\ii K'$ with residues \begin{align*}
    - \ii \sigma_3 w_{3}(\lambda) \Psi_1(x, t) \sigma_{3} \Psi_1(x, t)^{-1}\sigma_3, &&
    - \ii \sigma_1 w_{3}(\lambda) \Psi_1(x, t) \sigma_{3} \Psi_1(x, t)^{-1}\sigma_1, 
    \end{align*}
    \begin{align*}
    - \ii \sigma_2 w_{3}(\lambda) \Psi_1(x, t) \sigma_{3} \Psi_1(x, t)^{-1}\sigma_2
\end{align*} respectively. 
Let's write 
\begin{align}\label{eq:this}
   -\ii \Psi_1(x, t) \sigma_{3} \Psi_1(x, t)^{-1} {=:} - \ii \sum_{j=1}^3 c_{j}(x,t) \sigma_j.
\end{align}
Since the determinant of both the sides above is one, we deduce that $c_j(x,t)$ is a unit vector. Using the properties \eqref{symm:w1}-\eqref{symm:w3} we can notice that the expression\\ $-\ii\sum_{j=1}^3 c_{j}(x,t) \sigma_j w_j(\lambda)$ is an elliptic function which has the same residues as \\ $-\ii \Psi_1(x, t) \sigma_{3} \Psi_1(x, t)^{-1}$ at $\lambda=0, 2K, 2K+2\ii K', 2K+2\ii K'$. Therefore

\begin{align}\label{eq:A-ansatz}
    \partial_{x} \Psi(\lambda, x, t) \Psi(\lambda, x, t)^{-1} = -\ii \sum_{j=1}^3 c_{j}(x,t) \sigma_j w_j(\lambda)+\mathrm{const}.
\end{align}

\noindent Symmetries \eqref{eq:log_der_symm_1}-\eqref{eq:log_der_symm_2} imply that the constant matrix above is diagonal.
Since \\$\det(\Psi(\lambda,x,t))=1$, trace of the right hand side of \eqref{eq:A-ansatz} is zero and the constant matrix above is zero.

We can carry out a parallel study for the $t$ derivative. The above properties (1), (2) hold and the logarithmic derivative
    \begin{align}
    \partial_{t} \Psi(\lambda, x, t) \Psi(\lambda, x, t)^{-1} = \partial_{t} Y(\lambda, x, t) Y(\lambda, x, t)^{-1} +2\ii w_{1}(\lambda) w_2(\lambda) Y(\lambda, x, t) \sigma_{3} Y(\lambda, x, t)^{-1}.
\end{align}
In the limit $\lambda\to 0$, we consider the subleading term
\begin{align}
    Y(\lambda,x,t) = \Psi_1(x,t) + \lambda\, \Psi_2(x,t)+ \mathcal{O}(\lambda^2),
\end{align}
and substituting in the expression above, we obtain that
\begin{align}
    \partial_{t} \Psi(\lambda, x, t) \Psi(\lambda, x, t)^{-1} &= 2\ii w_{1}(\lambda) w_2(\lambda)  \Psi_1(x, t) \sigma_{3} \Psi_1(x, t)^{-1} \nonumber\\
    &+2 \ii \lambda w_{1}(\lambda) w_2(\lambda) \left( \Psi_{2}(x,t) \sigma_3 \Psi_1(x,t) - \Psi_{1}(x,t) \sigma_3 \Psi_2(x,t)\right)+ \mathcal{O}(1).
    \end{align}
    The expression $\partial_{t} \Psi(\lambda, x, t) \Psi(\lambda, x, t)^{-1}$ therefore has terms with double poles and simple poles at $\lambda=0$. Along with \eqref{eq:this} we write 
\begin{align}\label{eq:that}
   2 \ii\left( \Psi_{2}(x,t) \sigma_3 \Psi_1(x,t) - \Psi_{1}(x,t) \sigma_3 \Psi_2(x,t)\right) {=} \ii \sum_{j=1}^3 d_{j}(x,t) \sigma_j.
\end{align}
    Producing the argument similar to above one gets the following form
    \begin{align}
      \partial_{t} \Psi(\lambda, x, t) \Psi(\lambda, x, t)^{-1} &=    2\ii c_1(x,t) w_{2}(\lambda)w_{3}(\lambda) \sigma_1 + 2\ii c_2(x,t)  w_{3}(\lambda)w_{1}(\lambda)\sigma_2 \nonumber\\+2\ii c_3(x,t)  w_{1}(\lambda)w_{2}(\lambda)\sigma_3 
      &+ \ii d_1(x,t) w_{1}(\lambda) \sigma_1 + \ii d_2(x,t) w_{2}(\lambda) \sigma_2 + \ii d_3(x,t) w_{3}(\lambda) \sigma_3.\label{eq:B-ansatz}
    \end{align}
    The compatibility relation of \eqref{eq:A-ansatz}, \eqref{eq:B-ansatz} implies that $d_j(x,t)=\left({\partial_x c}(x,t) \times c(x,t)\right)_j$ and $c_j(x,t)$ solve the equations \eqref{def:eqLL}. Therefore, $c_i(x,t) = L_i(x,t)$, $d_i(x,t) =P_i(x,t)$ and \eqref{eq:this} becomes \eqref{sol_LL_RHP_init_1}. 
\end{proof}

\begin{proposition}
The solutions of the Landau-Lifshitz equation $L_j(x,t)$ are real valued.
\end{proposition}
\begin{proof}
  The solution of the Riemann-Hilbert problem \eqref{rhp2} has the symmetry
  \begin{align}
      \Psi(\lambda,x,t) =\pm \sigma_2 \overline{\Psi(\overline{\lambda},x,t)} \sigma_2,
  \end{align}
  as can be seen through the symmetry of the matrix $G(\lambda, x,t)$ in \eqref{rhp2:jump}. Plugging in the behaviour at $\lambda=0$ we see that $L_j(x,t)$ are real valued.
\end{proof}

 \section{Nonlinear steepest descent analysis}\label{sec:NSD}
We are going to develop nonlinear steepest descent method to compute asymptotics of $L(x,t)$ as $t\to \infty$ and $0< m\leq\varkappa=\frac{x}{t}\leq M$. The procedure of nonlinear steepest descent involves constructing the solution of RHP with jump close to identity. This construction is performed in several steps. First we construct the analytic continuations of jump matrices and use them to open lenses. Then we construct global parametrix in Subsection \ref{subsec:global} using the fact that the corresponding jumps are diagonal. Finally we construct local parametrices in Subsection \ref{subsec:local} in the neighbourhood of stationary points of characteristic exponent, which in the present case is $p(\lambda,\varkappa)$ defined in \eqref{def:plambda}. We then use the singular integral equation \eqref{def:W} and symmetrisation formula \eqref{eq:N-Z} to implement the analog of small norm theorem in our setup (see Subsection \ref{subsec:small_norm}).

The characteristic exponent $p(\lambda,\varkappa)$ has the periodicity properties
\begin{align}\label{symm:p}
    p(\lambda + 2K,\varkappa) = p(\lambda,\varkappa), && p(\lambda + 2\ii K',\varkappa) = - p(\lambda,\varkappa),
\end{align}
which can determined through the equations using identities \eqref{eq:w-derivatives}. Compute the derivative
\begin{equation}
    \partial_\lambda p(\lambda,\varkappa)
    =\frac{1}{\rho}\left(2w_3(\lambda)(w_1^2(\lambda)+w_2^2(\lambda))-{\varkappa}w_1(\lambda)w_2(\lambda)\right). \label{eq:pprimelambda}\end{equation}
We are interested in the critical points which are zeroes of $\partial_\lambda p(\lambda,\varkappa)$. It is an elliptic function with periods $2K$, $2\ii K'$ and pole of third order at $\lambda=0$. Therefore it has three zeroes at the quadrant $-2K\leq\mathrm{Re}(\lambda)\leq 0$, $0\leq \mathrm{Im}(\lambda)\leq 2\ii K'$. We are interested at the real zeroes. We consider the ratio \begin{equation}
W(\lambda)=\dfrac{w_3(\lambda)(w_1^2(\lambda)+w_2^2(\lambda))
}{w_1(\lambda)w_2(\lambda)}=\rho\dfrac{\cn(\lambda,k)(\dn^2(\lambda,k)+1)}{\sn(\lambda,k)\dn(\lambda,k)}.
\end{equation}
Using differential identities \cite[\href{http://dlmf.nist.gov/22.13.T1 }{ 23.13.1}]{DLMF} and elementary identities \cite[\href{http://dlmf.nist.gov/22.6.E1 }{ 22.6.1}]{DLMF} we compute \begin{equation}
W'(\lambda)=\rho\dfrac{3k^2\sn^2(\lambda,k)-2-k^4\sn^6(\lambda,k)}{\sn^2(\lambda,k)\dn^2(\lambda,k)}.
\end{equation}
The numerator of the expression above attains maximum of $2(k-1)$ for $\sn^2(\lambda,k)=k^{-1}$, which implies that $W'(\lambda)<0$ and $W(\lambda)$ is monotone for $-2K<\lambda<0$.    That means that we have only one stationary point on the segment $\left[-2K,0 \right] \subset\Gamma_1$, which we denote $\lambda_0$.
\begin{equation}
     \partial_\lambda p(\lambda_0,\kappa)=0, \quad\lambda_0<0
\end{equation}

\noindent The periodicity properties of $p(\lambda,\varkappa)$ in \eqref{symm:p} then imply that the points $\lambda_0+2K$, $\lambda_0+2\ii K'$, $\lambda_0+2K+2\ii K'$ are also stationary points of the characteristic exponent. 
We illustrate the sign chart of $\Im(p(\lambda,\varkappa))$ on Figure \ref{fig:sdc}. We see that there is only one stationary point on the interval $[-2K,0]$.
\begin{figure}[h!]
 \includegraphics[trim={0cm 0cm 0cm 0cm},clip, width=12cm]{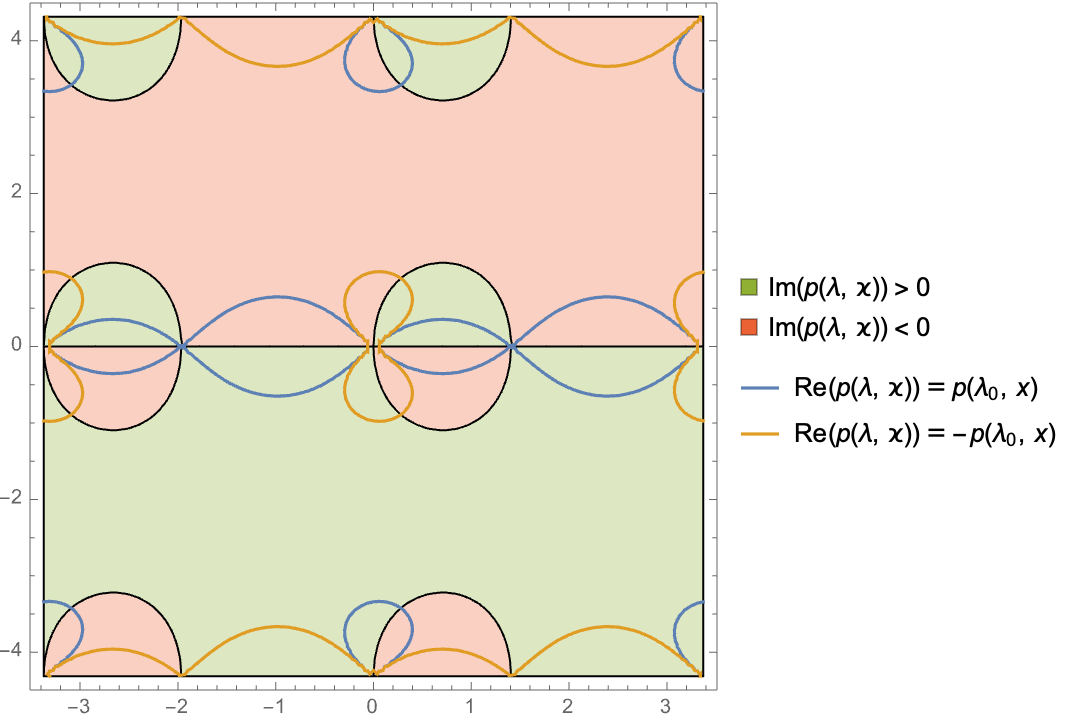}
\caption{Take $\varkappa=1$, $k=\frac{1}{2}$. We draw the following sign chart: the red and green parts denote the negative and positive domains of the imaginary part of the characteristic exponent $\Im(p(\lambda,\varkappa))$. We also draw the stokes lines $\Re(p(\lambda,\varkappa)=p(\lambda_0,\varkappa)$ centered at $\lambda_0$, $\lambda_0+2K$ with blue color and the stokes lines $\Re(p(\lambda,\varkappa)=-p(\lambda_0,\varkappa)$ centered at $\lambda_0+2\ii K'$, $\lambda_0+2K+2\ii K'$ with orange color.}\label{fig:sdc}
\end{figure}
We now restrict to the neighbourhood of the point $\lambda_0$ and define the subdomain to the torus
\begin{align*}
   \mathbb{T}^2(\lambda_0):=\left\lbrace \lambda: -2K\leq \Re \lambda\leq 0, |\Im \lambda|\leq K'\right\rbrace.
\end{align*}
The lens opening procedure is then carried out by noting that the jump matrix $G(\lambda,x,t)$ defined in \eqref{rhp2:jump} can be factorized in terms of upper triangular, lower triangular, and diagonal matrices as 
\begin{align}
    G &= \left( \begin{array}{cc}
      1    & \bar{{\sf r}} \ee^{-2 \ii t p} \\
     0  & 1 
    \end{array} \right) \left( \begin{array}{cc}
      1    & 0 \\
       {\sf r} \ee^{2 \ii t p}  & 1 
    \end{array} \right)=\widetilde{U}\widetilde{L}  \label{facG:1}\\
     & =\left( \begin{array}{cc}
      1 & 0 \\
       \frac{{\sf r} \ee^{2 \ii t p}}{ 1 + \vert {\sf r} \vert^2}  & 1 
    \end{array} \right)  \left( \begin{array}{cc}
      1 + \vert {\sf r} \vert^2   & 0\\
       0 & \frac{1}{ 1 + \vert {\sf r} \vert^2} 
    \end{array} \right) \left( \begin{array}{cc}
      1   & \frac{\bar{{\sf r}} \ee^{-2 \ii t p}}{ 1 + \vert {\sf r} \vert^2} \\
       0  & 1 
    \end{array} \right)=LDU\label{facG:2}.
\end{align}

In order to define the lens opening for generic reflection coefficient $\sf r$, we need to introduce the analytic continuation of the reflection coefficient in the upper and lower half-planes. Based on Figure \ref{fig:sdc} we can notice that the function $p(\lambda,\varkappa)$ maps the lenses to the upper and lower half-planes respectively, and the segment $[-2K,\lambda_0]$ to the half line $(-\infty,p(\lambda_0,\varkappa)]$.

Following \cite{deift_zhou_annals,deift1994long} we now consider the Taylor polynomial of degree $k$ of ${\sf r}(\lambda(p))(p+\ii)^n$ at $p_0=p(\lambda_0,\varkappa)$, $n>k$. Denote its error term by ${\mathcal{E}}_k(p)$. We then observe that ${\mathcal{E}}_k(p)$ vanishes up to the $k$th derivative at $p_0$, and continue it by zero for $p>p_0$. 

The analytic continuation can be implemented using Fourier transform in a similar way to \eqref{eq:tzerofourier}: 
\begin{align}
{\sf r}(\lambda(p))\ee^{2\ii t p}&={\sf r}_{dec}(p,t)+{\sf r}_{ana}(p,t)\label{eq:analytic_continuation},
\end{align}
where
\begin{align}
&{\sf r}_{dec}(p,t):=\intop_{t}^{\infty}\widehat{{\mathcal{R}}_k}(s,t) \ee^{\ii (2t-s) p}ds,\\& {\sf r}_{ana}(p,t):=\frac{\ee^{2\ii t p}}{(p+\ii)^n}\sum_{j=0}^k\dfrac{\partial^j ({\sf r}(\lambda(p))(p+\ii)^n)}{\partial p^j}(p_0)\dfrac{(p-p_0)^j}{j!}+\intop_{-\infty}^{t}\widehat{{\mathcal{R}}_k}(s,t)  \ee^{\ii (2t-s) p}ds,
\end{align}
and the Fourier transform of the reflection coefficient
\begin{align}
\widehat{{\mathcal{R}}_k}(s,t) =\frac{1}{2\pi}\intop_{-\infty}^{\infty}{\mathcal{R}}_k(\lambda(p),t)\ee^{\ii s p}dp, && {\mathcal{R}}_k(\lambda(p),t)=\frac{{\mathcal{E}}_k(\lambda(p),t)}{(p+\ii)^n}.
    \end{align}
We see that ${\sf r}_{ana}(p(\lambda),t)$ is analytic in the upper lens and ${\sf r}_{dec}(p(\lambda),t)=O(t^{-k})$, $t\to\infty$.
One can now mimic \cite{deift1994long} and  proceed with the lens opening procedure, by taking into account the analytic part of $\sfr(\lambda)$ in the factorization of the jump matrix $G$.
    
    In order to simplify the corresponding analysis, we shall, in fact,
    assume, from the very beginning, that the   reflection coefficient $\sfr(\lambda)$
    enjoys the proper analytic properties which  would allow us to use the factorizations
    \eqref{facG:1} and \eqref{facG:2} directly. That is, we would not need to use splitting  \eqref{eq:analytic_continuation} for performing the
    lens opening. Following the same arguments as in \cite{deift1994long}, one can then show that
     the presence of $\sfr_{dec}$  will not affect  the given in Theorem \ref{thm:asymp_LL}  leading terms
    of the large time  asymptotics of the Cauchy problem for the LL equations.

\subsection*{Lens opening considering analytic reflection coefficient:}
Assuming $\sf r$ is analytic in the upper lenses, we can open lenses in the usual way.
The resulting contour $\Sigma$ is depicted in Figure \ref{fig:lens_contour1}. We then have the following identification:
\begin{align}\label{domsubdom:Sigma}
    \Sigma= \Sigma(\lambda_0)+\Sigma(\lambda_0+2K)+ \Sigma^{-1}(\lambda_0+2\ii K')+ \Sigma^{-1}(\lambda_0+2K+2\ii K')
\end{align}
where $\Sigma^{-1}(\lambda_0+2\ii K')$ and $\Sigma^{-1}(\lambda_0+2K+2\ii K')$ have arrows inverted compared to $\Sigma(\lambda_0+2\ii K')$ and $\Sigma(\lambda_0+2K+2\ii K')$.
Similarly, the torus $\mathbb{T}^2$ is given by the sub-domain $\mathbb{T}^2(\lambda_0)$ with appropriate shifts
\begin{align}\label{domsubdom:T}
    \mathbb{T}^2= \mathbb{T}^2(\lambda_0) + \mathbb{T}^2(\lambda_0+ 2K) + \mathbb{T}^2(\lambda_0+ 2\ii K') + \mathbb{T}^2(\lambda_0+2K+2\ii K').   
\end{align}
Contour $\Sigma(\lambda_0)$ and domain $\mathbb{T}^2(\lambda_0)$ are depicted in Figure \ref{fig:lens_contour}.

\begin{figure}[htb]
            \centering
            \includegraphics[trim={0cm 0cm 0cm 0cm},clip, width=15cm]{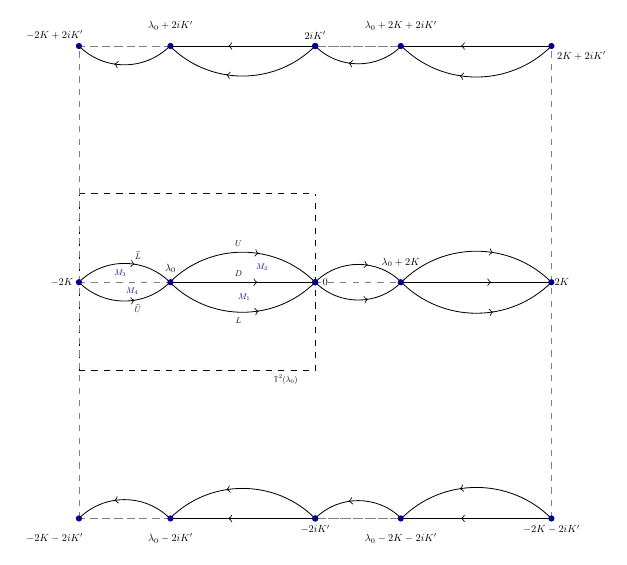}            \caption{Jump contour $\Sigma$ on the fundamental domain.}
            \label{fig:lens_contour1}
        \end{figure}
        Restricting to $ \Sigma(\lambda_0)$ in the domain $\mathbb{T}^2(\lambda_0)$ and
define a piecewise analytic function $M(\lambda,x,t)$ indicated with blue font on Figure \ref{fig:lens_contour} where

\begin{align}
         M_1 = L^{-1}, \quad
          M_2 = U, \quad
          M_3= \widetilde{L}, \quad
          M_4 = \widetilde{U}^{-1}.
\end{align}
where $L, U, \widetilde{L}, \widetilde{U}$ are indicated in \eqref{facG:1}, \eqref{facG:2}. On other parts of torus it is defined by symmetries
\begin{equation}
    M(\lambda+2K,x,t)=\sigma_3 M(\lambda,x,t)\sigma_3,\quad  M(\lambda+2\ii K',x,t)=\sigma_1 M(\lambda,x,t)\sigma_1.
\end{equation}

We can notice that
\begin{align}\label{asymp:M}
    M(\lambda,x,t) =\mathbb{1}+ \mathcal{O}(\lambda),\quad \lambda \to 0.
\end{align}
Moreover
\begin{align}\label{asymp_t:M}
    M(\lambda,x,t) =\mathbb{1}+ \mathcal{O}(t^{-k}),\quad \forall k\in\mathbb{N}\quad t\to\infty,\quad \lambda\neq \lambda_0,\lambda_0+2K,\lambda_0+2\ii K', \lambda_0+2K+2\ii K'.
\end{align}
We then define the function $T(\lambda,x,t)$ as a ratio of $M(\lambda,x,t)$ and the solution of the RHP \ref{rhp2}:
\begin{align}\label{transf:YT}
    T(\lambda,x,t):= Y(\lambda,x,t) M(\lambda,x,t)^{-1}.
\end{align}
        Note that the contour in Figure \ref{fig:lens_contour1} is simply the contour $\Sigma(\lambda_0)$ defined in the Figure \ref{fig:lens_contour} with appropriate shifts by $2K$, $2\ii K'$, and $2K + 2\ii K'$.

        \begin{figure}[H]
            \centering
            \includegraphics[trim={0cm 0cm 0cm 0cm},clip, width=14cm]{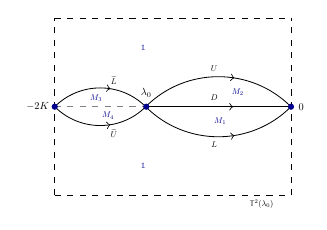}            \caption{Contour $\Sigma(\lambda_0)$ in the sub-domain $\mathbb{T}^2(\lambda_0)$.}
            \label{fig:lens_contour}
        \end{figure}

The following RHP is then solved by $T(\lambda,x,t)$.
\begin{rhp}\label{rhp:3}
    \begin{enumerate}
        \item The function $T(\lambda,x,t)$ is piecewise analytic on $\mathbb{T}^2\backslash \Sigma$,
        \item For $\lambda\in \Sigma$, it satisfies the jump condition
        \begin{align}
            T_{+}(\lambda,x,t) = T_{-}(\lambda,x,t) G_T(\lambda,x,t),
        \end{align}
        with the jump $G_T(\lambda,x,t)$ defined through \eqref{facG:1}, \eqref{facG:2} on each of the shifted contours $\Sigma(\lambda_0)$ as shown in the Figures \ref{fig:lens_contour}, \ref{fig:lens_contour1}.

        \item The symmetries conditions hold \begin{align}\label{rhp3:symmetry}
 \sigma_{3} T(\lambda + 2K,x,t) \sigma_{3} = T(\lambda,x,t), &&
 \sigma_{1} T(\lambda + 2\ii K',x,t) \sigma_{1} = T(\lambda,x,t).
\end{align}
    
\item The normalization $\det(T(\lambda,x,t))=1$ holds.
\end{enumerate}

\end{rhp}
We can notice that
        \begin{align}
            T(\lambda,x,t)= \Psi_1(x,t) + \mathcal{O}(\lambda),\quad \lambda\to 0.
        \end{align}

Also, on the contour $\Sigma (\lambda_0)$, the jump matrix $G_T$ approaches the identity as $t \to \infty$. Consequently, the asymptotic analysis of the RHP \ref{rhp:3} reduces to solving the corresponding global and local parametrices.

\subsection{Global parametrix}\label{subsec:global}
The {\it global} parametrix $T^{(gl)}(\lambda)$ solves the RHP below with the diagonal jump $D$ defined in \eqref{facG:2}.

\begin{rhp}\label{rhp:global_par}
\begin{enumerate}
\item The $t$-independent function $T^{(gl)}(\lambda)$ is analytic in \\$\mathbb{T}^2\backslash\left(\left[\lambda_0,0\right]\cup \left[\lambda_0+2K,2K\right]\cup \left[\lambda_0+2\ii K',2\ii K'\right]\cup \left[\lambda_0+2K+2\ii K',2K+2\ii K'\right]\right)$,
\item For $\lambda \in \left[\lambda_0,0\right]\cup \left[\lambda_0+2K,2K\right]\cup \left[\lambda_0+2\ii K',2\ii K'\right]\cup \left[\lambda_0+2K+2\ii K',2K+2\ii K'\right]$ it has a diagonal jump
\begin{align}
    T^{(gl)}_{+}(\lambda) = T^{(gl)}_{-}(\lambda) \left( 1+ |{\sf r}(\lambda)|^2 \right)^{\sigma_3}.\label{jump:Tgl}
\end{align}
\item The symmetry conditions hold \begin{align}\label{rhp:global_symmetry}
 \sigma_{3} T^{(gl)}(\lambda + 2K) \sigma_{3} = T^{(gl)}(\lambda), &&
 \sigma_{1} T^{(gl)}(\lambda + 2\ii K') \sigma_{1} = T^{(gl)}(\lambda).
\end{align}
\item The normalization $\det(T^{(gl)}(\lambda))=1$ holds.

\end{enumerate}
\end{rhp}
The RHP above is solvable and its solution reads
\begin{align}\label{def:Tgl_alpha}
    T^{(gl)}(\lambda)= \alpha(\lambda)^{\sigma_{3}}, 
\end{align}
where 
\begin{align}\label{def:global_alpha}
 \alpha(\lambda) := \exp\left\lbrace \frac{1}{2\pi \ii}\int^{0}_{\lambda_0} \log \left( 1+ |{\sf r}(\eta)|^2 \right) \frac{w_3(\eta-\lambda)}{\rho} d\eta \right\rbrace.
\end{align}
Note that, in the above expression, the elliptic function $w_3(\lambda)$ has a simple pole at $\lambda=0, 2K, 2\ii K', 2K+2\ii K'$ (see Figure \ref{fig:foobar}). This fact together with the symmetries \eqref{symm:w3} of function $w_3(\lambda)$ imply that $T^{(gl)}(\lambda)$ satisfies jump condition \eqref{jump:Tgl} and symmetry conditions \eqref{rhp:global_symmetry}.
Near $\lambda=0$,
\begin{align}
 \alpha(0) = \exp\left\lbrace -\frac{1}{2\pi \ii}\int_{0}^{\lambda_0} \log \left( 1+ |{\sf r}(\eta)|^2 \right) \frac{w_3(\eta)}{\rho} d\eta \right\rbrace.\label{eq:phi1}
\end{align}

We would like to deduce the behavior of the global parametrix \eqref{def:Tgl_alpha} near point $\lambda=\lambda_0$. We observe that using \eqref{eq:zeta_sigma}, \eqref{eq:w3-zeta} the function $w_3(\lambda)$ can be written as a total derivative:
\begin{align}\label{def:beta}
  \frac{ w_3(\lambda)}{\rho}= \partial_{\lambda}\left(\log \beta(\lambda)\right), &&
    \beta(\lambda) := \frac{\sigma(\lambda) \sigma(\lambda-2K)}{\sigma(\lambda +2\ii K') \sigma(\lambda-2\ii K'-2K)},
\end{align}
where $\sigma(\lambda)$ is the Weierstrass sigma function. Furthermore, the behaviour of $\sigma(\lambda)$ near $\lambda\to 0$ implies the following
\begin{equation}
  \lim_{\lambda\rightarrow 0} \frac{\sigma(\lambda)}{\lambda} =1 \quad\Rightarrow\quad  \lim_{\lambda\to \lambda_0} \frac{\beta(\lambda- \lambda_0)}{\lambda - \lambda_0} = \beta_0; \qquad\label{def:beta0}
  \beta_0 := \frac{\sigma(-2K)}{\sigma(2\ii K') \sigma(-2\ii K'-2K)}.
\end{equation} 
Using \cite[\href{http://dlmf.nist.gov/23.9.E7}{(23.9.7)}]{DLMF} we see that $\sigma(2\ii K')$ is pure imaginary. Using \cite[(6.2.18),(6.2.19)]{lawden2013elliptic} we see that $\sigma(2K+2\ii K')$ is also pure imaginary making $\beta_0$ real. 

Integrating by parts, we can then rewrite $\alpha(\lambda)$ in \eqref{def:global_alpha} as
\begin{align}
    \alpha(\lambda) &= \exp\left\lbrace \frac{1}{2\pi \ii}\int^{0}_{\lambda_0} \log \left( 1+ |{\sf r}(\eta)|^2 \right) \frac{w_3(\eta-\lambda)}{\rho} d\eta \right\rbrace \nonumber \\
    &= \exp\left\lbrace \frac{1}{2\pi \ii}\int^{0}_{\lambda_0} \log \left( 1+ |{\sf r}(\eta)|^2 \right)  d \log\beta(\eta- \lambda)\right\rbrace \nonumber \\
    &=  \exp\left\lbrace -\frac{1}{2\pi \ii}\left(\log \left( 1+ |{\sf r}(\lambda_0)|^2 \right)   \log\beta(\lambda_0-\lambda) \right)\right\rbrace \\
    &\hspace*{5cm}\times\exp\left\lbrace -\frac{1}{2\pi \ii}\int^{0}_{\lambda_0} d\left(\log \left( 1+ |{\sf r}(\eta)|^2 \right) \right)  \log\beta(\eta- \lambda)\right\rbrace \nonumber \\
     &= \left( \beta(\lambda_0 - \lambda)  \right)^{\ii\nu} 
    \exp\left\lbrace -\frac{1}{2\pi \ii}\int^{0}_{\lambda_0} d\left(\log \left( 1+ |{\sf r}(\eta)|^2 \right) \right)  \log\beta(\eta- \lambda)\right\rbrace \nonumber \\
    &=:\left( \beta(\lambda_0 - \lambda)  \right)^{\ii\nu}  e^{\ii c_0(\lambda)\sigma_3}, \label{def:c0lambda}
\end{align}
with the parameter
\begin{gather}\label{def:nu}
    \nu:=\frac{1}{2\pi }\log(1+|{\sf r}(\lambda_0)|^2),
\end{gather}
and
\begin{align}\label{def:c_0}
        c_0(\lambda):=\frac{1}{2\pi }\int^{0}_{\lambda_0} d\left(\log \left( 1+ |{\sf r}(\eta)|^2 \right) \right)  \log\beta(\eta- \lambda).
\end{align}
In the computations above we used the fact that ${\sf r}(0)=0$.

Finally, with the expression above, we have
\begin{align}\label{Tgl}
    T^{(gl)}(\lambda)= \left( \beta(\lambda_0 - \lambda)  \right)^{\ii\nu\sigma_3}  \ee^{\ii c_0(\lambda)\sigma_3}.
\end{align}
This alternative formula is convenient for studying the global parametrix near $\lambda=\lambda_0$. We now proceed to obtain the local parametrix, which in the proper asymptotic regime, should coincide with $T^{(gl)}(\lambda)$.
\subsection{Local parametrix}\label{subsec:local}
Before defining the local parametrix, we analyse the lenses contours in the neighbourhood of $\lambda_0$. In the vicinity of the stationary point $\lambda_0$, the characteristic exponential function behaves as
\begin{align}
    \ee^{-\ii t \sigma_{3} p(\lambda,\varkappa)} \sim \ee^{-\ii t \sigma_{3} \left( p(\lambda_0,\varkappa)+\frac{1}{2} \partial_\lambda^2p(\lambda_0,\varkappa)(\lambda - \lambda_0)^2\right) }, && \arg (\partial_\lambda^2p(\lambda_0,\varkappa)) = \pi,
\end{align}
where
 \begin{equation}
     \partial_\lambda^2 p(\lambda,\varkappa) =-\frac{1}{\rho^2}\left( 8 w_1(\lambda) w_2(\lambda) w_3^2(\lambda)+(w_1^2(\lambda)+w_2^2(\lambda))(2w_1(\lambda)w_2(\lambda)-\varkappa w_3(\lambda) \right). 
\label{eq:ppprimelambda}
    \end{equation}
The fact that $\partial_\lambda^2p(\lambda_0,\varkappa)<0$ can be shown in the following way. First, we eliminate parameter $\varkappa$ in the expression \eqref{eq:ppprimelambda} using \eqref{eq:pprimelambda}. Then rewrite it using \eqref{def:elliptic_curve} in terms of $w_1(\lambda_0)$. The result is 
\begin{align}
\partial_\lambda^2p(\lambda_0,\varkappa)=-\frac{1}{\rho^2}\left(4w_1^6(\lambda_0)-6k^2\rho^2 w_1^4(\lambda_0)+2k^4\rho^6. \right)  
\end{align}
This expression has maximum $-\frac{2}{\rho}k^4(1-k^2)\rho^6$ at $w_1^2(\lambda_0)=k^2\rho^2$, which confirms that $\partial_\lambda^2p(\lambda_0,\varkappa)<0$. That motivates us to introduce notation for $-\partial_\lambda^2p(\lambda_0,\varkappa)$
\begin{equation}\label{def:phi_0}
    \varphi_0=\frac{1}{\rho^2}\left( 8 w_1(\lambda_0) w_2(\lambda_0) w_3^2(\lambda_0)+(w_1^2(\lambda_0)+w_2^2(\lambda_0))(2w_1(\lambda_0)w_2(\lambda_0)-\varkappa w_3(\lambda_0)) \right)
\end{equation}
A suitable conformal transformation to obtain the local parametrix is then defined by
\begin{equation}\label{def:xi}
   \xi(\lambda) :=\sqrt{4\ii t(p(\lambda_0,\varkappa)-p(\lambda,\varkappa))}\simeq \ee^{\ii \pi/4}\sqrt{2t\varphi_0} (\lambda-\lambda_0) = \gamma\, t^{1/2}(\lambda-\lambda_0),
\end{equation}
with the parameter
\begin{align}\label{def:gamma}
    \gamma:= \ee^{\ii \pi/4}\sqrt{2\varphi_0} . 
\end{align}

Note that, although the local parametrix will be initially defined on the subdomain $\mathbb{T}(\lambda_0)$, the symmetric extension to the torus $\mathbb{T}$ with the identification \eqref{domsubdom:T} is crucial to obtain our final result. To this end, we note that the properties of $p(\lambda,\varkappa)$ in \eqref{symm:p} imply that 
$$
\xi(\lambda+2K)=\xi(\lambda),
$$
and we need to introduce 
\begin{align}\label{def:xitil}
   \widetilde{\xi}(\lambda) :=\xi(\lambda+2\ii K').
\end{align}
We will use $\xi(\lambda)$ in the neighborhood of $\lambda_0$, $\lambda_0+2K$ and $\widetilde{\xi}(\lambda)$ in the neighborhood of $\lambda_0+2\ii K'$, $\lambda_0+2K+2\ii K'$.

We would like to take the jump of the local parametrix near $\lambda_0$ as shown on Figure \ref{fig:contour_local}. Here $\sfr_0=\sfr(\lambda_0)$.
\begin{figure}[h!]
        \centering
\includegraphics[trim={0cm 0.5cm 0cm, 0.5cm},clip,width=1.0\columnwidth]{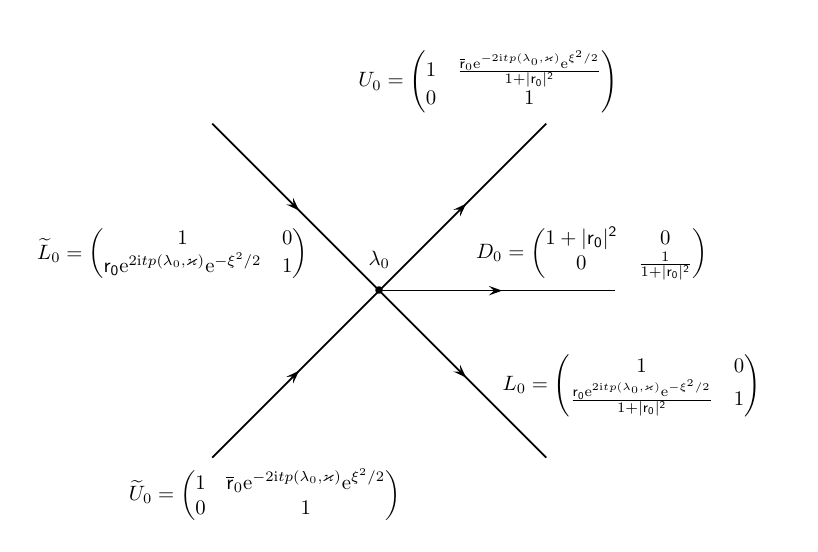}
    \caption{Contour $\Sigma^{(loc)}(\lambda_0)$ and the respective jumps in the disc $\mathbb{D}(\lambda_0)$.}
\label{fig:contour_local}
    \end{figure} 
    In order to provide matching with the jump of $T(\lambda,x,t)$ we need to take discs of local parametrix of shrinking radius. Denote the disc near point $\lambda_0$ described by the condition $|\lambda-\lambda_0|\leq t^{-\frac{1}{2}+\varepsilon}$ with $0<\varepsilon<\frac{1}{2}$ as $\mathbb{D}(\lambda_0)$. We will introduce local parametrix on the collection of its shifts
\begin{align}\label{disc_local}
    \mathbb{D}= \mathbb{D}(\lambda_0)+ \mathbb{D}(\lambda_0+2K)+  \mathbb{D}(\lambda_0+2\ii K')+  \mathbb{D}(\lambda_0+2K+2\ii K').
\end{align}
We also introduce the notation $\Sigma^{(loc)}(\lambda_0)$ for the jump  contour  for the local parametrix in the disc $\mathbb{D}(\lambda_0)$.   
Now using the shift we can introduce the jump contour for the local parametrix
\begin{align}\label{domsubdom:Sigma_local}
    \Sigma^{(loc)}= \Sigma^{(loc)}(\lambda_0)+\Sigma^{(loc)}(\lambda_0+2K)+ \left(\Sigma^{(loc)}(\lambda_0+2\ii K')\right)^{-1}+ \left(\Sigma^{(loc)}(\lambda_0+2K+2\ii K')\right)^{-1}.
\end{align}
As the result on the $\lambda$-plane, the {\it local} parametrix $T^{(loc)}(\lambda,x,t)$, solves the following RHP.
\begin{rhp}\label{rhp:localpar}
\begin{enumerate}
    \item The function $T^{(loc)}(\lambda,x,t)$ is analytic on $\mathbb{D}\backslash \Sigma^{(loc)}$, with the contour $\Sigma^{(loc)}$ given by \eqref{domsubdom:Sigma_local}, the collection of discs $\mathbb{D}$ is given by \eqref{disc_local} and contour $\Sigma^{(loc)}(\lambda_0)$ is shown on Figure \ref{fig:contour_local}.

    \item For $\lambda \in \Sigma^{(loc)}$, 
    \begin{align}
        T^{(loc)}_{+}(\lambda,x,t) = T^{(loc)}_{-}(\lambda,x,t) G_T^{(loc)}(\lambda,x,t),
    \end{align}
    with the jump $G_T^{(loc)}(\lambda,x,t)$ piecewise defined on each ray of the contour $\Sigma^{(loc)}$ as shown on Figure \ref{fig:contour_local}.
  
   \item The symmetry conditions hold \begin{align}\label{rhp:local_symmetry}
 \sigma_{3} T^{(loc)}(\lambda + 2K,x,t) \sigma_{3} = T^{(loc)}(\lambda,x,t), &&
 \sigma_{1} T^{(loc)}(\lambda + 2\ii K',x,t) \sigma_{1} = T^{(loc)}(\lambda,x,t).
\end{align}
\item The normalization $\det(T^{(loc)}(\lambda,x,t))=1$ holds.
    \item In asymptotic regime $\lambda\in \partial\mathbb{D}$, $t\to\infty$ we have
    \begin{align}
        T^{(loc)}(\lambda,x,t)= T^{(gl)}(\lambda) \left(\mathbb{1}+\mathcal{O}(t^{-\frac{1}{2}}) \right) = \alpha(\lambda)^{\sigma_3}\left(\mathbb{1}+\mathcal{O}(t^{-\frac{1}{2}}) \right).
    \end{align}
\end{enumerate}    
\end{rhp}

Solution of the RHP \ref{rhp:localpar} can be expressed in terms of parabolic cylinder functions parametrix presented in appendix \ref{sec:parabolic}. More precisely we compare RHP \ref{rhp:localpar} with RHP \ref{rhp:parabolic}. Comparing the jumps we see that the parameters $a, b$ should be chosen as
\begin{align}
-\frac{\ii\sqrt{2\pi}\ee^{-2\pi \nu}}{a \Gamma(-\ii\nu)}:=-{\sf r}_0\ee^{2\ii tp(\lambda_0,\varkappa)},&& -\frac{\sqrt{2\pi}\ee^{3\pi  \nu}}{b \Gamma(\ii\nu)}:=-\overline{{\sf r}}_0\ee^{-2\ii tp(\lambda_0,\varkappa)}, && ab=\ii\nu. \label{def:abnu}
\end{align}

The local parametrix $T^{(loc)}(\lambda,x,t)$ near $\lambda_0$ can then be written in terms of parabolic cylinders through the matrix $\mathbf{D}_{a,b}(\xi)$ defined in \eqref{eq:parabolic_cylinder_parametrix} as
\begin{equation}\label{Pl}
T^{(loc)}(\lambda)=\alpha(\lambda)^{\sigma_3}\left( \xi (\lambda)  \right)^{-\ii\nu \sigma_3}  \mathbf{D}_{a,b}(\xi) \ee^{-\frac{\xi^2}{4}\sigma_3},\quad \lambda\in\mathbb{D}(\lambda_0),\quad \frac{\pi}{4}\leq\arg(\xi(\lambda))\leq \frac{9\pi}{4}
\end{equation}
Using the behaviour of $p(\lambda,\varkappa)$ in \eqref{symm:p} along with the transformation $ \widetilde{\xi}(\lambda)$ defined in \eqref{def:xitil}, the local parametrix is determined to be

\begin{align}\label{def:PDab}
    T^{(loc)}(\lambda,x,t)&:= \begin{cases}
        \alpha(\lambda)^{\sigma_3}\left( \xi (\lambda)  \right)^{-\ii\nu \sigma_3}  \mathbf{D}_{a,b}(\xi) \ee^{-\frac{\xi^2}{4}\sigma_3},\quad \lambda\in\mathbb{D}(\lambda_0),\\
        \sigma_3 \alpha(\lambda)^{\sigma_3}\left( \xi (\lambda)  \right)^{-\ii\nu \sigma_3}  \mathbf{D}_{a,b}(\xi) \ee^{-\frac{\xi^2}{4}\sigma_3}\sigma_3,\quad \lambda\in\mathbb{D}(\lambda_0+2K),\\
         \sigma_1 \alpha(\lambda+2\ii K')^{\sigma_3}\left( \widetilde{\xi}(\lambda)  \right)^{-\ii\nu \sigma_3}  \mathbf{D}_{a,b}(\widetilde{\xi}) \ee^{-\frac{\widetilde{\xi}^2}{4}\sigma_3}\sigma_1, \quad \lambda\in\mathbb{D}(\lambda_0+2\ii K'),\\
      \sigma_2   \alpha(\lambda+2K+2\ii K')^{\sigma_3}\left( \widetilde{\xi}(\lambda)  \right)^{-\ii\nu \sigma_3}  \mathbf{D}_{a,b}(\widetilde{\xi}) \ee^{-\frac{\widetilde{\xi}^2}{4}\sigma_3}\sigma_2,\\\lambda\in\mathbb{D}(\lambda_0+2K+2\ii K').
    \end{cases}
\end{align}
The final step in the nonlinear steepest descent analysis is the use of the small norm theorem. 

\subsection{Small norm problem}\label{subsec:small_norm}
 We denote the circle of the boundary of disc $\mathbb{D}(\lambda_0)$ by $\mathcal{C}(\lambda_0)$, and the union of such circles in each of the subdomains of the torus \eqref{domsubdom:T} is denoted by
\begin{align}
    \mathcal{C} = \mathcal{C}(\lambda_0)+\mathcal{C}(\lambda_0+2K)+ \mathcal{C}(\lambda_0+2\ii K') + \mathcal{C}(\lambda_0+2K+2\ii K').
\end{align}
  
        \begin{figure}[H]
        \centering
        \includegraphics[trim={0.3cm 0.3cm 0.2cm 0cm},clip, width=0.8\columnwidth]{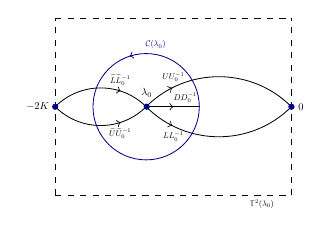}
        \caption{Jump contour $\Sigma^{(R)}(\lambda_0)$. It extends to the domain $\mathbb{T}(\lambda_0)$.}
        \label{fig:contour_boundary}
    \end{figure}
We define the ratio of the solution of the RHP \ref{rhp:3} and the local, global parametrices in \eqref{Pl}, \eqref{Tgl} respectively as
 \begin{align}\label{transf:2}
    R(\lambda,x,t):= \begin{cases}
        T(\lambda,x,t) \left(T^{(gl)}(\lambda)\right)^{-1} , \qquad \lambda\in\mathbb{T}\setminus\mathbb{D},\\
        T(\lambda,x,t) \left(T^{(loc)} (\lambda,x,t)\right)^{-1},\qquad \lambda\in\mathbb{D}.
    \end{cases}
\end{align}
It has the jump in the contour  $\Sigma^{(R)}(\lambda_0)$ shown on Figure \ref{fig:contour_boundary} in $\mathbb{T}(\lambda_0)$. We also extend it to $\mathbb{T}$.

\begin{align}
    \Sigma^{(R)} =\Sigma^{(R)}(\lambda_0)+\Sigma^{(R)}(\lambda_0+2K)+ \Sigma^{(R)}(\lambda_0+2\ii K')^{-1} + \Sigma^{(R)}(\lambda_0+2K+2\ii K')^{-1}.
\end{align}
\begin{rhp}\label{rhp:R}
    \begin{enumerate}
        \item The function $R(\lambda,x,t)$ is piecewise analytic on $\mathbb{T}^2\backslash\Sigma^{(R)}$.

\item On the circles $\mathcal{C}$, it has the jump
\begin{equation}\label{def:G_R}
    R_{+}(\lambda,x,t) = R_{-}(\lambda,x,t) G_{R}(\lambda,x,t), \qquad G_{R}(\lambda,x,t) = T^{(gl)}(\lambda) \left( T^{(loc)}(\lambda,x,t)\right)^{-1}.
\end{equation}
 Inside the circles $\mathcal{C}$, it has a piecewise jump on each of the rays. It is given by the expression indicated on the Figure \ref{fig:contour_boundary} and conjugated by $T^{(loc)}_-(\lambda,x,t)$.
 
\noindent Outside the circles $\mathcal{C}$, it has a piecewise jump 
 $$G_R(\lambda,x,t)=T^{(gl)}(\lambda)G_T(\lambda,x,t)\left(T^{(gl)}(\lambda)\right)^{-1}.$$

   \item The symmetry conditions hold \begin{align}\label{rhp:small_symmetry}
 \sigma_{3} R(\lambda + 2K,x,t) \sigma_{3} = R(\lambda,x,t), &&
 \sigma_{1} R(\lambda + 2\ii K',x,t) \sigma_{1} = T^{(loc)}(\lambda,x,t).
\end{align}
\item The normalization $\det(R(\lambda,x,t))=1$ holds.
    \end{enumerate}
\end{rhp}

   We can observe that $G_R(\lambda,x,t)$ satisfies
    \begin{equation}\label{symm:GR}
        G_{R}(\lambda+2K,x,t) = \sigma_3 G_R(\lambda,x,t) \sigma_3,\qquad G_{R}(\lambda+2\ii K',x,t) = \sigma_1 G_R(\lambda,x,t) \sigma_1.
    \end{equation}

An important step for our final analysis is the $t \to\infty$ behavior of the jump $G_R(\lambda,x,t)$.  which can be explicitly determined by the asymptotics of the matrix of parabolic cylinder functions \eqref{asymp:Dab} on the circles $\mathcal{C}$, and we have the following statement.
\begin{proposition}
 In the limit $t\to \infty$, the jump $G_R(\lambda,x,t)$ for $\lambda\in \mathcal{C}(\lambda_0)$ behaves as
    \begin{align}
    G_R(\lambda,x,t)
    =\mathbb{1} + \frac{1}{\xi(\lambda)}\left(\begin{array}{cc}
       0  & -a  \widetilde{\gamma}\ee^{2\ii c_0(\lambda)} \\
       b \widetilde{\gamma}^{-1}\ee^{-2\ii c_0(\lambda)}  & 0
    \end{array} \right)+ \mathcal{O}(\xi(\lambda)^{-2}), \\  \widetilde{\gamma} := \left(\frac{\gamma^2 t}{\beta_0^2}\right)^{-\ii\nu}\ee^{2\pi\nu}, \label{asymp:GR}
\end{align}
with
 $c_0(\lambda), \beta_0, \xi(\lambda), \gamma, a, b$ defined in \eqref{def:beta0}, \eqref{def:c_0}, \eqref{def:xi}, \eqref{def:gamma},  \eqref{def:abnu} respectively.
 Moreover let function $F(\lambda)$ be analytic in the neighborhood of $\lambda_0$. In the asymptotics $t\to \infty$, the integral of $(G_R(\lambda,x,t)-\mathbb{1})F(\lambda)$ over $\mathcal{C}(\lambda_0)$ behaves as
    \begin{align}
    \ointctrclockwise_{\mathcal{C}(\lambda_0)} (G_R(\lambda,x,t)-\mathbb{1})&F(\lambda)\frac{d\lambda}{2\pi \ii}   
    =\mathbb{1} + \frac{1}{\gamma t^{\frac{1}{2}}}
    \left(\begin{array}{cc}
       0  & -a  \widetilde{\gamma}\ee^{2\ii \sfc_0} \\
       b \widetilde{\gamma}^{-1}\ee^{-2\ii \sfc_0}  & 0
    \end{array} \right)
    F(\lambda_0)+ \mathcal{O}(t^{-1}),\\&\qquad \sfc_0=c_0(\lambda_0).\label{asymp:GRint}
\end{align}
\end{proposition}
\begin{proof}
The asymptotic behaviour of $T^{(loc)}(\lambda,x,t)$ for $t\to \infty$, or equivalently $\xi \to \infty$, is obtained from \eqref{asymp:Dab}, \eqref{Pl}:
\begin{align}
    T^{(loc)}(\lambda,x,t)\mathop{=} (\alpha(\lambda))^{\sigma_3} \xi^{-\ii\nu\sigma_3}\left( \mathbb{1} + \frac{1}{\xi}\begin{pmatrix}
        0&-a\\b&0
    \end{pmatrix}+ \mathcal{O}(\xi^{-2})\right)\xi^{\ii\nu\sigma_3}.
\end{align}
The values of $a, b$ are given in \eqref{def:abnu}.
Remembering the definition \eqref{Tgl} of global parametrix $T^{{(gl)}} (\lambda)$ we can see that for \\$|\lambda-\lambda_0|=\mathcal{O}(t^{-\frac{1}{2}+\varepsilon})$
\begin{align}
    T^{(gl)}(\lambda)&\mathop{=} \left( \lambda_0 - \lambda  \right)^{\ii\nu\sigma_3} \beta_0^{\ii\nu\sigma_3} \ee^{\ii c_0(\lambda)\sigma_3} \left(\mathbb{1} +\mathcal{O}(t^{-\frac{1}{2}+\varepsilon})\right),\quad -\pi\leq\arg(\lambda_0-\lambda)\leq\pi\\
    &=\left( \lambda - \lambda_0  \right)^{\ii\nu\sigma_3} \ee^{\pi\nu\sigma_3}\beta_0^{\ii\nu\sigma_3} \ee^{\ii c_0(\lambda)\sigma_3}\left(\mathbb{1} +\mathcal{O}(t^{-\frac{1}{2}+\varepsilon})\right),\quad 0\leq\arg(\lambda-\lambda_0)\leq 2\pi\\
    &=\xi^{\ii\nu\sigma_3}\left(\frac{\gamma t^{1/2}}{\beta_0}\right)^{-\ii\nu\sigma_3}\ee^{\pi\nu\sigma_3}\ee^{\ii c_0(\lambda)\sigma_3}\left(\mathbb{1} +\mathcal{O}(t^{-\frac{1}{2}+\varepsilon})\right).
\end{align}
We deduce that the jump \eqref{def:G_R} then has the following behaviour in the same limit
\begin{align}
    G_R(\lambda,x,t)&= T^{{(gl)}} (\lambda) \xi^{-\ii\nu\sigma_3}\left(\mathbb{1} - \frac{m_1}{\xi} + \mathcal{O}(\xi^{-2}) \right) \xi^{\ii\nu\sigma_3}\left(T^{{(gl)}} (\lambda)\right)^{-1}\nonumber\\
    &=\left(\begin{array}{cc}
       1  & -\frac{a \ee^{2\ii c_0(\lambda)+2\pi\nu} }{\xi} \left(\frac{\gamma^{2} t}{\beta_0^2}\right)^{-\ii\nu} \\
      \frac{ b \ee^{-2\ii c_0(\lambda)-2\pi\nu}}{\xi} \left(\frac{\gamma^{2} t}{\beta_0^2}\right)^{\ii\nu}  & 1
    \end{array} \right)+ \mathcal{O}(\xi^{-2}),
\end{align}
which, with the definition of $\widetilde{\gamma}$ gives \eqref{asymp:GR}. Evaluating the integral using residues we get \eqref{asymp:GRint}.
\end{proof}
Note that the off-diagonal elements in \eqref{asymp:GR} depend on $t$ as $t^{-\varepsilon \pm \ii\nu/2}$, but $\nu$ is real, so we can say
\begin{align}
G_R(\lambda,x,t)-\mathbb{1} = \mathcal{O}(t^{-\varepsilon}),\quad \lambda\in \mathcal{C}. 
\end{align}
Moreover for $\lambda\in\mathbb{D}\setminus (\Gamma_1\cup\Gamma_2)$ we need to have jump terms $\sfr(\lambda)\ee^{2\ii tp(\lambda,\varkappa)}-\sfr_0 \ee^{2\ii tp(\lambda_0,\varkappa)}\ee^{-\frac{\xi^2}{2}}$ and $\overline{\sfr(\lambda)}\ee^{-2\ii tp(\lambda,\varkappa)}- \overline{\sfr}_0 \ee^{-2\ii tp(\lambda_0,\varkappa)}\ee^{\frac{\xi^2}{2}}$ to be small. With the expression for $\xi(\lambda)$ in \eqref{def:xi}, we see that we want to minimize terms of type $|\lambda-\lambda_0|\ee^{-|\gamma|^2t|\lambda-\lambda_0|^2}$. The maximal value for this expression is achieved for $|\lambda-\lambda_0|\simeq \frac{1}{|\gamma|\sqrt{2t}}$. We use this value to estimate the jump and arrive at 
\begin{align}
G_R(\lambda,x,t)-\mathbb{1} = \mathcal{O}(t^{-\frac{1}{2}}),\quad \lambda\in (\Sigma^{(R)}\cap\mathbb{D})\setminus (\Gamma_1\cup\Gamma_2). 
\end{align}
On the $\mathbb{D}\cap(\Gamma_1\cup\Gamma_2)$ the estimate for the jump is given by $\mathcal{O}(\lambda-\lambda_0)$, so we get
\begin{align}
G_R(\lambda,x,t)-\mathbb{1} = \mathcal{O}(t^{-\frac{1}{2}+\varepsilon}),\quad \lambda\in \Sigma^{(R)}\cap\mathbb{D}\cap (\Gamma_1\cup\Gamma_2). 
\end{align}
Outside of discs $\mathbb{D}$ the sign of imaginary part of $p(\lambda,\varkappa)$ guarantees the exponential decay
\begin{align}
G_R(\lambda,x,t)-\mathbb{1} = \mathcal{O}(\ee^{-|\gamma|^2t^{2\varepsilon}}),\quad \lambda\in \Sigma^{(R)}\setminus\mathbb{D}. 
\end{align}
Keeping in mind the vanishing properties of reflection coefficient we get the estimate
\begin{align} \label{eq:GR_uniform_estimate}
G_R(\lambda,x,t)-\mathbb{1} = \mathcal{O}\left(\frac{t^{-\varepsilon_1}}{1+|w_3(\lambda)|}\right),\quad\varepsilon_1=\min\left(\varepsilon,\frac{1}{2}-\varepsilon\right),\quad \lambda\in \Sigma^{(R)}. 
\end{align}
Using the fact that domains $\mathbb{D}$ are shrinking as $t\to \infty$ we can get other useful estimates
\begin{align} \label{eq:GR_L2_estimate}
\|(G_R(\lambda,x,t)-\mathbb{1})C(\lambda,0)\|_{L_2(\Sigma^{(R)})} = \mathcal{O}\left(t^{-\varepsilon_2}\right),\quad\varepsilon_2=\min\left(\frac{1}{4}+\frac{\varepsilon}{2},\frac{3}{4}-\frac{3\varepsilon}{2}\right). 
\end{align}
\begin{align} \label{eq:GR_L1_estimate}
\|(G_R(\lambda,x,t)-\mathbb{1})C(\lambda,0)\|_{L_1(\Sigma^{(R)}\setminus \mathcal{C})} = \mathcal{O}\left(t^{-1+2\varepsilon}\right). 
\end{align}
\section{Proof of Theorem \ref{thm:asymp_LL}}\label{Sec:asymp}
Let us start by recollecting the relation between the solution of the Landau-Lifshitz equation and the solution of the RHP \ref{rhp2} given by \eqref{sol_LL_RHP_init_1}:
\begin{align}\label{sol_LL_RHP_init_2}
    \Psi_1(x,t) \sigma_3 \left(\Psi_1(x,t)\right)^{-1} = \sum_{j=1}^3 L_{j}(x,t) \sigma_{j} =:\mathbf{L}(x,t), 
\end{align}
where $\Psi_1(x,t)$ is the leading behaviour of the solution of the RHP \ref{rhp2} at $\lambda\to 0$, see \eqref{def:asympY}.
 Tracing the transformations so far we have
\begin{align}
Y(\lambda,x,t)&\mathop{=}^{\eqref{transf:YT}}T(\lambda,x,t) M(\lambda,x,t)
    \mathop{=}^{\eqref{transf:2}} R(\lambda,x,t) T^{(gl)}(\lambda) M(\lambda,x,t) ,\label{eq:Psiexp}
\end{align}
and in the vicinity of $\lambda\to 0$ due to \eqref{asymp:M}, \eqref{def:Tgl_alpha} and \eqref{eq:phi1} 
\begin{align}\label{eq:Psi1R0}
  \Psi_1(x,t)= R(0,x,t) T^{(gl)}(0)=R(0,x,t)\alpha(0)^{\sigma_3}.
\end{align}
To determine the above quantity, we compute $R(0,x,t)$ using the singular integral equations following section \ref{sec:singular_integral_equation}
\begin{align}\label{def:chi-}
    \chi^{(R)}(\lambda,x,t) = \mathbb{1} + \frac{1}{2\pi \ii}\int_{\Sigma^{(R)}} \chi^{(R)}(\mu,x,t) \left( G_{R}(\mu,x,t)- \mathbb{1}\right) C(\mu, \lambda-i0) d\mu,
\end{align}
and
\begin{align}
    \Phi^{(R)}(\lambda,x,t) = \mathbb{1} + \frac{1}{2\pi \ii}\int_{\Sigma^{(R)}}  \chi^{(R)}(\mu,x,t) \left( G_{R}(\mu,x,t)- \mathbb{1}\right) C(\mu, \lambda) d\mu.
\end{align}
According to \cite{rodin1989} normalized version of RHP \ref{rhp2} is solvable. Therefore the solution of normalized version of Riemann-Hilbert problem for $R(\lambda,x,t)$ is also solvable and solution is given by $ \Phi^{(R)}(\lambda,x,t)$. 
The estimate \eqref{eq:GR_uniform_estimate} implies that we can iterate equation \eqref{def:chi-}. Using \eqref{eq:GR_L2_estimate} we get
 \begin{equation}\label{eq:chiR_estimate}\|\chi^{(R)}(\lambda,x,t)-\mathbb{1}\|_{L_2(\Sigma^{(R)})}=\mathcal{O}(t^{-\varepsilon_2}).\end{equation}

The symmetrization of $\Phi^{(R)}(\lambda,x,t)$ determines $R(\lambda,x,t)$ 
\begin{align}
    R(\lambda) = \frac{1}{\sqrt{c}} \left(\Phi^{(R)}(\lambda) + \sigma_1 \Phi^{(R)}(\lambda+ 2\ii K') \sigma_1 + \sigma_2 \Phi^{(R)}(\lambda+ 2K+2\ii K') \sigma_2 \right.\\\left.+  \sigma_3 \Phi^{(R)}(\lambda+ 2K) \sigma_3\right),\\
    c=\det\left(\Phi^{(R)}(\lambda) + \sigma_1 \Phi^{(R)}(\lambda+ 2\ii K') \sigma_1 +\sigma_2 \Phi^{(R)}(\lambda+ 2K+2\ii K') \sigma_2 \right.\\\left. + \sigma_3 \Phi^{(R)}(\lambda+ 2K) \sigma_3\right).
\end{align}
Our final goal is to obtain the $t\to +\infty$ asymptotics of the solution of the LL equation through the identities \eqref{sol_LL_RHP_init_2}, \eqref{eq:Psi1R0}. 
To this end, the term $R(0)$ is given by
\begin{align}
    R(0)\sqrt{c} &=  \left(\Phi^{(R)}(0) + \sigma_1 \Phi^{(R)}( 2\ii K') \sigma_1 +\sigma_3 \Phi^{(R)}(2K) \sigma_3+ \sigma_2 \Phi^{(R)}( 2K+2\ii K') \sigma_2  \right)\nonumber \\
    & =  4 \mathbb{1} + \frac{1}{2\pi \ii}\int_{\Sigma^{(R)}}  \chi^{(R)}(\mu,x,t) \left( G_{R}(\mu,x,t)- \mathbb{1}\right) C(\mu,0) d\mu\nonumber\\
    &+  \sigma_3\frac{1}{2\pi \ii}\int_{\Sigma^{(R)}}  \chi^{(R)}(\mu,x,t) \left( G_{R}(\mu,x,t)- \mathbb{1}\right) C(\mu, 2K)\sigma_3 d\mu \nonumber \\
    &+  \sigma_1\frac{1}{2\pi \ii}\int_{\Sigma^{(R)}}  \chi^{(R)}(\mu,x,t) \left( G_{R}(\mu,x,t)- \mathbb{1}\right) C(\mu, 2\ii K')\sigma_1 d\mu  \nonumber\\
    &+\sigma_2\frac{1}{2\pi \ii}\int_{\Sigma^{(R)}}  \chi^{(R)}(\mu,x,t) \left( G_{R}(\mu,x,t)- \mathbb{1}\right) C(\mu, 2K+2\ii K')\sigma_2 d\mu.\label{eq:r01}
\end{align}
 So, for $t\to \infty$, using \eqref{eq:GR_L2_estimate}, \eqref{eq:chiR_estimate} we get
\begin{align}
   R(0,x,t)\sqrt{c}&=  4 \mathbb{1} + \frac{1}{2\pi \ii}\int_{\Sigma^{(R)}}   \left( G_{R}(\mu)- \mathbb{1}\right) C(\mu,0) d\mu\\&+  \sigma_3\frac{1}{2\pi \ii}\int_{\Sigma^{(R)}}  \left( G_{R}(\mu)- \mathbb{1}\right) C(\mu, 2K)\sigma_3 d\mu \nonumber \\
    &+  \sigma_1\frac{1}{2\pi \ii}\int_{\Sigma^{(R)}}   \left( G_{R}(\mu)- \mathbb{1}\right) C(\mu, 2\ii K')\sigma_1 d\mu \nonumber\\
    &+\sigma_2\frac{1}{2\pi \ii}\int_{\Sigma^{(R)}}  \left( G_{R}(\mu)- \mathbb{1}\right) C(\mu, 2K+2\ii K')\sigma_2 d\mu + \mathcal{O}(t^{-2\varepsilon_2}).
\end{align}
On the next step we replace contour $\Sigma^{(R)}$ with the collection of circles $\mathcal{C}$. We use estimate \eqref{eq:GR_L1_estimate} and get
\begin{align}
   R(0,x,t)\sqrt{c} &=  4 \mathbb{1} + \frac{1}{2\pi \ii}\int_{\mathcal{C}}  \left( G_{R}(\mu)- \mathbb{1}\right) C(\mu,0) d\mu+  \sigma_3\frac{1}{2\pi \ii}\int_{\mathcal{C}} \left( G_{R}(\mu)- \mathbb{1}\right) C(\mu, 2K)\sigma_3 d\mu \nonumber \\
    &+  \sigma_1\frac{1}{2\pi \ii}\int_{\mathcal{C}}  \left( G_{R}(\mu)- \mathbb{1}\right) C(\mu, 2\ii K')\sigma_1 d\mu \nonumber \\
    &+\sigma_2\frac{1}{2\pi \ii}\int_{\mathcal{C}} \left( G_{R}(\mu)- \mathbb{1}\right) C(\mu, 2K+2\ii K')\sigma_2 d\mu + \mathcal{O}(t^{-\varepsilon_1-\varepsilon_2})+\mathcal{O}(t^{-1+2\varepsilon}) \nonumber\\
    &=:  4 \mathbb{1}+I_1+I_2+I_3+I_4+ \mathcal{O}(t^{-\frac{2}{3}}).
\end{align}
On the last step we took $\varepsilon=\frac{1}{6}$, which corresponds to $\mathop{\max}\limits_{0<\varepsilon<\frac{1}{2}}\min(1-2\varepsilon,2\varepsilon_2)=\frac{2}{3}$ and optimizes the error term. 
Let us expand the individual integrals above:
\begin{align}
    &2\pi \ii I_1:= \int_{\mathcal{C}}  \left( G_{R}(\mu)- \mathbb{1}\right) C(\mu,0) d\mu\nonumber\\
    &=  \ointctrclockwise_{\lambda_0}  \left( G_{R}(\mu)- \mathbb{1}\right) C(\mu,0) d\mu+\ointctrclockwise_{\lambda_0+2K}  \left( G_{R}(\mu)- \mathbb{1}\right) C(\mu,0) d\mu \nonumber\\
    &+\ointctrclockwise_{\lambda_0+2\ii K'}  \left( G_{R}(\mu)- \mathbb{1}\right) C(\mu,0) d\mu +\ointctrclockwise_{\lambda_0+2K+2\ii K'}  \left( G_{R}(\mu)- \mathbb{1}\right) C(\mu,0) d\mu \nonumber   \\
    &= \ointctrclockwise_{\lambda_0}  \left( G_{R}(\mu)- \mathbb{1}\right) C(\mu,0) d\mu+\ointctrclockwise_{\lambda_0}  \left( G_{R}(\mu+2K)- \mathbb{1}\right) C(\mu+2K,0) d\mu \nonumber\\
    &+\ointctrclockwise_{\lambda_0}  \left( G_{R}(\mu+2\ii K')- \mathbb{1}\right) C(\mu+2\ii K',0) d\mu \nonumber \\
    &+\ointctrclockwise_{\lambda_0}  \left( G_{R}(\mu+2K+2\ii K')- \mathbb{1}\right) C(\mu+2K+2\ii K',0) d\mu\nonumber\\
    &\mathop{=}^{\eqref{symm:GR}} \ointctrclockwise_{\lambda_0}  \left( G_{R}(\mu)- \mathbb{1}\right) C(\mu,0) d\mu+\ointctrclockwise_{\lambda_0} \sigma_3 \left( G_{R}(\mu)- \mathbb{1}\right)\sigma_3 C(\mu+2K,0) d\mu \nonumber\\
    &+\ointctrclockwise_{\lambda_0} \sigma_1 \left( G_{R}(\mu)- \mathbb{1}\right)\sigma_1 C(\mu+2\ii K',0) d\mu \nonumber\\
    &+\ointctrclockwise_{\lambda_0}\sigma_2  \left( G_{R}(\mu)- \mathbb{1}\right) \sigma_2 C(\mu+2K+2\ii K',0) d\mu. \label{I1}
\end{align}
Remaining integrals are computed in a similar fashion:
\begin{align}
    2\pi \ii I_2&:= \int_{\mathcal{C}} \sigma_3 \left( G_{R}(\mu)- \mathbb{1}\right) \sigma_3 C(\mu,2K) d\mu\nonumber\\
    &\mathop{=}^{\eqref{symm:GR}} \ointctrclockwise_{\lambda_0} \sigma_3 \left( G_{R}(\mu)- \mathbb{1}\right) \sigma_3 C(\mu,2K) d\mu+\ointctrclockwise_{\lambda_0}  \left( G_{R}(\mu)- \mathbb{1}\right) C(\mu+2K,2K) d\mu \nonumber\\
    &+\ointctrclockwise_{\lambda_0} \sigma_2 \left( G_{R}(\mu)- \mathbb{1}\right) \sigma_2C(\mu+2\ii K',2K) d\mu \nonumber\\
    &+\ointctrclockwise_{\lambda_0} \sigma_1 \left( G_{R}(\mu)- \mathbb{1}\right) \sigma_1 C(\mu+2K+2\ii K',2K) d\mu, \label{I2}
\end{align}
\begin{align}
     2\pi \ii I_3&:= \int_{\mathcal{C}} \sigma_1 \left( G_{R}(\mu)- \mathbb{1}\right) \sigma_1 C(\mu,2\ii K') d\mu\nonumber\\
     &=\ointctrclockwise_{\lambda_0} \sigma_1 \left( G_{R}(\mu)- \mathbb{1}\right) \sigma_1 C(\mu,2\ii K') d\mu+\ointctrclockwise_{\lambda_0} \sigma_2 \left( G_{R}(\mu)- \mathbb{1}\right) \sigma_2 C(\mu+2K,2\ii K') d\mu \nonumber\\
    &+\ointctrclockwise_{\lambda_0}  \left( G_{R}(\mu)- \mathbb{1}\right) C(\mu+2\ii K',2\ii K') d\mu \\
    & \hspace{5cm}+\ointctrclockwise_{\lambda_0} \sigma_3 \left( G_{R}(\mu)- \mathbb{1}\right) \sigma_3 C(\mu+2K+2\ii K',2\ii K') d\mu,\label{I3}
\end{align}
\begin{align}
    & 2\pi \ii I_4:= \int_{\mathcal{C}} \sigma_2 \left( G_{R}(\mu)- \mathbb{1}\right) \sigma_2 C(\mu,2K+2\ii K') d\mu\nonumber\\
     &=\ointctrclockwise_{\lambda_0} \sigma_2 \left( G_{R}(\mu)- \mathbb{1}\right) \sigma_2 C(\mu,2K+2\ii K') d\mu \\
     & \hspace{6cm}+\ointctrclockwise_{\lambda_0} \sigma_1 \left( G_{R}(\mu)- \mathbb{1}\right) \sigma_1 C(\mu+2K,2K+2\ii K') d\mu \nonumber\\
    &+\ointctrclockwise_{\lambda_0} \sigma_3 \left( G_{R}(\mu)- \mathbb{1}\right)\sigma_3 C(\mu+2\ii K',2K+2\ii K') d\mu \nonumber\\
    &+\ointctrclockwise_{\lambda_0}  \left( G_{R}(\mu)- \mathbb{1}\right)  C(\mu+2K+2\ii K',2K+2\ii K') d\mu.\label{I4}
\end{align}
In order to simplify the computation of $R(0,x,t)$, we group the terms in \eqref{I1}-\eqref{I4} by introducing the following integrals. Define
\begin{align}
    &I_5:=\ointctrclockwise_{\lambda_0}  \left( G_{R}(\mu)- \mathbb{1}\right) C(\mu,0) \frac{d\mu}{2\pi \ii}+\ointctrclockwise_{\lambda_0}  \left( G_{R}(\mu)- \mathbb{1}\right) C(\mu+2K,2K) \frac{d\mu}{2\pi \ii}\nonumber\\
    &+\ointctrclockwise_{\lambda_0}  \left( G_{R}(\mu)- \mathbb{1}\right) C(\mu+2\ii K',2\ii K') \frac{d\mu}{2\pi \ii} \\
    & \hspace{6cm}+\ointctrclockwise_{\lambda_0}  \left( G_{R}(\mu)- \mathbb{1}\right)  C(\mu+2K+2\ii K',2K+2\ii K') \frac{d\mu}{2\pi \ii} \nonumber\\
    &=\ointctrclockwise_{\lambda_0}  \left( G_{R}(\mu)- \mathbb{1}\right) \Big(C(\mu,0) + C(\mu+2K,2K) \nonumber\\
    &\hspace{5cm}+C(\mu+2\ii K',2\ii K')  +C(\mu+2K+2\ii K',2K+2\ii K')\Big) \frac{d\mu}{2\pi \ii} \nonumber \\
    &=\ointctrclockwise_{\lambda_0}  \left( G_{R}(\mu)- \mathbb{1}\right) f(\mu) \frac{d\mu}{2\pi \ii}, \label{Ia}
\end{align}
where 
\begin{gather}\label{def:fmu}
    f(\mu):=C(\mu,0) + C(\mu+2K,2K) +C(\mu+2\ii K',2\ii K')  +C(\mu+2K+2\ii K',2K+2\ii K').
\end{gather}
Similarly, we define three other integrals
\begin{align}
    \sigma_3 I_6 \sigma_3&:=\ointctrclockwise_{\lambda_0}  \left( G_{R}(\mu)- \mathbb{1}\right) \Big(C(\mu+2K,0) + C(\mu,2K)+ \nonumber\\
    &\hspace{4cm}+C(\mu+2K+2\ii K',2\ii K')+  C(\mu+2\ii K',2K+2\ii K') \Big)\frac{d\mu}{2\pi \ii}\nonumber\\
    &=\ointctrclockwise_{\lambda_0}  \left( G_{R}(\mu)- \mathbb{1}\right) f(\mu+2K)\frac{d\mu}{2\pi \ii},\label{Ib}
\end{align}

\begin{align}
    \sigma_1 I_7 \sigma_1&:=\ointctrclockwise_{\lambda_0}  \left( G_{R}(\mu)- \mathbb{1}\right) \Big(C(\mu+2\ii K',0) +C(\mu+2K+2\ii K',2K)\nonumber\\
    & \hspace{6cm}+ C(\mu,2\ii K') + C(\mu+2K,2K+2\ii K') \Big)\frac{d\mu}{2\pi \ii}\nonumber\\
    &=\ointctrclockwise_{\lambda_0}  \left( G_{R}(\mu)- \mathbb{1}\right) f(\mu+2\ii K')\frac{d\mu}{2\pi \ii}\label{Ic}
\end{align}

\begin{align}
    \sigma_2 I_8 \sigma_2&:=\ointctrclockwise_{\lambda_0}  \left( G_{R}(\mu)- \mathbb{1}\right) \Big(C(\mu+2K+2\ii K',0) +C(\mu+2\ii K',2K) \nonumber\\
    & \hspace{6cm}+ C(\mu+2K,2\ii K') +C(\mu,2K+2\ii K')\Big) \frac{d\mu}{2\pi \ii} \nonumber\\
    &=\ointctrclockwise_{\lambda_0}  \left( G_{R}(\mu)- \mathbb{1}\right) f(\mu+2K+2\ii K') \frac{d\mu}{2\pi \ii}.\label{Id}
\end{align}
In terms of integrals above, \eqref{eq:r01} reads 
\begin{align}\label{R0abcd}
    R(0,x,t)\sqrt{c}&= 4\mathbb{1} + I_5 +  I_6 + I_7 + I_8  + \mathcal{O}(t^{-\frac{2}{3}}).
\end{align}
We now compute each of the terms individually. Let us begin with
\begin{align}
    I_5 &\mathop{=}^{\eqref{asymp:GRint}} \frac{1}{\gamma t^{1/2}} \left(\begin{array}{cc}
      0   &  -a  \widetilde{\gamma}\ee^{2\ii \sfc_0}\\
      b   \widetilde{\gamma}^{-1}\ee^{-2\ii \sfc_0}  & 0
    \end{array} \right) f(\lambda_0)+\mathcal{O}(t^{-1}).\label{exp:Ia}
\end{align}
With similar computation, we have
\begin{align}
    I_6 &=\frac{1}{\gamma t^{1/2}}  \left(\begin{array}{cc}
      0   &  a  \widetilde{\gamma}\ee^{2\ii \sfc_0}\\
      -b \widetilde{\gamma}^{-1} \ee^{-2\ii \sfc_0}  & 0
    \end{array} \right) f(\lambda_0+2K)+\mathcal{O}(t^{-1}),\label{exp:Ib}\\
    I_7 &=\frac{1}{\gamma t^{1/2}}   \left(\begin{array}{cc}
      0   &  b  \widetilde{\gamma}^{-1}\ee^{-2\ii \sfc_0} \\
     -a \widetilde{\gamma}\ee^{2\ii \sfc_0}  & 0
    \end{array} \right) f(\lambda_0+2\ii K')+\mathcal{O}(t^{-1}), \label{exp:Ic}\\
    I_8&= \frac{1}{\gamma t^{1/2}} \left(\begin{array}{cc}
      0   &  -b \widetilde{\gamma}^{-1}\ee^{-2\ii \sfc_0}\\
     a  \widetilde{\gamma}\ee^{2\ii \sfc_0}  & 0
    \end{array} \right)f(\lambda_0+2K+2\ii K')+\mathcal{O}(t^{-1}).\label{exp:Id}
\end{align}
Substituting the above expressions \eqref{exp:Ia}-\eqref{exp:Id} in \eqref{R0abcd} gives us
\begin{align}
     R(0,x,t) = \frac{4}{\sqrt{c}} \,\left(\begin{array}{cc}
        1  &  \frac{A}{4t^{1/2}{\gamma}} \\
        \frac{B}{4t^{1/2}{\gamma}}  & 1
     \end{array} \right)+\mathcal{O}(t^{-\frac{2}{3}}),\label{exp:R0}
\end{align}
where 
\begin{align}
   A&:= -a\ee^{2\ii \sfc_0}  \widetilde{\gamma} \left(f(\lambda_0) -f(\lambda_0+2K) \right) + b \ee^{-2\ii \sfc_0} \widetilde{\gamma}^{-1} \left(f(\lambda_0+2\ii K') - f(\lambda_0+2K+2\ii K') \right), \label{exp:A}\\
    B&:=  b \ee^{-2\ii \sfc_0} \widetilde{\gamma}^{-1} \left(f(\lambda_0) - f(\lambda_0+2K) \right) - a \ee^{2\ii \sfc_0} \widetilde{\gamma} \left(f(\lambda_0+2\ii K') - f(\lambda_0+2K+2\ii K') \right). \label{exp:B}
\end{align}
With \eqref{eq:Psi1R0}, and \eqref{sol_LL_RHP_init_2} we get 
\begin{align}
    R(0,x,t) \sigma_3 R(0,x,t)^{-1} = \mathbf{L}(x,t) = \left( \begin{array}{cc}
       L_3(x,t)  & L_1(x,t)-\ii L_2(x,t) \\
       L_1(x,t)+\ii L_2(x,t)  & -L_3(x,t)
    \end{array}\right).
\end{align}
Substituting the explicit value of $R(0,x,t)$ given by \eqref{exp:R0} in the equation above implies that the solution of the LL equation for $t\to \infty$ reads
\begin{align}
    \mathbf{L} &\mathop{=}^{\eqref{sol_LL_RHP_init_2}}  \left(\begin{array}{cc}
        1+ \mathcal{O}(t^{-\frac{2}{3}}) & -\frac{A}{2\gamma \sqrt{t}} + \mathcal{O}(t^{-\frac{2}{3}}) \\
       \frac{B}{2\gamma \sqrt{t}} + \mathcal{O}(t^{-\frac{2}{3}})  &  -1+ \mathcal{O}(t^{-\frac{2}{3}})
    \end{array} \right).
\end{align}
Element-wise, we have
\begin{align}\label{eq:LsAB}
    L_1 = \frac{B-A}{4\gamma \sqrt{t}}+ \mathcal{O}(t^{-\frac{2}{3}}), &&
    L_2 = -\frac{\ii(A+B)}{4\gamma\sqrt{t}}+ \mathcal{O}(t^{-\frac{2}{3}}),
    \end{align}
    \begin{align}\label{eq:LsAB1}
    L_3 =1- \frac{1}{2} \left(L_1^2 + L_2^2 \right)+\mathcal{O}(t^{-\frac{7}{6}}). 
\end{align}
As a final step, we simplify the terms above
\begin{align}
    A+B &= \left( f(\lambda_0) - f(\lambda_0 +2K) + f(\lambda_0 +2\ii K')\right.\\&\left.-f(\lambda_0+2K+2\ii K') \right) \left( b \widetilde{\gamma}^{-1}\ee^{-2\ii \sfc_0} - a \widetilde{\gamma}\ee^{2\ii \sfc_0}\right), \label{def:A+B}\\
    B-A &= \left( f(\lambda_0) - f(\lambda_0 +2K) - f(\lambda_0 +2\ii K')\right.\\ &\left.+f(\lambda_0+2K+2\ii K') \right) \left( b\widetilde{\gamma}^{-1}\ee^{-2\ii \sfc_0} + a  \widetilde{\gamma}\ee^{2\ii \sfc_0}\right).\label{def:B-A}
\end{align}
Recall the definitions of the parameters
\begin{align}\label{littlegammaanditstilde}
    \widetilde{\gamma} = \left(\frac{\gamma^2 t}{\beta_0^2} \right)^{-\ii\nu}\ee^{2\pi\nu}, && \gamma = \ee^{\ii \pi/4}\sqrt{2\varphi_0} ,
\end{align}
\begin{align*}
a=\frac{\ii\sqrt{2\pi}\ee^{-2\pi \nu}}{{\sf r_0}\Gamma(-\ii\nu)}\ee^{-2\ii tp(\lambda_0,\varkappa)},&& b=\frac{\sqrt{2\pi}\ee^{3\pi \nu}}{\overline{{\sf r}}_0\Gamma(\ii\nu)}\ee^{2\ii tp(\lambda_0,\varkappa)},&&{\sf r}(\lambda_0)=\sfr_0 &&ab=\ii\nu,
\end{align*}
\begin{align*}
    \sfc_0=\frac{1}{2\pi }\int_{0}^{\lambda_0} d\left(\log \left( 1+ |{\sf r_0}(\eta)|^2 \right) \right)  \log\beta(\eta- \lambda_0).
\end{align*}
Using the identities
\begin{align}\label{reflectionGammar}
    |\Gamma(\ii\nu)|^2 = \Gamma(\ii\nu) \Gamma(-\ii\nu) = \frac{\pi}{\nu \sinh(\pi \nu)}, && \nu = \frac{1}{2\pi}\log(1+|{\sfr_0}|^2) \Rightarrow |{\sf r}_0|^2= \ee^{ \pi \nu} 2  \sinh (\pi \nu),
\end{align}
we obtain the relations
\begin{align}\label{Gammaridentities}
   \Gamma(\ii\nu) \overline{{\sfr}}_0 = \frac{2\pi \ee^{\pi\nu}}{{\sfr_0}\Gamma(-\ii\nu) \nu}, && |\Gamma(\ii\nu)| |{\sf r}_0|\sqrt{\frac{\nu}{2\pi}} =  \ee^{\pi \nu/2} .
\end{align}
With all the expressions above, we simplify the following term appearing in $L_1$:
\begin{align}
   & b \widetilde{\gamma}^{-1}\ee^{-2\ii \sfc_0}+ a \widetilde{\gamma} \ee^{2\ii \sfc_0}=\frac{\sqrt{2\pi}\ee^{\pi  \nu} \ee^{2\ii tp(\lambda_0,\varkappa)-2\ii \sfc_0}}{\overline{{\sfr}}_0\, \Gamma(\ii\nu)}  \left(\frac{\gamma^2 t }{\beta_0^2} \right)^{\ii\nu}  + \frac{\ii \sqrt{2\pi}\ee^{-2\ii tp(\lambda_0,\varkappa)+2\ii \sfc_0}}{{\sf r}_0 \,\Gamma(-\ii\nu)}  \left(\frac{\gamma^2 t }{\beta_0^2} \right)^{-\ii\nu} \nonumber \\
   &\mathop{=}^{\eqref{littlegammaanditstilde}}\left(\frac{\sqrt{2\pi}\ee^{\pi  \nu/2}\ee^{\frac{\ii\pi}{4}} }{\overline{{\sf r}}_0\, \Gamma(\ii\nu)}\right)\ee^{-\frac{\ii\pi}{4}}  {\ee^{2\ii tp(\lambda_0,\varkappa)-2\ii \sfc_0} \left(\frac{2\varphi_0 t }{\beta_0^2} \right)^{\ii\nu}} \\
   &\hspace{6cm}+ \left(\frac{\ee^{\frac{\ii\pi}{4}}\sqrt{2\pi}\ee^{\pi  \nu/2} }{{\sf r}_0\, \Gamma(-\ii\nu)}\right) \ee^{\frac{\ii\pi}{4}} {\ee^{-2\ii tp(\lambda_0,\varkappa)+2\ii \sfc_0} \left(\frac{2\varphi_0 t }{\beta_0^2} \right)^{-\ii\nu}}  \nonumber \\
    &\mathop{=}^{\eqref{Gammaridentities}}\ee^{\frac{\ii\pi}{4}}\nu^{1/2}\Big[{ \ee^{-\ii\pi/4+\ii \arg(\Gamma(\ii\nu))+ \ii \arg(\sfr_0)}\ee^{2\ii tp(\lambda_0,\varkappa)-2\ii \sfc_0} \left(\frac{2\varphi_0 t }{\beta_0^2} \right)^{\ii\nu}} \nonumber\\
    &\hspace{5cm}+ { }{\ee^{\ii\pi/4-\ii \arg(\Gamma(\ii\nu))- \ii \arg(\sfr_0)}{\ee^{-2\ii tp(\lambda_0,\varkappa)+2\ii \sfc_0} \left(\frac{2\varphi_0 t }{\beta_0^2} \right)^{-\ii\nu}}} \Big] \nonumber\\
    &=2\ee^{\frac{\ii\pi}{4}}\nu^{1/2} \cos\theta(x,t),\label{eq:cos}
\end{align}
where
\begin{align}
    \theta(x,t) =  2tp(\lambda_0,\varkappa)+ \nu \log t-\frac{\pi}{4}- \arg\Gamma(\ii\nu) + \arg \sfr_0  - 2\sfc_0 + \nu \log\left(\frac{2\varphi_0 }{\beta_0^2} \right).
\end{align}
A similar computation gives the following expression for the term appearing in $L_2$ (see \eqref{eq:LsAB}, \eqref{eq:LsAB1}):
\begin{align}\label{eq:sin}
     b\ee^{-2\ii\sfc_0} \widetilde{\gamma}^{-1}- a\ee^{2\ii\sfc_0} \widetilde{\gamma} =2\ee^{\frac{3\pi\ii}{4}} \nu^{1/2} \sin\theta(x,t).
\end{align}
With \eqref{eq:cos} and \eqref{eq:sin}, $L_1$, $L_2$ in \eqref{eq:LsAB}, \eqref{eq:LsAB1} simplify as
\begin{align}
    L_1 &= \frac{1}{2}\left(\frac{\nu}{2 \varphi_0t} \right)^{1/2} {\cos\theta(x,t)}  \Big( f(\lambda_0) - f(\lambda_0 +2K) - f(\lambda_0 +2\ii K')\\
    & \hspace{5cm}+f(\lambda_0+2K+2\ii K') \Big) + \mathcal{O}(t^{-\frac{2}{3}}),\label{simpl:L1}
\end{align}
\begin{align}
      L_2 &= \frac{1}{2}\left(\frac{\nu}{2 \varphi_0t} \right)^{1/2} {\sin\theta(x,t)}{ } \Big( f(\lambda_0) - f(\lambda_0 +2K) + f(\lambda_0 +2\ii K')\\
      &\hspace{5cm}-f(\lambda_0+2K+2\ii K') \Big) + \mathcal{O}(t^{-\frac{2}{3}}).\label{simpl:L2}
\end{align}
The above expressions can be further simplified as follows. Substituting the Cauchy kernel \eqref{def:Cauchy} in the expression for $f(\lambda_0)$ in \eqref{def:fmu} gives
\begin{align*}
     f(\lambda_0)
     &= \zeta(\lambda_0) - \zeta(\lambda_0- \ii K') +\zeta(- K- \ii K')+ \zeta(K)\\
     &+ \zeta(\lambda_0) - \zeta(\lambda_0+2K- \ii K') +\zeta(2K- K- \ii K')+ \zeta(K)\\
     &+ \zeta(\lambda_0) - \zeta(\lambda_0+2\ii K'- \ii K') +\zeta(2\ii K'- K- \ii K')+ \zeta(K)\\
     &+ \zeta(\lambda_0) - \zeta(\lambda_0+2K+2\ii K'- \ii K') +\zeta(2K+2\ii K'- K- \ii K')+ \zeta(K).
\end{align*}
The identity below then follows
\begin{align*}
    f(\lambda_0)-f(\lambda_0+2K) &= 4\left(\zeta(\lambda_0)-\zeta(\lambda_0+2K) \right)-\left(\zeta(\lambda_0-\ii K')-\zeta(\lambda_0+2K-\ii K') \right)\\
    &-\left(\zeta(\lambda_0+2K-\ii K')-\zeta(\lambda_0+4K-\ii K') \right)\\
    &-\left(\zeta(\lambda_0+2\ii K'-\ii K')-\zeta(\lambda_0+2K+2\ii K'-\ii K') \right)\\
    &-\left(\zeta(\lambda_0+2K+2\ii K'-\ii K')-\zeta(\lambda_0+4K+2\ii K'-\ii K') \right) \\
&\mathop{=}^{\eqref{id:zeta2K}}4\left(\frac{w_1(\lambda_0)+w_2(\lambda_0)}{2\rho} \right).
\end{align*}
With the above expressions and the symmetry relations of $w_1$, $w_2$ in \eqref{symm:w1}, \eqref{symm:w2} respectively, the following terms in $L_1$, $L_2$ given by \eqref{simpl:L1}, \eqref{simpl:L2} further simplify
\begin{align}
f(\lambda_0) - f(\lambda_0+2K)+ f(\lambda_0+2\ii K')-f(\lambda_0+2K+2\ii K') 
    &= \frac{4w_1(\lambda_0)}{\rho},\label{simpl:fl1}\\
     f(\lambda_0) - f(\lambda_0+2K)- f(\lambda_0+2\ii K')+f(\lambda_0+2K+2\ii K')  &=\frac{4w_2(\lambda_0)}{\rho}.\label{simpl:fl2}
\end{align}
With the above equations, \eqref{simpl:L1}, \eqref{simpl:L2} are
\begin{align}
    L_1 
    &= \frac{1}{\rho}\left(\frac{2\nu}{t\varphi_0} \right)^{1/2} { w_2(\lambda_0) \cos\theta(x,t)}{}  + \mathcal{O}(t^{-\frac{2}{3}}),
\end{align}
and
\begin{align}
      L_2 
      &=\frac{1}{\rho}\left(\frac{2\nu}{t\varphi_0} \right)^{1/2} { w_1(\lambda_0) \sin\theta(x,t)}{} + \mathcal{O}(t^{-\frac{2}{3}}).
\end{align}

\appendix

\section{Jacobi elliptic functions}\label{sec:jacobi-peoperties}

We write the properties of functions $w_1(\lambda), w_2(\lambda), w_3(\lambda)$, see \eqref{def:w1w2w3}. 
We start with periodicity properties, see \cite[\href{https://dlmf.nist.gov/22.4.T3}{Table 22.4.3}]{DLMF} 
\begin{align}
w_1(\lambda + 2 K ) = -w_1 (\lambda), && 
w_1(\lambda + 2\ii K' ) = w_1 (\lambda), && 
w_1(\lambda + 2K + 2 \ii K' ) =- w_1 (\lambda),\\\label{symm:w1}\\
w_2(\lambda + 2 K ) = -w_2 (\lambda), && 
w_2(\lambda + 2 \ii K' ) = -w_2 (\lambda), && 
w_2(\lambda + 2K + 2 \ii K' ) = w_2 (\lambda),\\\label{symm:w2}\\
w_3(\lambda + 2 K ) = w_3 (\lambda), && 
w_3(\lambda + 2\ii K' ) = -w_3 (\lambda), && 
w_3(\lambda + 2K + 2 \ii K' ) =- w_3 (\lambda).\\ \label{symm:w3}
\end{align}
The poles and zeros are detailed in the Figure \ref{fig:foobar} below, see \cite[\href{http://dlmf.nist.gov/22.4.T1}{Table 22.4.1}]{DLMF}, \cite[\href{http://dlmf.nist.gov/22.4.T2}{Table 22.4.2}]{DLMF}. The residues at $\lambda=0$ are equal to 1, see \cite[\href{http://dlmf.nist.gov/22.10.i}{ \S 22.10(i)}]{DLMF}. 
\begin{figure}[H]
    \centering
    \subfigure[$w_1$]{\includegraphics[width=0.32\textwidth]{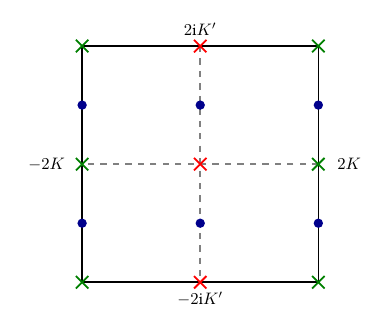}} 
    \subfigure[$w_2$]{\includegraphics[width=0.32\textwidth]{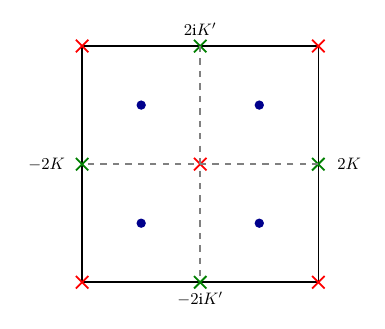}} 
    \subfigure[$w_3$]{\includegraphics[width=0.32\textwidth]{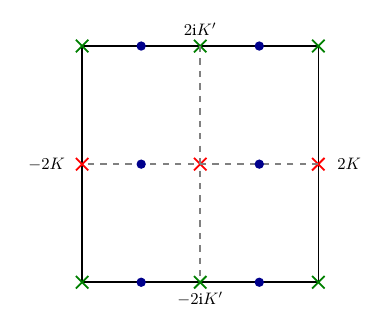}}
    \caption{{ Zeros} and { poles} of $w_1(\lambda)$, $w_2(\lambda)$, $w_3(\lambda)$ in the fundamental domain of the torus. Zeroes are denoted with color {\color{MyBlue}blue}, poles with residue $\rho$ are denoted with color {\color{red}red}, poles with residue $-\rho$ are denoted with color {\color{darkgreen}green}.  }
    \label{fig:foobar}
\end{figure}
The functions $w_1(\lambda), w_2(\lambda), w_3(\lambda)$ satisfy the identities with respect to complex conjugation and are odd functions, which can be derived from the Fourier series representations, see \cite[\href{http://dlmf.nist.gov/22.11}{\S 22.11}]{DLMF}
\begin{align}
    \label{conjugation-w}\overline{w_1(\lambda)}=w_1(\overline{\lambda}), &&
    \overline{w_2(\lambda)}=w_2(\overline{\lambda}),  && 
     \overline{w_3(\lambda)}=w_3(\overline{\lambda}).  && \\
     \label{odd-w}
    w_1(-\lambda) = -w_1(\lambda), && w_2(-\lambda)= -w_2(\lambda), && w_3(-\lambda) =-w_3(\lambda).
\end{align}
The derivatives are given below, see \cite[\href{http://dlmf.nist.gov/22.13.T1}{Table 22.13.1}]{DLMF}
\begin{align}
    \dfrac{\mathrm{d}w_1}{\mathrm{d}\lambda} =-\frac{w_2 w_3}{\rho}, &&  \dfrac{\mathrm{d}w_2}{\mathrm{d}\lambda} = - \frac{w_1 w_3}{\rho}, &&  \dfrac{\mathrm{d}w_3}{\mathrm{d}\lambda} = - \frac{w_1 w_2}{\rho}. \label{eq:w-derivatives}
\end{align}
The addition theorem \cite[\href{http://dlmf.nist.gov/22.8.i}{\S 22.8(i)}]{DLMF}  and Jacobi’s imaginary transformation formulas
\cite[\href{http://dlmf.nist.gov/22.6.T1}{Table 22.6.1}]{DLMF} imply
\begin{align}
{w_1(\lambda, k)}=\frac{w_1(\Re\lambda,k)\,w_2(\Im\lambda,k')\,w_1(\Im\lambda,k')}{w_3^2(\Re\lambda,k)+w_1^2(\Im\lambda,k')}-\ii\,\frac{w_3(\Im\lambda,k')\,w_3(\Re\lambda,k)\,w_2(\Re\lambda,k)}{w_3^2(\Re\lambda,k)+w_1^2(\Im\lambda,k')}\label{eq:w1-complexplane}\\
   {w_2(\lambda, k)}=\frac{w_2(\Re\lambda,k)\,w_1(\Im\lambda,k')\,w_3(\Im\lambda,k')}{w_3^2(\Re\lambda,k)+w_1^2(\Im\lambda,k')}-\ii\,\frac{w_2(\Im\lambda,k')\,w_3(\Re\lambda,k)\,w_1(\Re\lambda,k)}{w_3^2(\Re\lambda,k)+w_1^2(\Im\lambda,k')}\label{w2-complexplane}\\
    {w_3(\lambda, k)}=\frac{w_3(\Re\lambda,k)\,w_2(\Im\lambda,k')\,w_3(\Im\lambda,k')}{w_3^2(\Re\lambda,k)+w_1^2(\Im\lambda,k')}-\ii\,\frac{w_1(\Im\lambda,k')\,w_2(\Re\lambda,k)\,w_1(\Re\lambda,k)}{w_3^2(\Re\lambda,k)+w_1^2(\Im\lambda,k')}\\\label{w3-complexplane}.
\end{align}
Special values are given below, see \cite[\href{http://dlmf.nist.gov/22.5.T1}{Table  22.5.1}]{DLMF}
\begin{align}\label{eq:special values}
    w_1(K)=1,&&w_2(K)=k'.
\end{align}

\section{Weierstrass zeta function} \label{sec:Weierstrass_zeta}
Weierstrass $\zeta$-function can be written in terms of Weierstrass $\sigma$-function, see \cite[\href{http://dlmf.nist.gov/23.2.E8 }{ 23.2.8}]{DLMF}
\begin{align}\label{eq:zeta_sigma}
\zeta(\lambda)=\partial_\lambda(\log(\sigma(\lambda))).
\end{align}
The Weierstrass $\zeta$-function has the following periodicity properties, see \cite[\href{http://dlmf.nist.gov/23.2.iii}{\S 23.2(iii)}]{DLMF}
\begin{align}\label{per:zeta1}
    \zeta(\lambda+4K)-\zeta(\lambda) = 2 \zeta(2K), \\\label{per:zeta2} \zeta(\lambda+4K+4\ii K')-\zeta(\lambda) = 2\zeta(2K+2\ii K'),\\ \label{per:zeta3}\zeta(\lambda+4\ii K')-\zeta(\lambda) = 2\zeta(2\ii K').
\end{align}
Combining \cite[\href{http://dlmf.nist.gov/23.10.E2 }{23.10.2}]{DLMF},  \cite[\href{http://dlmf.nist.gov/23.2.E7}{23.2.7}]{DLMF}, and  \cite[\href{http://dlmf.nist.gov/23.2.E10}{ 23.2.10}]{DLMF} we can see the identity
\begin{align}\label{eq:zeta-identity}
    \zeta(2K)+\zeta(2\ii K')= \zeta(2K+2\ii K')
\end{align}
Following \cite[Chapter 8.11]{lawden2013elliptic} and \cite[\href{https://dlmf.nist.gov/23.6.ii}{\S 23.6(ii)}]{DLMF} we can derive the following identities
\begin{align}
    w_1(\lambda)&=\rho\Big(  \zeta(\lambda) - \zeta(\lambda+2K) - \zeta(\lambda+2K+2\ii K')+ \zeta(\lambda+2\ii K')\\
    &\hspace*{6cm}+ \zeta(2K) + \zeta(2K+2\ii K') -\zeta(2\ii K')\Big)\label{eq:w1-zeta}\\
        w_2(\lambda)&=\rho\Big(  \zeta(\lambda) - \zeta(\lambda+2K) +\zeta(\lambda+2K+2\ii K')- \zeta(\lambda+2\ii K')\\
        &\hspace*{6cm}+ \zeta(2K) - \zeta(2K+2\ii K') +\zeta(2\ii K')\Big)\label{eq:w2-zeta}\\   
        w_3(\lambda)&=\rho\Big(  \zeta(\lambda) + \zeta(\lambda+2K) - \zeta(\lambda+2K+2\ii K')- \zeta(\lambda+2\ii K')\\
        &\hspace*{6cm}- \zeta(2K) + \zeta(2K+2\ii K') +\zeta(2\ii K')\Big)\label{eq:w3-zeta}
\end{align}
Actually,  using \eqref{per:zeta1}--\eqref{per:zeta3}, \eqref{eq:zeta-identity}, Figure \ref{fig:foobar}, and the fact that $\zeta(\lambda)$ is odd function we can notice that left and right hand side of \eqref{eq:w1-zeta}--\eqref{eq:w3-zeta} have the same zeroes. Since their poles coincide too, together with residues, we get the identities. They can be rewritten as
\begin{align}
    \zeta(\lambda) - \zeta(\lambda+2K) + \zeta(2K) & = \frac{w_1(\lambda)+w_2(\lambda)}{2\rho},\label{id:zeta2K}\\
     \zeta(\lambda) - \zeta(\lambda+2\ii K') + \zeta(2\ii K') &= \frac{w_2(\lambda)+w_3(\lambda)}{2\rho},\label{id:zeta2iK'}\\
      \zeta(\lambda) - \zeta(\lambda+2K+2\ii K') + \zeta(2K+2\ii K') &= \frac{w_1(\lambda)+w_3(\lambda)}{2\rho}\label{id:zeta2Kp2iK'}.
\end{align}
\section{Parabolic cylinder parametrix}\label{sec:parabolic}
Following the ideas of \cite{its1981} define the matrix 
\begin{align}
&\mathbf{D}_{a,b}(\xi):=\begin{pmatrix}
     1&0\\-\frac{\xi}{2a}&\frac{1}{a} 
    \end{pmatrix}\\&\times \left\{\begin{array}{l}
     \begin{pmatrix}
     D_{ab}(\ee^{-\frac{\ii\pi}{2}}\xi)&D_{-ab-1}(\xi)\\ D'_{ab}(\ee^{-\frac{\ii\pi}{2}}\xi)&D'_{-ab-1}(\xi)
    \end{pmatrix}\begin{pmatrix}
     \ee^{\frac{\ii\pi}{2} ab }&0\\0&-a\end{pmatrix},\quad \frac{\pi}{4}<\arg (\xi)<\frac{\pi}{2},\\[0.35cm]
     \begin{pmatrix}
     D_{ab}(\ee^{-\frac{\ii\pi}{2}}\xi)&D_{-ab-1}(\ee^{-{\ii\pi }{}}\xi)\\ D'_{ab}(\ee^{-\frac{\ii\pi}{2}}\xi)&D'_{-ab-1}(\ee^{-{\ii\pi }{}}\xi)
    \end{pmatrix}\begin{pmatrix}
     \ee^{\frac{\ii\pi}{2} ab }&0\\0&a \ee^{-{\ii\pi}{} ab } \end{pmatrix},\quad \frac{\pi}{2}<\arg (\xi)<\pi,\\[0.35cm]
    \begin{pmatrix}
     D_{ab}(\ee^{-\frac{3\pi \ii}{2}}z)&D_{-ab-1}(\ee^{-{\ii\pi }{}}\xi)\\D'_{ab}(\ee^{-\frac{3\pi \ii}{2}}\xi)&D'_{-ab-1}(\ee^{-{\ii\pi }{}}\xi)
    \end{pmatrix} \begin{pmatrix}
     \ee^{\frac{3\pi \ii}{2} ab }&0\\0&a \ee^{-{\ii\pi}{} ab } \end{pmatrix},\quad {\pi}{}<\arg (\xi)<\frac{3\pi}{2},\\[0.35cm]
     \begin{pmatrix}
     D_{ab}(\ee^{-\frac{3\pi \ii}{2}}\xi)&D_{-ab-1}(\ee^{-{2\pi \ii}{}}\xi)\\ D'_{ab}(\ee^{-\frac{3\pi \ii}{2}}\xi)&D'_{-ab-1}(\ee^{-{2\pi \ii}{}}\xi)
    \end{pmatrix}\begin{pmatrix}
     \ee^{\frac{3\pi \ii}{2} ab }&0\\0&-a \ee^{-{2\pi \ii}{} ab } \end{pmatrix}.\quad \frac{3\pi}{2}<\arg (\xi)<2\pi,\\
     \begin{pmatrix}
     D_{ab}(\ee^{-\frac{5\pi \ii}{2}}\xi)&D_{-ab-1}(\ee^{-{2\pi \ii}{}}\xi)\\ D'_{ab}(\ee^{-\frac{5\pi \ii}{2}}\xi)& D'_{-ab-1}(\ee^{-{2\pi \ii}{}}\xi)
    \end{pmatrix}\begin{pmatrix}
     \ee^{\frac{5\pi \ii}{2} ab }&0\\0&-a \ee^{-{2\pi \ii}{} ab } \end{pmatrix},\quad 2\pi<\arg (\xi)<\frac{9\pi}{4},\\[0.35cm]
     \end{array}\right.,\label{eq:parabolic_cylinder_parametrix}
\end{align}
where $D_\nu(\xi)$ is the parabolic cylinder function.
The matrix $\mathbf{D}_{a,b}(\xi)$ solves the linear system
\begin{align}
\dfrac{d}{d\xi}\mathbf{D}_{a,b}(\xi)=\left(\frac{\xi}{2}\sigma_3+\begin{pmatrix}
0&a\\
b&0\end{pmatrix}
\right)\mathbf{D}_{a,b}(\xi),
\end{align}
and has the following behavior at infinity (see \cite[\href{http://dlmf.nist.gov/12.9.E1}{ 23.9(i)}]{DLMF})
\begin{equation}\label{asymp:Dab}
\mathbf{D}_{a,b}(\xi)=\left(\mathbb{1}+\frac{m_1}{\xi}+\frac{m_2}{\xi^2}+\mathcal{O}(\xi^{-3})\right)\xi^{ab\sigma_3}\ee^{\frac{\xi^2}{4}\sigma_3},\quad \xi\to \infty,\quad \frac{\pi}{4}<\arg (\xi)<\frac{9\pi}{4}.
\end{equation}
where
\begin{align*}
m_1=
\begin{pmatrix}
0& -a\\
b&0\end{pmatrix}, 
&&    m_2 = \left( \begin{array}{cc}
     \frac{ab(ab-1)}{2}    & 0 \\
       0  & \frac{ab (ab+1)}{2}
    \end{array} \right).
\end{align*}

Using properties \cite[\href{http://dlmf.nist.gov/12.2.E19It}{(12.2.19)}]{DLMF} of parabolic cylinder functions one can show that $\mathbf{D}_{a,b}(\xi)$ satisfies the following Riemann-Hilbert problem
\begin{rhp}\label{rhp:parabolic}
\begin{enumerate}
    \item The function $\mathbf{D}_{a,b}(\xi)$ is analytic for $\xi\in\mathbb{C}\setminus \Sigma^{(par)}$.

    \item For $\lambda \in \Sigma^{(par)}$ the boundary values satisfy the jump condition
    \begin{align}
        \mathbf{D}_{a,b,+}(\xi) = \mathbf{D}_{a,b,-}(\xi) G^{(par)}(\xi),
    \end{align}
    with the jump $G^{(par)}(\xi)$ piecewise defined on each ray of the contour $\Sigma^{(par)}$ as shown on Figure \ref{fig:contour_parabolic}.
    \item Function $\mathbf{D}_{a,b,+}(\xi)$ satisfies the asymptotic condition at infinity
    \begin{align}
        \mathbf{D}_{a,b}(\xi)= \mathbf{D}_{a,b}(\xi)=\left(\mathbb{1}+\mathcal{O}(\xi^{-1})\right)\xi^{ab\sigma_3}\ee^{\frac{\xi^2}{4}\sigma_3},\quad \xi\to\infty
    \end{align}
\end{enumerate}    
\end{rhp}
\begin{figure}[h!]
        \centering
\includegraphics[width=0.7\columnwidth]{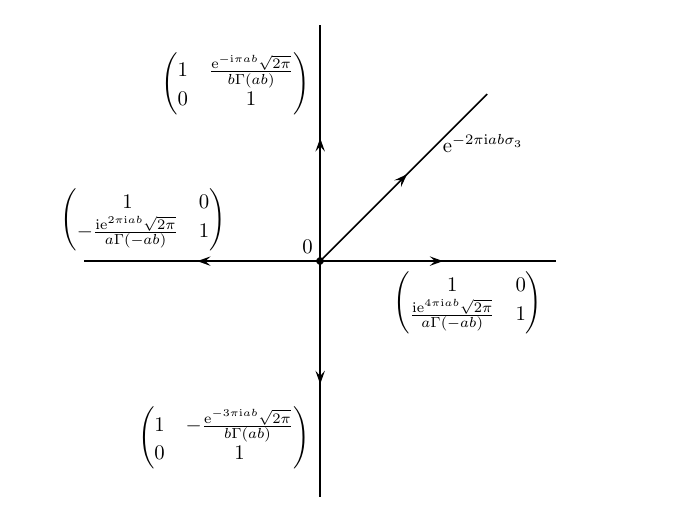}
    \caption{Jump contour $\Sigma^{(par)}$ for RHP \ref{rhp:parabolic}}
\label{fig:contour_parabolic}
    \end{figure}

\section{Proofs of Propositions \ref{Prop:Rodin}, \ref{prop:int_ab}, and \ref{prop:a_b_smooth}}
\label{prop_proof}

\begin{proof}[Proof of Proposition \ref{Prop:Rodin}]
    For $\lambda\neq 0, 2K, 2\ii K', 2K+2\ii K' $ we can obtain the result just by iterating the integral equation. 
We will focus our analysis on function $\upsilon_{-}^{(1)}(\lambda,x)$ first. We write the Neumann series expansion
\begin{align}\label{eq:neumann_series}
\upsilon_{-}^{(1)}(\lambda,x)=\sum_{n=0}^{\infty}\left(\upsilon_{-}^{(1)}\right)_n(\lambda,x)
\end{align}
where $\left(\upsilon_{-}^{(1)}\right)_n(\lambda,x)$ are defined recursively
\begin{align}
 \left(\upsilon_{-}^{(1)}\right)_{n+1}(\lambda,x)=  \int_{-\infty}^{x} \ee^{\ii (\mathbb{1}-\sigma_3) w_3(\lambda) (x-\tau)}\left(U(\lambda,\tau) +\ii w_3(\lambda) \sigma_3 \right)  \left(\upsilon_{-}^{(1)}\right)_{n}(\lambda,\tau)d\tau,\\
   \left(\upsilon_{-}^{(1)}\right)_{0}(\lambda,x)=    \left( \begin{array}{c}
         1  \\
         0 
    \end{array}\right).
\end{align}
Observe that for $\lambda\in \Omega_-$ we have
\begin{align}
\left|  \left(\upsilon_{-}^{(1)}\right)_{0}(\lambda,x)\right|\leq 1,\quad \left|  \left(\upsilon_{-}^{(1)}\right)_{1}(\lambda,x)\right|\leq \intop_{-\infty}^{x}
\left|U(\lambda,\tau) +\ii w_3(\lambda) \sigma_3\right|d\tau=\sigma_-(x),\\
\left|  \left(\upsilon_{-}^{(1)}\right)_{2}(\lambda,x)\right|\leq \intop_{-\infty}^{x}
\left|U(\lambda,\tau) +\ii w_3(\lambda) \sigma_3\right|\sigma_-(\tau)d\tau=\frac{\sigma_-^2(x)}{2}.
\end{align}
Here we denoted the Hilbert-Schmidt matrix norm by $|M|$. Then we have the following estimates proven by induction
\begin{align}
\left| \left(\upsilon_{-}^{(1)}\right)_{n}(\lambda,x)\right|\leq
\frac{\sigma_-^n(x)}{n!}\leq \frac{\sigma_-^n(\infty)}{n!}.
\end{align}
That shows that the series \eqref{eq:neumann_series} converges uniformly, function $\upsilon_{-}^{(1)}(\lambda,x)$ satisfies \eqref{eq:int_ypm1} and is analytic in $\Omega_-$ for $\lambda\neq 0, 2K, 2\ii K', 2K+2\ii K' $. 

Next, we construct $\widehat{\upsilon}_{+}^{(2)}(\lambda,x)$ using the Neumann series 
\begin{align}\label{eq:neumann_series2}
\upsilon_{+}^{(2)}(\lambda,x)=\sum_{n=0}^{\infty}\left(\upsilon_{+}^{(2)}\right)_n(\lambda,x)
\end{align}
with iterations given by
\begin{align}
 \left(\upsilon_{+}^{(2)}\right)_{n+1}(\lambda,x)=-  \int^{\infty}_{x} \ee^{-\ii (\mathbb{1}+\sigma_3) w_3(\lambda) (x-\tau)}\left(U(\lambda,\tau) +\ii w_3(\lambda) \sigma_3\right)  \left(\upsilon_{-}^{(2)}\right)_{n}(\lambda,\tau)d\tau,\\
   \left(\upsilon_{+}^{(2)}\right)_{0}(\lambda,x)=    \left( \begin{array}{c}
         0  \\
         1 
    \end{array}\right).
\end{align}
We can show similarly to above that the series \eqref{eq:neumann_series2} converges uniformly, function $\upsilon_{+}^{(2)}(\lambda,x)$ satisfies \eqref{eq:int_ypm2} and is analytic in $\Omega_-$ for $\lambda\neq 0, 2K, 2\ii K', 2K+2\ii K' $. This procedure can be repeated for $\upsilon_{+}^{(1)}(\lambda,x)$ and $\upsilon_{-}^{(2)}(\lambda,x)$.

Next, let's show that these functions are bounded at $\lambda=0$. We follow the idea of \cite{rodin1984} to approximate solution at that point with the solution of Lax pair of NLS equation. The fact that isotropic degeneration of Landau-Lifshitz equation is equivalent to NLS equation was shown in \cite{zakharov1979equivalence}. We introduce functions $\accentset{\circ}{\upsilon}_-^{(1)}(\lambda,x)$, $\accentset{\circ}{\upsilon}_+^{(2)}(\lambda,x)$ as the solutions of the integral equations
\begin{align}
\accentset{\circ}{\upsilon}_{-}^{(1)}(\lambda,x) &= \left( \begin{array}{c}
         1  \\
         0 
    \end{array}\right) +  \int_{-\infty}^{x} \ee^{{\ii (\mathbb{1}-\sigma_3)  (x-\tau)\rho}/{\lambda}}\left(\accentset{\circ}{U}(\lambda,\tau) +\frac{\ii  \rho\sigma_3 }{\lambda} \right) \accentset{\circ}{\upsilon}_{-}^{(1)}(\lambda,\tau) d\tau,
    \\
    \accentset{\circ}{\upsilon}_{+}^{(2)}(\lambda,x) &= \left( \begin{array}{c}
         0  \\
         1 
    \end{array}\right) -  \int^{\infty}_{x} \ee^{-\ii (\mathbb{1}+\sigma_3)  (x-\tau)\rho/\lambda}\left(\accentset{\circ}{U}(\lambda,\tau) +\frac{\ii  \rho\sigma_3}{\lambda} \right) \accentset{\circ}{\upsilon}_{+}^{(2)}(\lambda,\tau) d\tau,
\end{align}
where $\accentset{\circ}{U}(\lambda,x) = -\ii\sum_{j=1}^{3} \dfrac{L_{j}(x)\rho\sigma_{j}}{\lambda}$. We form matrix $\accentset{\circ}{v}_{-}(\lambda,\tau)=(\accentset{\circ}{\upsilon}_{-}^{(1)}(\lambda,\tau),\accentset{\circ}{\upsilon}_{+}^{(2)}(\lambda,\tau))$. It satisfies the differential equation
\begin{align}
   \frac{\partial \accentset{\circ}{v}_{-}^{} (\lambda, x)}{\partial x} = \accentset{\circ}{U}(\lambda,x)\accentset{\circ}{v}_{-}^{} (\lambda, x)+\frac{\ii\rho}{\lambda}\accentset{\circ}{v}_{-}^{} (\lambda, x)\sigma_3. 
\end{align}
In other words, we replaced the coefficient matrix $U(\lambda,x)$ with the leading term of its asymptotics at $\lambda=0$. This is part of the Lax pair for the isotropic Landau-Lifshitz equation. To deal with function $\accentset{\circ}{v}_{-}^{} (\lambda, x)$ we consider $\accentset{\circ}{f}_{-}^{} (\lambda, x)=\accentset{\circ}{v}_{-}^{} (\lambda, x)\ee^{-\frac{\ii x\rho}{\lambda}\sigma_3}$. It satisfies the differential equation
\begin{align}
    \frac{\partial \accentset{\circ}{f}_{-}^{(1)} (\lambda, x)}{\partial x} = \accentset{\circ}{U}(\lambda,x)\accentset{\circ}{f}_{-}^{(1)} (\lambda, x) 
\end{align}
We can notice $\accentset{\circ}{U}(\lambda,x)$ has eigenvalues $\pm\frac{\ii\rho}{\lambda}$. Let's diagonalize it. Its matrix of eigenvectors has two forms
\begin{align}T_1(x)=\begin{pmatrix}
\frac{1}{2}\left(L_3(x)+1\right)&-\dfrac{L_1(x)-\ii L_2(x)}{L_3(x)+1}\\\frac{1}{2}\left(L_1(x)+\ii L_2(x)\right)&1
\end{pmatrix},\quad\mbox{if } L_3(x)\neq -1,\\
T_2(x)=\begin{pmatrix}
-\dfrac{L_1(x)-\ii L_2(x)}{L_3(x)-1}&\frac{1}{2}({L_3(x)-1}{})\\1&\frac{1}{2}(L_1(x)+\ii L_2(x))
\end{pmatrix},\quad\mbox{if } L_3(x)\neq 1.
\end{align}
We can notice that $T_2(x)=T_1(x)\left(\frac{2}{L_1+\ii L_2}\right)^{\sigma_3}$. Since function $L_3(x)-1$ is smooth, there is $\varepsilon>0$ such that the  sets $\{x:|L_3(x)-1|<\varepsilon \}$ and $\{x:|L_3(x)+1|<\varepsilon\}$ are disjoint. We would like to construct the eigenvectors using the partition of unity. We take nonvanishing infinitely smooth function $c(x)$ such that 
\begin{align}
    c(x)=1,\quad \mbox{if } |L_3(x)-1|<\varepsilon,\quad \mbox{and}&& c(x)=\frac{2}{L_1(x)+\ii L_2(x)}, \quad \mbox{if } 0<|L_3(x)+1|<\varepsilon.
\end{align}
and define \begin{align}T(x)=T_1(x)c^{\sigma_3}(x)d^{\sigma_3}(x),\end{align}
where function $d(x)$ will be determined in a moment.
We can check the identity \begin{align}\left(T(x)\right)^{-1}\left(\sum_{j=1}^{3} {L_{j}(x)\sigma_{j}}\right)T(x)=\sigma_3.\end{align}
We introduce function $s(\lambda,x)=\left(T(x)\right)^{-1}\accentset{\circ}{f}_{-}^{} (\lambda, x) $. It satisfies the differential equation
\begin{align}\label{eq:almost_NLS}
   \frac{\partial s (\lambda, x)}{\partial x} = \left(-\frac{\ii \rho}{\lambda}\sigma_3-T^{-1}\dfrac{dT}{dx}\right)s (\lambda, x) 
\end{align}
It is already very similar to the Zakharov-Shabat system, which are part of the Lax pair of NLS equation. On the next step we would like to cancel the diagonal terms of $T^{-1}\dfrac{dT}{dx}$. For that we can choose function $d(x)$ to satisfy the equation
\begin{align}
\dfrac{d'(x)}{d(x)}=-\dfrac{c'(x)}{c(x)}-\frac{L_3'(x)}{2}-\dfrac{(L_1'(x)+\ii L_2'(x))(L_1(x)-\ii L_2(x))}{2(L_3(x)+1)}
\end{align}
We can notice that for $0<|L_3(x)+1|<\varepsilon$ it takes form
\begin{align}
    \dfrac{d'(x)}{d(x)}=-\frac{L_3'(x)}{2}-\dfrac{(L_1'(x)+\ii L_2'(x))(L_1(x)-\ii L_2(x))}{2(L_3(x)-1)}. 
\end{align}
Since $\frac{d'(x)}{d(x)}\to 0$ as $|x|\to \infty$  the function $d(x)$ approaches to nonvanishing constants at $\pm\infty$, one of which we can choose. 
As the result we have equation \eqref{eq:almost_NLS} becoming Zakharov-Shabat system with spectral parameter $\frac{1}{\lambda}$. More specifically the off-diagonal elements of $T^{-1}\dfrac{dT}{dx}$ which correspond to potentials are given by
\begin{align}
    T^{-1}\frac{dT}{dx}=\begin{pmatrix}
        0&-\dfrac{1}{c^2(x)d^2(x)}\left(\dfrac{L_1(x)-\ii L_2(x)}{L_3(x)+1}\right)'\\
      (L_3(x)+1)^2\dfrac{c^2(x)d^2(x)}{4}\left(\dfrac{L_1(x)+\ii L_2(x)}{L_3(x)+1}\right)' &0
    \end{pmatrix}.
\end{align}
\noindent We can notice that the entries of the matrix above belong to the Schwartz class $\mathcal{S}(\mathbb{R})$. Using standard methods of theory of NLS equation we can show that columns of $s(\lambda,x)=(s^{(1)}(\lambda,x),s^{(2)}(\lambda,x))$ admit the following integral representations, see \cite[Chapter 1, (5.10)]{faddeev1987hamiltonian}.
\begin{align}
    s^{(1)}(\lambda,x)=\begin{pmatrix}
        \ee^{-\frac{\ii\rho x}{\lambda}}\\0
    \end{pmatrix}+\intop_{-\infty}^x\ee^{-\frac{\ii\rho\tau}{\lambda}}\Gamma^{(1)}(x,\tau)d\tau,\quad \Gamma^{(1)}(x,\tau)\in L_1(-\infty,x)\\
    s^{(2)}(\lambda,x)=\begin{pmatrix}
     0  \\ \ee^{\frac{\ii \rho x}{\lambda}}
    \end{pmatrix}-\intop^{\infty}_x\ee^{\frac{\ii\rho \tau}{\lambda}}\Gamma^{(2)}(x,\tau)d\tau,\quad \Gamma^{(2)}(x,\tau)\in L_1(x,\infty)
\end{align}
We can notice that $s(\lambda,x)\ee^{\frac{\ii\rho  x}{\lambda}\sigma_3}$ is bounded for $\lambda\in \overline{\Omega}_-$. Therefore since $\accentset{\circ}{v}_{-}^{} (\lambda, x)=T(x)s^{}(\lambda,x)\ee^{\frac{\ii\rho x}{\lambda}\sigma_3}$ we see that $\accentset{\circ}{v}_{-}^{} (\lambda, x)$ is bounded near $\lambda=0$ as well.

Introduce the ratio between the matrix function $\accentset{\circ}{v}_{-}^{} (\lambda, x)$ and original vector function ${\upsilon}_{-}^{(1)} (\lambda, x)$ 
\begin{align}
{\eta}_{-}^{(1)} (\lambda, x) =\left(\accentset{\circ}{v}_{-}^{} (\lambda, x)\right)^{-1} {\upsilon}_{-}^{(1)} (\lambda, x). 
\end{align}
It satisfies the differential equation
\begin{align}
  \frac{\partial {\eta}_{-}^{(1)} (\lambda, x)}{\partial x}=\left(\accentset{\circ}{v}_{-}^{} (\lambda, x)\right)^{-1}\left(U(\lambda,x)-\accentset{\circ}{U}(\lambda,x)+\ii w_3(\lambda)-\frac{\ii\rho}{\lambda}\right)\accentset{\circ}{v}_{-}^{} (\lambda, x){\eta}_{-}^{(1)} (\lambda, x).
\end{align}
 We see that the coefficient matrix does not have singularity at $\lambda=0$ anymore. That implies that repeating the iteration procedure for the corresponding integral equation we can show that ${\eta}_{-}^{(1)} (\lambda, x)$ is bounded at zero. As result solution ${\upsilon}_{-}^{(1)} (\lambda, x)$ is bounded at zero as well, as desired. The proof for $\upsilon_{+}^{(2)}(\lambda,x)$, $\upsilon_{+}^{(1)}(\lambda,x)$, $\upsilon_{-}^{(2)}(\lambda,x)$ is similar.
\end{proof}

\begin{proof}[Proof of Proposition \ref{prop:int_ab}]\hspace*{5cm}
\begin{enumerate}
\item
Denote $\Lambda_\pm(\lambda,x)=\ee^{\ii x w_3(\lambda)  \sigma_3}F_\pm(\lambda,x)=(\phi_{\pm}^{(1)}(\lambda,x),\phi_{\pm}^{(2)}(\lambda,x))$.
In the limit $x\to \pm \infty$, we have
\begin{align}
\lim_{x\to \pm\infty}\Lambda_\pm(\lambda,x)=I,\qquad
    \lim_{x\to \mp\infty}\Lambda_\pm(\lambda,x)=S(\lambda).
\end{align}
We can then notice that
\begin{align}
 \upsilon_{\pm}^{(1)}(\lambda,x)=\ee^{\ii (\mathbb{1}-\sigma_3) x w_3(\lambda)}\phi_{\pm}^{(1)}(\lambda,x),  \\ 
     \upsilon_{\pm}^{(2)}(\lambda,x)=\ee^{-\ii (\mathbb{1}+\sigma_3) x w_3(\lambda)}\phi_{\pm}^{(2)}(\lambda,x) . 
\end{align}
The function $\phi_{-}^{(1)}(\lambda,x)$ can be expressed as the following integrals
\begin{align}
    \phi_{-}^{(1)}(\lambda,x) &= \left( \begin{array}{c}
         1  \\
         0 
\end{array}\right) +  \int_{-\infty}^{x} \ee^{-\ii (\mathbb{1}-\sigma_3) w_3(\lambda) \tau} \left(U(\lambda,\tau) +\ii w_3 (\lambda)\sigma_3 \right) \upsilon_{-}^{(1)}(\lambda,\tau) d\tau.
\end{align}
Taking $x\to+\infty$ limit of the above expressions, we get \eqref{eq:b_integral}. 
\item
Using \eqref{jump_jost} we can see that
\begin{align}\label{eq:explicit_jump}
\upsilon_{+}^{(1)}(\lambda,x)=\upsilon_{-}^{(1)}(\lambda,x){\sf a}(\lambda)+\upsilon_{-}^{(2)}(\lambda,x)\ee^{2\ii  w_3(\lambda) x}{\sf b}(\lambda)
\end{align}
We know that the Jost solutions are unimodular $\det(\Upsilon_{\pm}(\lambda,x))=1$. 
Taking the determinant of the right hand side of expressions \eqref{eq:a_det_form}, \eqref{eq:b_det_form} and using \eqref{eq:explicit_jump} we obtain the desired result.

\end{enumerate}
\end{proof}

\begin{proof}[Proof of Proposition \ref{prop:a_b_smooth}]
Equation \eqref{eq:a_det_form} and Proposition \ref{Prop:Rodin} imply that $\sfa(\lambda)$ is analytic in $\Omega_+$, while the expression \eqref{eq:b_det_form} and Proposition \ref{Prop:Rodin} say that ${\sf b}(\lambda)\in C^{\infty}(\Gamma_1\cup\Gamma_2)$. 

To obtain the behavior of ${\sf b}(\lambda)$ at zero, we look at the integral representation \eqref{eq:b_integral}. We rewrite $\upsilon_{-}^{(1)}(\lambda,\tau)$ using functions constructed in the proof of Proposition \ref{Prop:Rodin}:
\begin{align}
    \upsilon_{-}^{(1)}(\lambda,\tau)=T(\tau)s(\lambda,\tau)\ee^{\frac{\ii \rho \tau}{\lambda}\sigma_3} \eta_{-}^{(1)}(\lambda,\tau).
\end{align}
As the result we have 
\begin{align}
    {\sf b} (\lambda)=\int_{-\infty}^{ \infty} \ee^{-2\ii w_3(\lambda) \tau}  \left[\left(U(\lambda,\tau) +\ii w_3(\lambda) \sigma_3 \right)  T(\tau)s(\lambda,\tau)\ee^{\frac{\ii \rho\tau}{\lambda}\sigma_3} \eta_{-}^{(1)}(\lambda,\tau)\right]_2 d\tau.
\end{align}
 where columns of $s(\lambda,x)\ee^{\frac{\ii\rho x}{\lambda}\sigma_3}$ can be written as
\begin{align}
    s^{(1)}(\lambda,x)\ee^{\frac{\ii\rho x}{\lambda}}=\begin{pmatrix}
        1\\0
    \end{pmatrix}+\intop_{-\infty}^0\ee^{-\frac{\ii\rho y}{\lambda}}\Gamma^{(1)}(x,x+y)dy,\\
    s^{(2)}(\lambda,x)\ee^{-\frac{\ii\rho x}{\lambda}}=\begin{pmatrix}
     0  \\ 1
    \end{pmatrix}-\intop^{\infty}_0\ee^{\frac{\ii \rho y}{\lambda}}\Gamma^{(2)}(x,x+y)dy.
\end{align}
Where functions $\Gamma^{(1)}(x,\tau)$, $\Gamma^{(2)}(x,\tau)$ actually are infinitely smooth in both variables and decay faster than any power as $\tau\to\infty$.  The behavior of ${\sf b}(\lambda)$ for $\lambda\to 0$ can be obtained using integration by parts. More precisely, denote 
\begin{align}
    \mathcal{F}(\lambda,\tau)=\left[\left(U(\lambda,\tau) +\ii w_3(\lambda) \sigma_3 \right)  T(\tau)s(\lambda,\tau)\ee^{\frac{\ii \rho\tau}{\lambda}\sigma_3} \eta_{-}^{(1)}(\lambda,\tau)\right]_2
\end{align}

We can see that

\begin{align}
    {\sf b}(\lambda)=\left.\sum_{k=0}^{n-1}(-1)^k \frac{\ee^{-2\ii w_3(\lambda) \tau}}{(-2\ii w_3(\lambda))^{k+1}}\dfrac{d^k}{d\tau^k}\mathcal{F}(\lambda,\tau)\right|_{\tau=-\infty}^{\tau=\infty}\\+\frac{(-1)^n}{(-2\ii w_3(\lambda))^{n}}\int_{-\infty}^{ \infty} \ee^{-2\ii w_3(\lambda) \tau}\dfrac{d^n}{d\tau^n}\mathcal{F}(\lambda,\tau)d\tau,\quad \forall n\in \mathbb{N}.
\end{align}
Since the function $\mathcal{F}(\lambda,\tau)$ belongs to the Schwartz class $\mathcal{S}(\mathbb{R})$ in variable $\tau$, the substitution at $\pm\infty$ vanishes. In the same time derivatives with respect $\tau$ applied to $\mathcal{F}(\lambda,\tau)$ do not increase the power of singularity at $\lambda=0$. That provides as with decay of ${\sf b}(\lambda)$ at $\lambda=0$ faster than any power. 
The derivatives with respect to $\lambda$ preserve $\mathcal{F}(\lambda,\tau)$ in the Schwartz class $\mathcal{S}(\mathbb{R})$ with respect to the variable $\tau$, which allows us to repeat integration by parts argument and get decay of derivatives of ${\sf b}(\lambda)$ as well.
\end{proof}

\printbibliography

\end{document}